\documentclass[reqno]{amsart}

\usepackage{amsmath}
\usepackage{amsthm}
\usepackage{amsfonts}
\usepackage{amssymb}
\usepackage{enumerate}
\usepackage{enumitem}
\usepackage[pdftex]{graphicx}
\usepackage{epstopdf}
\usepackage{todonotes}
\usepackage{verbatim}

\newcommand{\indicator}[1]{\ensuremath{\mathbf{1}_{\{#1\}}}}
\newcommand{\oindicator}[1]{\ensuremath{\mathbf{1}_{{#1}}}}

\DeclareMathOperator{\var}{Var}
\DeclareMathOperator{\tr}{tr}

\DeclareMathOperator{\diag}{diag}

\DeclareMathOperator{\rank}{rank}

\newcommand{\Prob}{\mathbb{P}}
\newcommand{\E}{\mathbb{E}}
\newcommand{\C}{\mathbb{C}}
\renewcommand{\P}{\mathbb{P}}
\renewcommand\Re{\operatorname{Re}}
\renewcommand\Im{\operatorname{Im}}
\newcommand{\eps}{\varepsilon}

\theoremstyle{plain}
  \newtheorem{theorem}{Theorem}[section]
  
  \newtheorem{lemma}[theorem]{Lemma}
  \newtheorem{corollary}[theorem]{Corollary}
  
  \newtheorem{proposition}[theorem]{Proposition}

\theoremstyle{definition}
  \newtheorem{definition}[theorem]{Definition}
  \newtheorem{example}[theorem]{Example}
  \newtheorem{assumption}[theorem]{Assumption}

\theoremstyle{remark}
  \newtheorem{remark}[theorem]{Remark}

\newcommand{\Var}{\text{Var}}
\newcommand{\lnorm}{\left\lVert}
\newcommand{\rnorm}{\right\rVert}

\newcommand{\iidmat}[3]{{#1}_{#2,#3}}

\newcommand{\triidmat}[3]{\hat{{#1}}_{#2,#3}}

\newcommand{\blmat}[2]{\mathcal{#1}_{#2}}

\newcommand{\trblmat}[2]{\hat{\mathcal{#1}}_{#2}}

\newcommand{\iident}[4]{#1_{#2,(#3,#4)}} 

\newcommand\diagA{\diag(\iidmat A n 1, \dots, \iidmat A n m)}

\begin{document}
\title[Outliers in the spectrum]{Outliers in the spectrum for products of independent random matrices} 

\author[N. Coston]{Natalie Coston}
\address{Department of Mathematics, University of Colorado at Boulder, Boulder, CO 80309 }
\email{natalie.coston@colorado.edu}

\author[S. O'Rourke]{Sean O'Rourke}
\address{Department of Mathematics, University of Colorado at Boulder, Boulder, CO 80309 }
\email{sean.d.orourke@colorado.edu}

\author[P. Wood]{Philip Matchett Wood}
\thanks{S. O'€™Rourke has been supported in part by NSF grant ECCS-1610003. P. Wood was partially supported by National Security Agency (NSA) Young Investigator Grant number H98230-14-1-0149.} 
\address{Department of Mathematics, University of Wisconsin-Madison, 480 Lincoln Dr., Madison, WI 53706 }
\email{pmwood@math.wisc.edu}

\begin{abstract}
For fixed $m \geq 1$, we consider the product of $m$ independent $n \times n$ random matrices with iid entries as $n \to \infty$.  Under suitable assumptions on the entries of each matrix, it is known that the limiting empirical distribution of the eigenvalues is described by the $m$-th power of the circular law.  Moreover, this same limiting distribution continues to hold if each iid random matrix is additively perturbed by a bounded rank deterministic error.  However, the bounded rank perturbations may create one or more outlier eigenvalues.  We describe the asymptotic location of the outlier eigenvalues, which extends a result of Tao \cite{Tout} for the case of a single iid matrix.  Our methods also allow us to consider several other types of perturbations, including multiplicative perturbations.  
\end{abstract}

\maketitle

\setcounter{tocdepth}{1}
\tableofcontents

\newpage

\section{Introduction}
\label{Sec:Intro}

This paper is concerned with the asymptotic behavior of outliers in the spectrum of bounded-rank perturbations of large random matrices.  We begin by fixing the following notation and introducing an ensemble of random matrices with independent entries.  

The \emph{eigenvalues} of an $n \times n$ matrix $M_n$ are the roots in $\mathbb{C}$ of the characteristic polynomial $\det (M_n- zI)$, where $I$ is the identity matrix.  We let $\lambda_1(M_n), \ldots, \lambda_n(M_n)$ denote the eigenvalues of $M_n$ counted with (algebraic) multiplicity.  The \emph{empirical spectral measure} $\mu_{M_n}$ of $M_n$ is given by
$$ \mu_{M_n} := \frac{1}{n} \sum_{j=1}^n \delta_{\lambda_j(M_n)}. $$
If $M_n$ is a random $n \times n$ matrix, then $\mu_{M_n}$ is also random.  In this case, we say $\mu_{M_n}$ converges weakly in probability (resp. weakly almost surely) to another Borel probability measure $\mu$ on the complex plane $\mathbb{C}$ if, for every bounded and continuous function $f:\mathbb{C} \to \mathbb{C}$, 
$$ \int_{\mathbb{C}} f d \mu_{M_n} \longrightarrow \int_{\mathbb{C}} f d \mu $$
in probability (resp. almost surely) as $n \to \infty$.  

Throughout the paper, we use asymptotic notation (such as $O, o$) under the assumption that $n \to \infty$; see Section \ref{sec:notation} for a complete description of our asymptotic notation. 

\subsection{iid random matrices}
In this paper, we focus on random matrices whose entries are independent and identically distributed.   

\begin{definition}[iid random matrix]
Let $\xi$ be a complex-valued random variable.  We say $X_n$ is an $n \times n$ \emph{iid random matrix} with atom variable $\xi$ if $X_n$ is an $n \times n$ matrix whose entries are independent and identically distributed (iid) copies of $\xi$.  
\end{definition}

The circular law describes the limiting distribution of the eigenvalues of an iid random matrix.  For any matrix $M$, we denote the Hilbert-Schmidt norm $\|M\|_2$ by the formula
\begin{equation} \label{eq:def:hs}
	\|M\|_2 := \sqrt{ \tr (M M^\ast) } = \sqrt{ \tr (M^\ast M)}. 
\end{equation}

\begin{theorem}[Circular law; Corollary 1.12 from \cite{TVesd}] \label{thm:circ}
Let $\xi$ be a complex-valued random variable with mean zero and unit variance.  For each $n \geq 1$, let $X_n$ be an $n \times n$ iid random matrix with atom variable $\xi$, and let $A_n$ be a deterministic $n \times n$ matrix.  If $\rank(A_n) = o(n)$ and $\sup_{n \geq 1} \frac{1}{n} \|A_n\|^2_2 < \infty$, then the empirical measure $\mu_{\frac{1}{\sqrt{n}} X_n + A_n}$ of $\frac{1}{\sqrt{n}} X_n + A_n$ converges weakly almost surely to the uniform probability measure on the unit disk centered at the origin in the complex plane as $n \to \infty$.  
\end{theorem}

This result appears as \cite[Corollary 1.12]{TVesd}, but is the culmination of work by many authors.  We refer the interested reader to the excellent survey \cite{BC} for further details.

From Theorem \ref{thm:circ}, we see that the low-rank perturbation $A_n$ does not affect the limiting spectral measure (i.e., the limiting measure is the same as the case when $A_n = 0$).  However, the perturbation $A_n$ may create one or more outliers.  An example of this phenomenon is illustrated in Figure \ref{Fig:GaussianM1}.

\begin{figure}[h]
	\includegraphics[scale=.4]{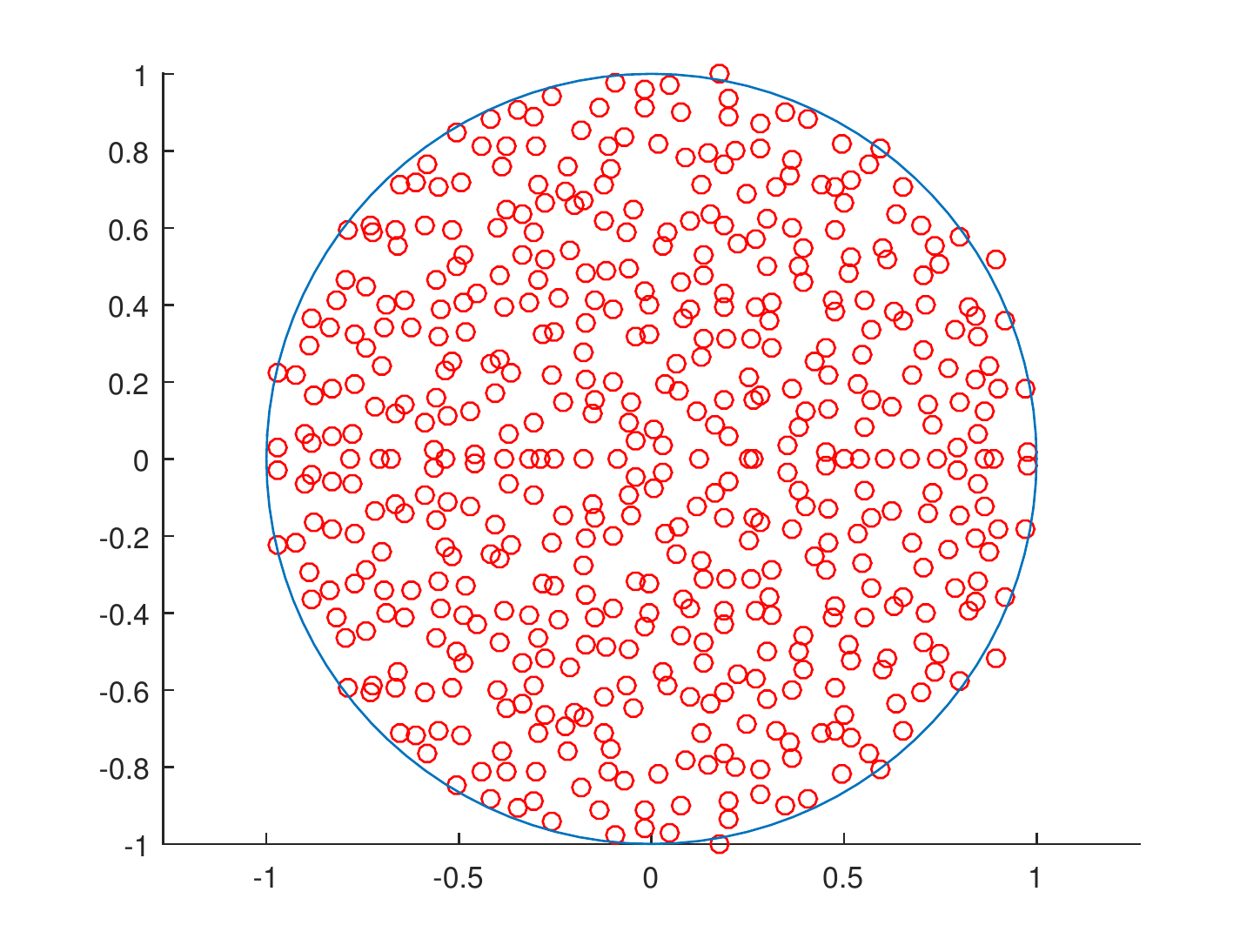}
	\includegraphics[scale=.4]{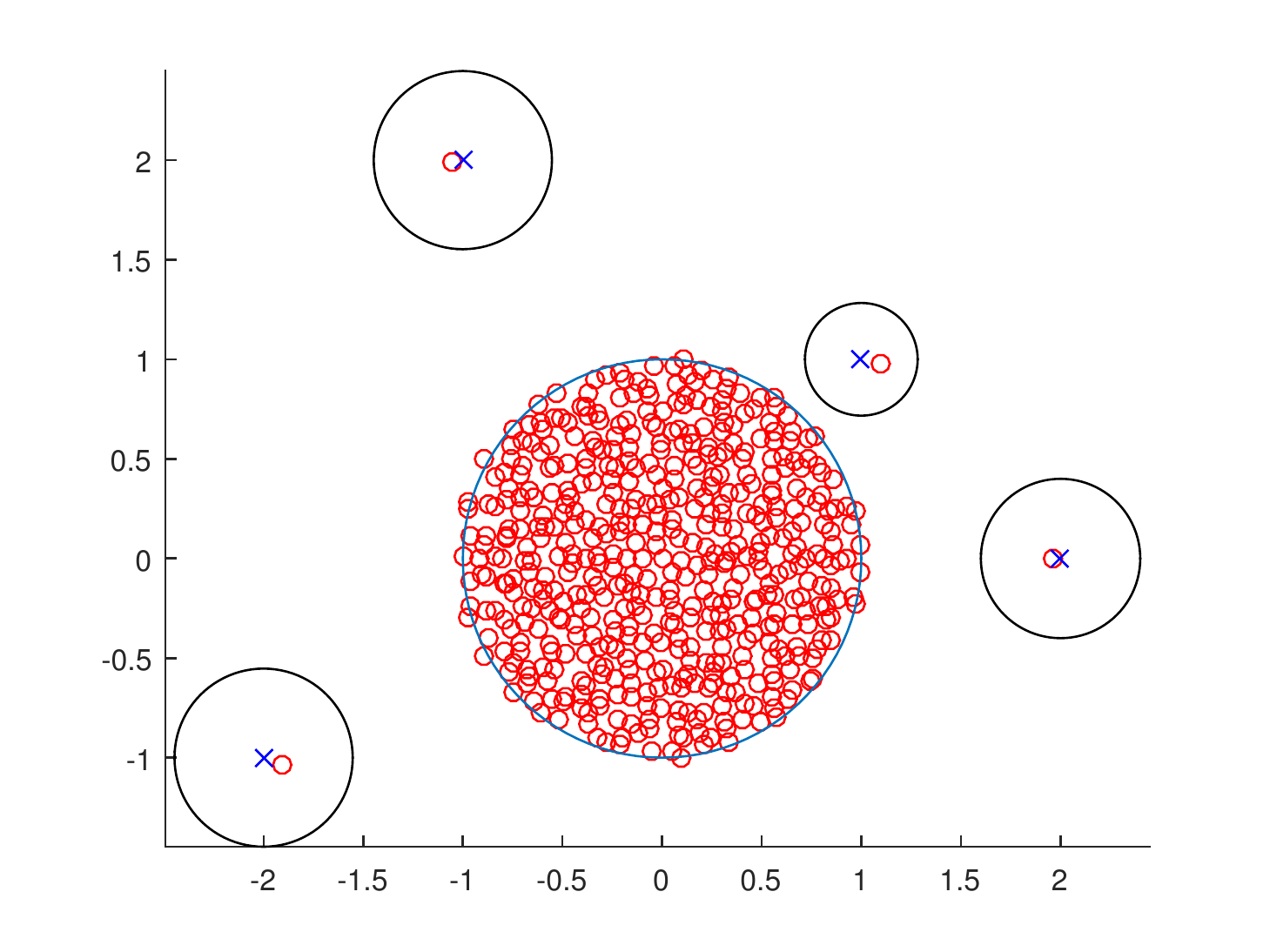}
	\caption{On the left, we have plotted the eigenvalues of a $500\times 500$ random matrix with iid standard Gaussian entries scaled by $\frac{1}{\sqrt{500}}$. Additionally, the unit circle is plotted for reference. The image on the right contains the eigenvalues of $\frac{1}{\sqrt{500}} X + A$, where $X$ is a $500\times 500$ random matrix with iid symmetric $\pm 1$ Bernoulli entries and $A=\text{diag}(1+i,2i-1,2,-i-2,0,\ldots,0)$.  For reference, we have also plotted each nonzero eigenvalue of $A$ with a cross.}
	\label{Fig:GaussianM1}
\end{figure}

Recall that the spectral radius of a square matrix $M$ is the largest eigenvalue of $M$ in absolute value.  Among other things, Theorem \ref{thm:circ} implies that with probability tending to one, the spectral radius of $\frac{1}{\sqrt{n}}X_n$ is at least $1 - o(1)$.  When the atom variable $\xi$ has finite fourth moment, it is possible to improve the lower bound on the spectral radius and give a matching upper bound.

\begin{theorem}[No outliers for iid matrices; Theorem 5.18 from \cite{BSbook}] \label{thm:nooutlier:iid}
Let $\xi$ be a complex-valued random variable with mean zero, unit variance, and finite fourth moment.  For each $n \geq 1$, let $X_n$ be an iid random matrix with atom variable $\xi$.  Then the spectral radius of $\frac{1}{\sqrt{n}}X_n$ converges to $1$ almost surely as $n \to \infty$.  
\end{theorem}

\begin{remark} \label{rem:twomomsr}
It is conjectured in \cite{BCCT} that the spectral radius of $\frac{1}{\sqrt{n}}X_n$ converges to $1$ in probability as $n \to \infty$ only assuming that $\xi$ has mean zero and unit variance.  
\end{remark}

Theorem \ref{thm:nooutlier:iid} asserts that almost surely all eigenvalues of $\frac{1}{\sqrt{n}}X_n$ are contained in the disk of radius $1 + o(1)$ centered at the origin.  However, as we saw in Figure \ref{Fig:GaussianM1}, this need not be the case for the eigenvalues of the additive perturbation $\frac{1}{\sqrt{n}}X_n + A_n$.  In this case, it is possible for eigenvalues of the perturbation to be larger than $1 +o(1)$, and Tao precisely describes the location of these outlying eigenvalues  in \cite{Tout}.  

\begin{theorem}[Outliers for small low-rank perturbations of iid matrices; Theorem 1.7 from \cite{Tout}] \label{thm:outlier:iid}
Let $\xi$ be a complex random variable with mean zero, unit variance, and finite fourth moment.  For each $n \geq 1$, let $X_n$ be a $n \times n$ random matrix whose entries are iid copies of $\xi$, and let $A_n$ be a deterministic matrix with rank $O(1)$ and operator norm $O(1)$.  Let $\eps > 0$, and suppose that for all sufficiently large $n$, there are no eigenvalues of $A_n$ in the band $\{z \in \C : 1 + \eps < |z| < 1 + 3\eps\}$, and there are $j$ eigenvalues $\lambda_1(A_n), \ldots, \lambda_j(A_n)$ for some $j = O(1)$ in the region $\{z \in \C : |z| \geq 1+3 \eps\}$.  Then, almost surely, for sufficiently large $n$, there are precisely $j$ eigenvalues $\lambda_1(\frac{1}{\sqrt{n}} X_n + A_n), \ldots, \lambda_j(\frac{1}{\sqrt{n}} X_n + A_n)$ of $\frac{1}{\sqrt{n}}X_n + A_n$ in the region $\{z \in \C : |z| \geq 1 + 2\eps\}$, and after labeling these eigenvalues properly, 
$$ \lambda_i\left( \frac{1}{\sqrt{n}} X_n + A_n \right) = \lambda_i(A_n) + o(1) $$
as $n \rightarrow \infty$ for each $1 \leq i \leq j$.  
\end{theorem}

Analogous results describing the location and fluctuation of the outlier eigenvalues have been obtained for many ensembles of random matrices; we refer the reader to \cite{BBC, BBP, BGCR, BGM, BGM2, BR, BR2, BGR, BCap, BelC, CDF1, CDF, CDFF, FP, KY, KY2, OR, OW, P, PRS, R, RS, Rochet, Tout} and references therein.  In particular, the results in \cite{R} extend Theorem \ref{thm:outlier:iid} by also describing the joint fluctuations of the outlier eigenvalues about their asymptotic locations.

\subsection{Products of independent iid matrices}

In this paper, we focus on the product of several independent iid matrices.  In this case, the analogue of the circular law (Theorem \ref{thm:circ}) is the following result from \cite{ORSV}, due to Renfrew, Soshnikov, Vu, and the second author of the current paper.    

\begin{theorem}[Theorem 2.4 from \cite{ORSV}] \label{thm:ORSV}
Fix an integer $m \geq 1$, and let $\tau > 0$.  Let $\xi_1, \ldots, \xi_m$ be real-valued random variables with mean zero, and assume, for each $1 \leq k \leq m$,  that $\xi_k$ has nonzero variance $\sigma^2_k$ and satisfies $\E|\xi_k|^{2 + \tau} < \infty$.  For each $n \geq 1$ and $1 \leq k \leq m$, let $X_{n,k}$ be an $n \times n$ iid random matrix with atom variable $\xi_k$, and let $A_{n,k}$ be a deterministic $n \times n$ matrix.  Assume $X_{n,1}, \ldots, X_{n,m}$ are independent.  If
$$ \max_{1 \leq k \leq m} \rank(A_{n,k}) = O(n^{1 - \eps}) \quad \text{and} \quad \sup_{n \geq 1} \max_{1 \leq k \leq m} \frac{1}{n} \|A_{n,k}\|_2^2 < \infty $$
for some fixed $\eps > 0$, then the empirical spectral measure $\mu_{P_n}$ of the product\footnote{Here and in the sequel, we use Pi (product) notation for products of matrices.  To avoid any ambiguity, if $M_1, \ldots, M_m$ are $n \times n$ matrices, we define the order of the product \[ \prod_{k=1}^m M_k := M_1 \cdots M_m. \]  In many cases, such as in Theorem \ref{thm:ORSV}, the order of matrices in the product is irrelevant by simply relabeling indices.}
$$ P_n := \prod_{k=1}^m \left( \frac{1}{\sqrt{n}} X_{n,k} + A_{n,k} \right) $$
converges weakly almost surely to a (non-random) probability measure $\mu$ as $n \to \infty$.  Here, the probability measure $\mu$ is absolutely continuous with respect to Lebesgue measure on $\mathbb{C}$ with density
\begin{equation} \label{eq:density}
    f(z) := \left\{
     \begin{array}{rr}
       \frac{1}{m \pi} \sigma^{-2/m} |z|^{\frac{2}{m} - 2}, & \text{if } |z| \leq \sigma, \\
       0, & \text{if } |z| > \sigma,
     \end{array}
   \right. 
\end{equation}
where $\sigma := \sigma_1 \cdots \sigma_m$.  
\end{theorem}
\begin{remark}
When $\sigma = 1$, the density in \eqref{eq:density} is easily related to the circular law (Theorem \ref{thm:circ}).  Indeed, in this case, $f$ is the density of $\psi^m$, where $\psi$ is a complex-valued random variable uniformly distributed on the unit disk centered at the origin in the complex plane.  
\end{remark}

Theorem \ref{thm:ORSV} is a special case of \cite[Theorem 2.4]{ORSV}.  Indeed, \cite[Theorem 2.4]{ORSV} applies to so-called elliptic random matrices, which generalize iid matrices.  Theorem \ref{thm:ORSV} and the results in \cite{ORSV} are stated only for real random variables, but the proofs can be extended to the complex setting.  Similar results have also been obtained in \cite{B, GTprod, OS}.  The Gaussian case was originally considered by Burda, Janik, and Waclaw \cite{BJW}; see also \cite{Bsurv}.  We refer the reader to \cite{AB,ABK, AIK, AIK2, AKW, AS, BJLNS, F, F2, KZ, S} and references therein for many other results concerning products of random matrices with Gaussian entries. 


\section{Main results}
\label{Sec:Main}

From Theorem \ref{thm:ORSV}, we see that the low-rank deterministic perturbations $A_{n,k}$ do not affect the limiting empirical spectral measure.  However, as was the case in Theorem \ref{thm:circ}, the perturbations may create one or more outlier eigenvalues.  The goal of this paper is to study the asymptotic behavior of these outlier eigenvalues.  In view of Theorem \ref{thm:outlier:iid}, we will assume the atom variables $\xi_1, \ldots, \xi_m$ have finite fourth moment.  

\begin{assumption} \label{assump:4th}
The complex-valued random variables $\xi_1, \ldots, \xi_m$ are said to satisfy Assumption \ref{assump:4th} if, for each $1 \leq k \leq m$, 
\begin{itemize}
\item the real and imaginary parts of $\xi_k$ are independent, 
\item $\xi_k$ has mean zero and finite fourth moment, and
\item $\xi_k$ has nonzero variance $\sigma_k^2$.
\end{itemize}
\end{assumption}

We begin with the analogue of Theorem \ref{thm:nooutlier:iid} for the product of $m$ independent iid matrices.  

\begin{theorem}[No outliers for products of iid matrices] \label{thm:nooutlier}
Let $m \geq 1$ be a fixed integer, and assume $\xi_1, \ldots, \xi_m$ are complex-valued random variables which satisfy Assumption \ref{assump:4th}.  For each $n \geq 1$, let $X_{n,1}, \ldots, X_{n,m}$ be independent $n \times n$ iid random matrices with atom variables $\xi_1, \ldots, \xi_m$, respectively.  Define the products
$$ P_n := n^{-m/2} X_{n,1} \cdots X_{n,m} $$
and $\sigma := \sigma_1 \cdots \sigma_m$.  Then, almost surely, the spectral radius of $P_n$ is bounded above by $\sigma + o(1)$ as $n \to \infty$.  In particular, for any fixed $\eps > 0$, almost surely, for $n$ sufficiently large, all eigenvalues of $P_n$ are contained in the disk $\{ z \in \mathbb{C} : |z| < \sigma + \eps \}$.  
\end{theorem}

%

\begin{remark}
A version of Theorem \ref{thm:nooutlier} was proven by Nemish in \cite{N} under the additional assumption that the atom variables $\xi_1, \ldots, \xi_m$ satisfy a sub-exponential decay condition.  In particular, this condition implies that all moments of $\xi_1, \ldots, \xi_m$ are finite.  Theorem \ref{thm:nooutlier} only requires the fourth moments of the atom variables to be finite. 
\end{remark}
\begin{remark}
In view of Remark \ref{rem:twomomsr}, it is natural to also conjecture that the spectral radius of $P_n$ is bounded above by $\sigma + o(1)$ in probability as $n \to \infty$ only assuming the atom variables $\xi_1, \ldots, \xi_m$ have mean zero and unit variance.  Here, we need the result to hold almost surely, and hence require the atom variables have finite fourth moment.  
\end{remark}  

In view of Theorem \ref{thm:ORSV}, it is natural to consider perturbations of the form
\[ P_n := \prod_{k=1}^m \left( \frac{1}{\sqrt{n}} X_{n,k} + A_{n,k} \right). \]
However, there are many other types of perturbations one might consider, such as multiplicative perturbations
\begin{equation} \label{eq:prodmult}
	P_n := \frac{1}{\sqrt{n}} X_{n,1} (I + A_{n,1}) \frac{1}{\sqrt{n}} X_{n,2} (I + A_{n,2}) \cdots \frac{1}{\sqrt{n}} X_{n,m} (I + A_{n,m}) 
\end{equation}
or perturbations of the form
\[ P_n := n^{-m/2} \prod_{k=1}^m X_{n,k} + A_n. \]
In any of these cases, the product $P_n$ can be written as
\[ P_n = n^{-m/2} X_{n,1} \cdots X_{n,m} + M_n + A_n, \]
where $A_n$ is deterministic and $M_n$ represents the ``mixed'' terms, each containing at least one random factor and one deterministic factor.  Our main results below show that only the deterministic term $A_n$ determines the location of the outliers.  The ``mixed'' terms $M_n$ do not effect the asymptotic location of the outliers.  

This phenomenon is most easily observed in the case of multiplicative perturbations \eqref{eq:prodmult}, for which there is no deterministic term (i.e., $A_n=0$ and the perturbation consists entirely of ``mixed'' terms).  In this case, the heuristic above suggests that there should be no outliers, and this is the content of the following theorem.  

\begin{theorem}[No outliers for products of iid matrices with multiplicative perturbations] Let $m \geq 1$ be a fixed integer, and assume $\xi_1, \ldots, \xi_m$ are complex-valued random variables which satisfy Assumption \ref{assump:4th}.  For each $n \geq 1$, let $X_{n,1}, \ldots, X_{n,m}$ be independent $n \times n$ iid random matrices with atom variables $\xi_1, \ldots, \xi_m$, respectively. In addition for any fixed integer $s\geq 1$, let $\iidmat{A}{n}{1},\iidmat{A}{n}{2},\ldots,\iidmat{A}{n}{s}$ be $n\times n$ deterministic matrices, each of which has rank $O(1)$ and operator norm $O(1)$. Define the product $P_{n}$ to be the product of the terms $$\frac{1}{\sqrt{n}}\iidmat{X}{n}{1},\ldots,\frac{1}{\sqrt{n}}\iidmat{X}{n}{m},\left(I+\iidmat{A}{n}{1}\right),\ldots,\left(I+\iidmat{A}{n}{s}\right)$$ 
in some fixed order. Then for any $\delta > 0$, almost surely, for sufficiently large $n$, $P_n$ has no eigenvalues in the region $\{z\in\C\;:\;|z|>\sigma+\delta\}$ where $\sigma := \sigma_1 \cdots \sigma_m$.  
\label{Thm:NoOutlierInProductPert}
\end{theorem} 

%
%

We now consider the case when there is a deterministic term and no ``mixed'' terms.  

\begin{theorem}[Outliers for small low-rank perturbations of product matrices] \label{thm:nomixed}
Let $m \geq 1$ be a fixed integer, and assume $\xi_1, \ldots, \xi_m$ are complex-valued random variables which satisfy Assumption \ref{assump:4th}.  For each $n \geq 1$, let $X_{n,1}, \ldots, X_{n,m}$ be independent $n \times n$ iid random matrices with atom variables $\xi_1, \ldots, \xi_m$, respectively.  In addition, let $A_n$ be an $n \times n$ deterministic matrix with rank $O(1)$ and operator norm $O(1)$.  Define
\begin{equation} \label{eq:nomixedproduct}
	P_n := n^{-m/2} \prod_{k=1}^m X_{n,k} + A_n 
\end{equation}
and $\sigma := \sigma_1 \cdots \sigma_m$.  Let $\eps > 0$, and suppose that for all sufficiently large $n$, there are no eigenvalues of $A_n$ in the band $\{z \in \C : \sigma + \eps < |z| < \sigma + 3\eps\}$, and there are $j$ eigenvalues $\lambda_1(A_n), \ldots, \lambda_j(A_n)$ for some $j = O(1)$ in the region $\{z \in \C : |z| \geq \sigma+3 \eps\}$.  Then, almost surely, for sufficiently large $n$, there are precisely $j$ eigenvalues $\lambda_1(P_n), \ldots, \lambda_j(P_n)$ of $P_n$ in the region $\{z \in \C : |z| \geq \sigma + 2\eps\}$, and after labeling these eigenvalues properly, 
$$ \lambda_i\left(P_n \right) = \lambda_i(A_n) + o(1) $$
as $n \rightarrow \infty$ for each $1 \leq i \leq j$.  
\end{theorem}

\begin{figure}[h]
	\includegraphics[scale=.5]{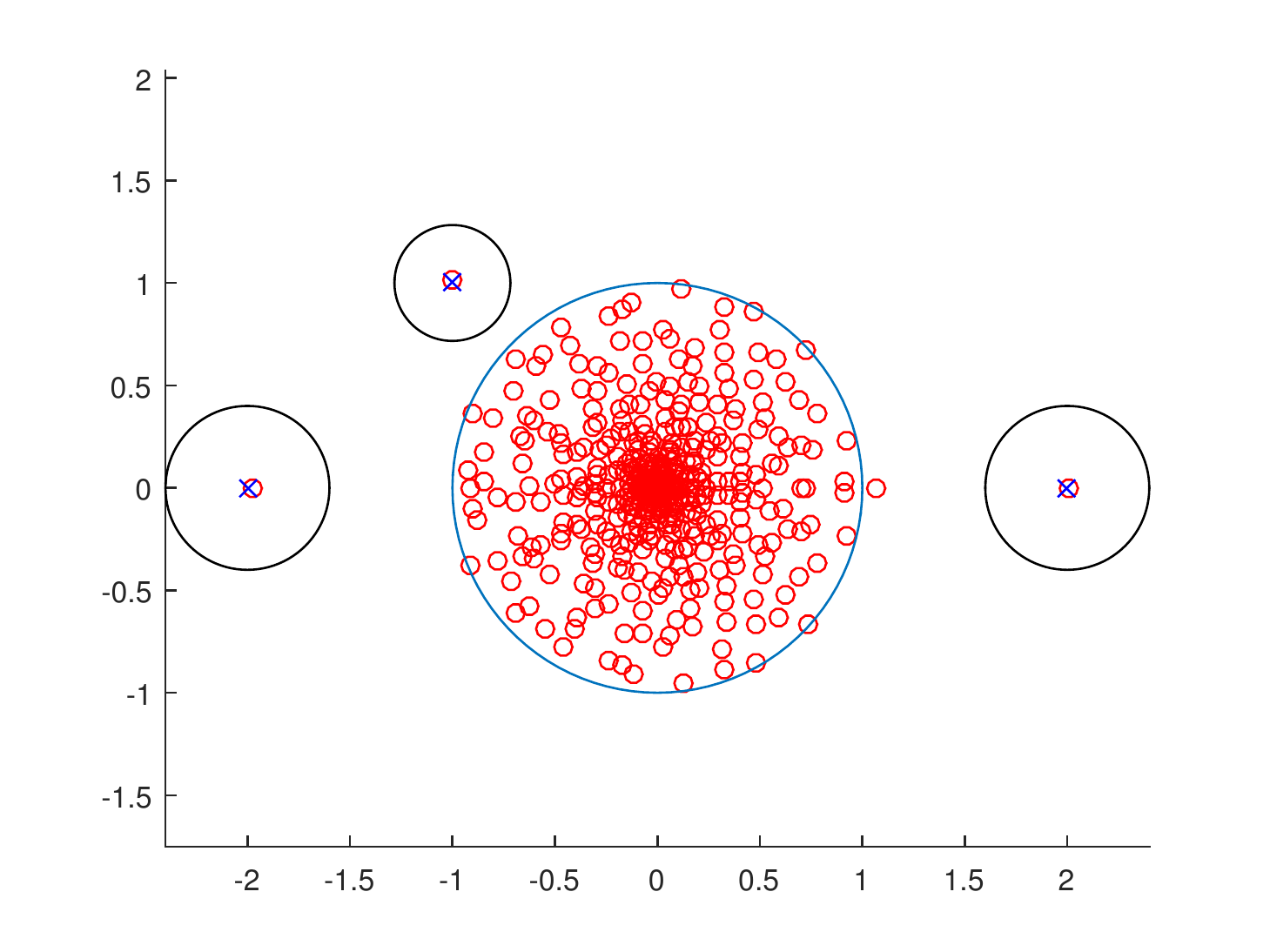}
	\caption{In this figure, we have plotted the eigenvalues of $(500)^{-2} X_1 X_2 X_3 X_4 + A$, where $X_1, \ldots, X_4$ are independent $500\times 500$ iid random matrices with symmetric $\pm 1$ Bernoulli entries and $A=\text{diag}(-1+i,-2,2,0,\ldots,0)$.  The majority of the eigenvalues cluster inside the unit disc with the exception of three outliers. These outliers are close to the eigenvalues of $A$, each of which is marked with a cross.}
	\label{Fig:PertProdWithOne}
\end{figure}

Figure \ref{Fig:PertProdWithOne} presents a numerical simulation of Theorem \ref{thm:nomixed}.  In the case that all the entries of $A_n$ take the same value, the product $P_n$ in \eqref{eq:nomixedproduct} can be viewed as a product matrix whose entries have the same nonzero mean.  Technically, Theorem \ref{thm:nomixed} cannot be applied in this case, since such a matrix $A_n$ does not have operator norm $O(1)$.  However, using a similar proof, we establish the following result.  

\begin{theorem}[Outliers for a product matrix with nonzero mean] \label{thm:nonzeromean}
Let $m \geq 1$ be an integer, and let $\mu \in \mathbb{C}$ be nonzero.  Assume $\xi_1, \ldots, \xi_m$ are complex-valued random variables which satisfy Assumption \ref{assump:4th}.  For each $n \geq 1$, let $X_{n,1}, \ldots, X_{n,m}$ be independent $n \times n$ iid random matrices with atom variables $\xi_1, \ldots, \xi_m$, respectively.  Let $\phi_n := \frac{1}{\sqrt{n}} (1, \ldots, 1)^\ast$ and fix $\gamma > 0$.  Define
$$ P_n := n^{-m/2} \prod_{k=1}^m X_{n,k} + \mu n^{\gamma} \phi_n \phi_n^\ast $$ 
and $\sigma := \sigma_1 \cdots \sigma_m$, and fix $\eps > 0$.  Then, almost surely, for $n$ sufficiently large, all eigenvalues of $P_n$ lie in the disk $\{z \in \mathbb{C} : |z| \leq \sigma + \eps\}$ with a single exception taking the value $\mu n^{\gamma} + o(1)$.   
\end{theorem}


Lastly, we consider the case of Theorem \ref{thm:ORSV}, where there are both ``mixed'' terms and a deterministic term.  

\begin{theorem} \label{thm:outliers}
Let $m \geq 1$ be an integer, and assume $\xi_1, \ldots, \xi_m$ are complex-valued random variables which satisfy Assumption \ref{assump:4th}.  For each $n \geq 1$, let $X_{n,1}, \ldots, X_{n,m}$ be independent $n \times n$ iid random matrices with atom variables $\xi_1, \ldots, \xi_m$, respectively.  In addition, for each $1 \leq k \leq m$, let $A_{n,k}$ be a deterministic $n \times n$ matrix with rank $O(1)$ and operator norm $O(1)$.  Define the products
\begin{equation} \label{eq:products}
	P_n := \prod_{k=1}^m \left( \frac{1}{\sqrt{n}} X_{n,k} + A_{n,k} \right), \quad A_n := \prod_{k=1}^m A_{n,k}, 
\end{equation}
and $\sigma := \sigma_1 \cdots \sigma_m$.  Let $\eps > 0$, and suppose that for all sufficiently large $n$, there are no eigenvalues of $A_n$ in the band $\{z \in \C : \sigma + \eps < |z| < \sigma + 3\eps\}$, and there are $j$ eigenvalues $\lambda_1(A_n), \ldots, \lambda_j(A_n)$ for some $j = O(1)$ in the region $\{z \in \C : |z| \geq \sigma+3 \eps\}$.  Then, almost surely, for sufficiently large $n$, there are precisely $j$ eigenvalues $\lambda_1(P_n), \ldots, \lambda_j(P_n)$ of the product $P_n$ in the region $\{z \in \C : |z| \geq \sigma + 2\eps\}$, and after labeling these eigenvalues properly, 
$$ \lambda_i\left(P_n \right) = \lambda_i(A_n) + o(1) $$
as $n \rightarrow \infty$ for each $1 \leq i \leq j$.  
\end{theorem}

Theorem \ref{thm:outliers} can be viewed as a generalization of Theorem \ref{thm:outlier:iid}.  In fact, when $m = 1$, Theorem \ref{thm:outliers} is just a restatement of Theorem \ref{thm:outlier:iid}.  However, the most interesting cases occur when $m \geq 2$.  Indeed, in these cases, Theorem \ref{thm:outliers} implies that the outliers of $P_n$ are asymptotically close to the outliers of the product $A_n$.  Specifically, if even one of the deterministic matrices $A_{n,k}$ is zero, asymptotically, there cannot be any outliers for the product $P_n$.  
Figure \ref{Fig:PertProdEach} presents a numerical simulation of Theorem \ref{thm:outliers}.

\begin{figure}[h]
	\includegraphics[scale=.5]{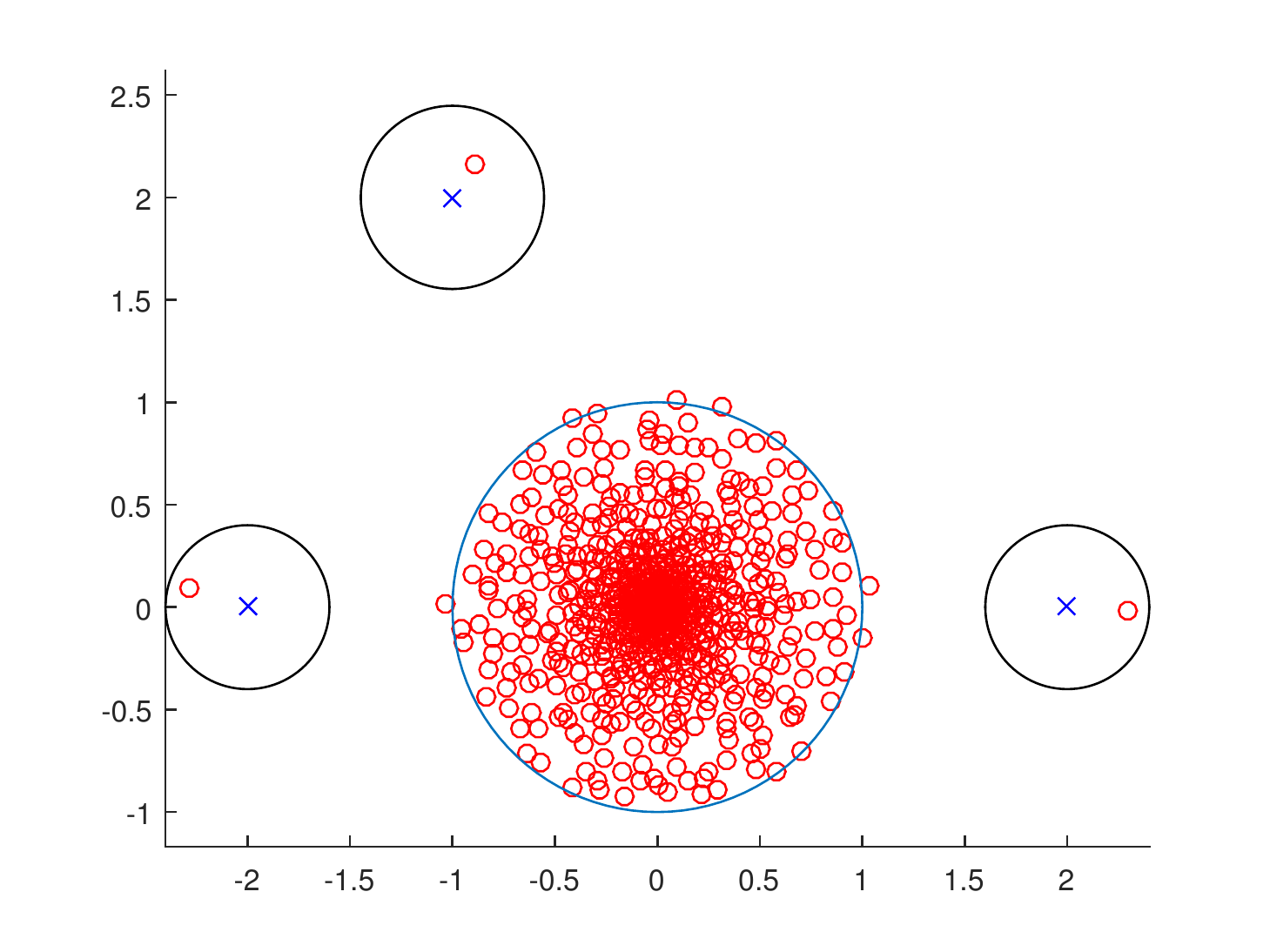}
	\caption{In the above figure, we display the eigenvalues of products of random matrices of the form $\prod_{k=1}^{5}\left(\frac{1}{\sqrt{1000}}X_{k}+A_{k}\right)$, where $X_1, \ldots, X_5$ are $1000\times 1000$ independent iid matrices with symmetric $\pm 1$ Bernoulli entries, and the product of the deterministic matrices $A_1, \ldots, A_5$ is $\text{diag}(-2,-1+2i,2,0,\ldots,0)$.  Each nonzero eigenvalue of the product $A_1 \cdots A_5$ is marked with a cross.  }
	\label{Fig:PertProdEach}
\end{figure}

\subsection{Outline}
The main results of this paper focus on iid matrices. However, one may ask if these results may be generalized to other matrix ensembles. These questions and others are discussed in Section \ref{Sec:RelatedResults}.  In Sections \ref{Sec:PreToolsAndNotation} and \ref{sec:overview}, we present some preliminary results and tools. In particular, Section \ref{Sec:PreToolsAndNotation} presents our notation and some standard linearization results that will be used throughout the paper. Section \ref{sec:overview} contains a key eigenvalue criterion lemma and a brief overview of the proofs of the main results. In Section \ref{sec:isotropic}, we state an isotropic limit law and use it to prove our main results.  
The majority of the paper (Sections \ref{Sec:TruncAndTools}--\ref{sec:combin}) is devoted to the proof of this isotropic limit law.  Section \ref{Sec:RelatedResultProofs} contains the proofs of the related results from Section \ref{Sec:RelatedResults}. A few auxiliary results and proofs are presented in the appendices.

\section{Related results, applications, and open questions}
\label{Sec:RelatedResults}
There are a number of related results which are similar to the main results of this paper, but do not directly follow from the theorems in Section 2. We discuss these results as well as some applications of our main results in this section.

\subsection{Related Results}

While our main results have focused on independent iid matrices, it is also possible to consider the case when the random matrices $X_{n,1}, \ldots, X_{n,m}$ are no longer independent.  In particular, we consider the extreme case where $X_{n,1} = \cdots = X_{n,m}$ almost surely.  In this case, we obtain the following results, which are analogs of the results from Section \ref{Sec:Main}.

\begin{theorem}[No outliers for multiplicative perturbations] Assume that $\xi$ is a complex-valued random variable which satisfies Assumption \ref{assump:4th} with $\sigma^2:=\var(\xi)$.  For each $n \geq 1$, let $X_{n}$ be an $n \times n$ iid random matrix with atom variable $\xi$. In addition for any finite integer $s\geq 1$, let $\iidmat{A}{n}{1},\iidmat{A}{n}{2},\ldots,\iidmat{A}{n}{s}$ be $n\times n$ deterministic matrices, each of which has rank $O(1)$ and operator norm $O(1)$. Define the product $P_{n}$ to be the product of $m$ copies of $\frac{1}{\sqrt{n}}X_{n}$ with the terms $$\left(I+\iidmat{A}{n}{1}\right),\left(I+\iidmat{A}{n}{2}\right),\ldots,\left(I+\iidmat{A}{n}{s}\right)$$ 
	in some fixed order. Then for any $\delta > 0$, almost surely, for sufficiently large $n$, $P_n$ has no eigenvalues in the region $\{z\in\C\;:\;|z|>\sigma^{m}+\delta\}$.
	\label{Thm:NoOutlierInPowerPert}
\end{theorem} 

For a single additive perturbation, we have the following analog of Theorem~\ref{thm:nomixed}.

\begin{theorem}[Outliers for a single additive perturbation] \label{thm:nomixedinpower}
	Assume $\xi$ is a complex-valued random variable which satisfies Assumption \ref{assump:4th} with $\sigma^2 := \var(\xi)$.  For each $n \geq 1$, let $X_{n}$ be an $n \times n$ iid random matrix with atom variable $\xi$. In addition, let $A_n$ be an $n \times n$ deterministic matrix with rank $O(1)$ and operator norm $O(1)$.  Define
	\begin{equation} \label{eq:nomixedpower}
	P_n := n^{-m/2} X_{n}^{m} + A_n .
	\end{equation}
	Let $\eps > 0$, and suppose that for all sufficiently large $n$, there are no eigenvalues of $A_n$ in the band $\{z \in \C : \sigma^{m} + \eps < |z| < \sigma^{m} + 3\eps\}$, and there are $j$ eigenvalues $\lambda_1(A_n), \ldots, \lambda_j(A_n)$ for some $j = O(1)$ in the region $\{z \in \C : |z| \geq \sigma^{m}+3 \eps\}$. Then, almost surely, for sufficiently large $n$, there are precisely $j$ eigenvalues $\lambda_1(P_n), \ldots, \lambda_j(P_n)$ of $P_n$ in the region $\{z \in \C : |z| \geq \sigma^{m} + 2\eps\}$, and after labeling these eigenvalues properly, 
	$$ \lambda_i\left(P_n \right) = \lambda_i(A_n) + o(1) $$
	as $n \rightarrow \infty$ for each $1 \leq i \leq j$.  
\end{theorem}

Note that Theorem~\ref{thm:nonzeromean} can also be generalized in an analogous way to Theorem~\ref{thm:nomixedinpower} above.

Lastly, we have the following analog of Theorem \ref{thm:outliers}.  

\begin{theorem}
	\label{Thm:RepeatedProdOutliers}
	Assume $\xi$ is a complex-valued random variable which satisfies Assumption \ref{assump:4th} with $\sigma^2 := \var(\xi)$.  For each $n \geq 1$, let $X_{n}$ be an $n \times n$ iid random matrix with atom variable $\xi$.  In addition, let $m \geq 1$ be an integer and for each $1 \leq k \leq m$, let $A_{n,k}$ be a deterministic $n \times n$ matrix with rank $O(1)$ and operator norm $O(1)$.  Define the products
	\begin{equation} \label{Equ:RepeatedProds}
	P_n := \prod_{k=1}^m \left( \frac{1}{\sqrt{n}} X_{n} + A_{n,k} \right), \quad A_n := \prod_{k=1}^m A_{n,k}.
	\end{equation}
	Let $\eps > 0$, and suppose that for all sufficiently large $n$, there are no eigenvalues of $A_n$ in the band $\{z \in \C : \sigma^{m} + \eps < |z| < \sigma^{m} + 3\eps\}$, and there are $j$ eigenvalues $\lambda_1(A_n), \ldots, \lambda_j(A_n)$ for some $j = O(1)$ in the region $\{z \in \C : |z| \geq \sigma^{m}+3 \eps\}$.  Then, almost surely, for sufficiently large $n$, there are precisely $j$ eigenvalues $\lambda_1(P_n), \ldots, \lambda_j(P_n)$ of the product $P_n$ in the region $\{z \in \C : |z| \geq \sigma^{m} + 2\eps\}$, and after labeling these eigenvalues properly, 
	$$ \lambda_i\left(P_n \right) = \lambda_i(A_n) + o(1) $$
	as $n \rightarrow \infty$ for each $1 \leq i \leq j$.  
\end{theorem}


The proofs of these results are presented in Section \ref{Sec:RelatedResultProofs} and use similar techniques to the proofs of the main results from Section \ref{Sec:Main}. 

\subsection{Applications} Random matrices are useful tools in the study of many physically motivated systems, and we note here two potential applications for products of perturbed random matrices.  First, iid Gaussian matrices can be used to model neural networks as in, for example, \cite{AR, AbSu, ACHLM}. In the case of a linear version of the feed-forward networks in \cite{AbSu}, the model becomes a perturbation of a product of iid random matrices, which, if the interactions in the model were fixed, could potentially be analyzed using approaches in the current paper.  Second, one can conceive of a dynamical system (see, for example, \cite{I}) evolving according to a matrix equation, which, when iterated, would lead a matrix product of the form discussed in Theorem~\ref{thm:outliers}.

\subsection{Open Questions}

While our main results have focused on iid matrices, it is natural to ask if the same results can be extended to other matrix models.  For example, Theorem \ref{thm:ORSV} and the results in \cite{ORSV} also hold for products of so-called elliptic random matrices.  However, the techniques used in this paper (in particular, the combinatorial techniques in Section \ref{sec:combin}) rely heavily on the independence of the entries of each matrix.  It is an interesting question whether an alternative proof can be found for the case when the entries of each matrix are allowed to be dependent.  

In this paper, we have focused on the asymptotic location of the outlier eigenvalues.  One can also ask about the fluctuations of the outliers.  For instance, in \cite{R}, the joint fluctuations of the outlier eigenvalues from Theorem \ref{thm:outlier:iid} were studied.  We plan to pursue this question elsewhere.  


\section{Preliminary tools and notation}
\label{Sec:PreToolsAndNotation}

This section is devoted to introducing some additional concepts and notation required for the proofs of our main results.  
In Section \ref{sec:overview}, we present a brief overview of our proofs and explain how these concepts will be used.  

\subsection{Notation} \label{sec:notation}

We use asymptotic notation (such as $O,o, \Omega$) under the assumption that $n \to \infty$.  In particular, $X= O(Y)$, $Y = \Omega(X)$, $X \ll Y$, and $Y \gg X$ denote the estimate $|X| \leq C Y$, for some constant $C > 0$ independent of $n$ and for all $n \geq C$.  If we need the constant $C$ to depend on another constant, e.g. $C = C_k$, we indicate this with subscripts, e.g. $X = O_{k}(Y)$, $Y = \Omega_k(X)$, $X \ll_k Y$, and $Y\gg_k X$.  We write $X = o(Y)$ if $|X| \leq c(n) Y$ for some sequence $c(n)$ that goes to zero as $n \to \infty$.  Specifically, $o(1)$ denotes a term which tends to zero as $n \to \infty$.  If we need the sequence $c(n)$ to depend on another constant, e.g. $c(n) = c_k(n)$, we indicate this with subscripts, e.g. $X = o_k(Y)$.  

Throughout the paper, we view $m$ as a fixed integer.  Thus, when using asymptotic notation, we will allow the implicit constants (and implicit rates of convergence) to depend on $m$ without including $m$ as a subscript (i.e. we will not write $O_m$ or $o_m$).  

An event $E$, which depends on $n$, is said to hold with \emph{overwhelming probability} if $\Prob(E) \geq 1 - O_C(n^{-C})$ for every constant $C > 0$.  We let $\oindicator{E}$ denote the indicator function of the event $E$, and we let $E^{c}$ denote the complement of the event $E$.  We write a.s. for almost surely.  

For a matrix $M$, we let $\|M\|$ denote the spectral norm of $M$, and we let $\|M\|_2$ denote the Hilbert-Schmidt norm of $M$ (defined in \eqref{eq:def:hs}). We denote the eigenvalues of an $n \times n$ matrix $M$ by $\lambda_{1}(M),\ldots,\lambda_{n}(M)$, and we let $\rho(M):=\max\{|\lambda_{1}(M)|,\ldots,|\lambda_{n}(M)|\}$ denote its spectral radius. We let $I_n$ denote the $n \times n$ identity matrix and $0_n$ denote the $n \times n$ zero matrix.  Often we will just write $I$ (or $0$) for the identity matrix (alternatively, zero matrix) when the size can be deduced from context. 

The singular values of an $n\times n$ matrix $M_{n}$ are the non-negative square roots of the eigenvalues of the matrix $M_{n}^{*}M_{n}$ and we will denote their ordered values $s_{1}(M_{n})\geq s_{2}(M_{n})\geq \cdots \geq s_{n}(M_{n})$.

We let $C$ and $K$ denote constants that are non-random and may take on different values from one appearance to the next.  The notation $K_p$ means that the constant $K$ depends on another parameter $p$.  We allow these constants to depend on the fixed integer $m$ without explicitly denoting or mentioning this dependence. For a positive integer $N$, we let $[N]$ denote the discrete interval $\{1, \ldots, N\}$.  For a finite set $S$, we let $|S|$ denote its cardinality.  We let $\sqrt{-1}$ denote the imaginary unit and reserve $i$ as an index.

\subsection{Linearization}
Let $M_1, \ldots, M_m$ be $n \times n$ matrices, and suppose we wish to study the product $M_1 \cdots M_m$.  A useful trick is to linearize this product and instead consider the $mn \times mn$ block matrix 
\begin{equation} \label{def:M}
	\mathcal{M} := \begin{bmatrix} 
				0 &  {M}_{1} &     &            & 0       \\
                         		0 & 0    & {M}_{2} &            & 0        \\      
                           		&      & \ddots & \ddots     &         \\
                         		0 &      &     &          0 & {M}_{m-1} \\
                       		{M}_{m} &      &     &            &  0      
	\end{bmatrix}.
\end{equation}
The following proposition relates the eigenvalues of $\mathcal{M}$ to the eigenvalues of the product $M_1 \cdots M_m$.  We note that similar linearization tricks have been used previously; see, for example, \cite{A, BJW, HT, ORSV, OS} and references therein.  

\begin{proposition} \label{prop:linear}
Let $M_1, \ldots, M_m$ be $n \times n$ matrices.  Let $P := M_1 \cdots M_m$, and assume $\mathcal{M}$ is the $mn \times mn$ block matrix defined in \eqref{def:M}.  Then
$$ \det(\mathcal{M}^m - z I) = [\det( P - z I)]^m $$
for every $z \in \mathbb{C}$.  In other words, the eigenvalues of $\mathcal{M}^m$ are the eigenvalues of $P$, each with multiplicity $m$.  
\end{proposition}
\begin{proof}
A simple computation reveals that $\mathcal{M}^m$ is a block diagonal matrix of the form
$$ \mathcal{M}^m = \begin{bmatrix} 
				Z_1 &  & 0 \\
				 & \ddots &  \\
				0 &  & Z_m 
			\end{bmatrix}, $$
where $Z_1 := P$ and
$$ Z_k := M_k \cdots M_m M_1 \cdots M_{k-1} $$
for $1 < k \leq m$.  Since each product $Z_2, \ldots, Z_m$ has the same characteristic polynomial\footnote{This fact can easily be deduced from Sylvester's determinant theorem; see \eqref{eq:sylvester}.} as $P$, it follows that
$$ \det( \mathcal{M}^m - zI) = \prod_{k=1}^m \det(Z_k - z I) = [\det (P - zI)]^m $$
for all $z \in \mathbb{C}$.  
\end{proof}

We will exploit Proposition \ref{prop:linear} many times in the coming proofs.  For instance, in order to study the product $X_{n,1} \cdots X_{n,m}$, we will consider the $mn \times mn$ block matrix
\begin{equation} \label{def:Y}
	\mathcal{Y}_n := \begin{bmatrix} 
				0 &  {X}_{n,1} &     &            & 0       \\
                         		0 & 0    & {X}_{n,2} &            & 0        \\      
                           		&      & \ddots & \ddots     &         \\
                         		0 &      &     &          0 & {X}_{n,m-1} \\
                       		{X}_{n,m} &      &     &            &  0      
	\end{bmatrix}
\end{equation}
and its resolvent
\begin{equation} \label{def:G}
	\mathcal{G}_n(z) := \left( \frac{1}{\sqrt{n}} \mathcal{Y}_n - z I \right)^{-1},
\end{equation}
defined for $z \in \mathbb{C}$ provided $z$ is not an eigenvalue of $\frac{1}{\sqrt{n}} \mathcal{Y}_n$.  We study the location of the eigenvalues of $\frac{1}{\sqrt{n}} \mathcal{Y}_n$ in Theorem \ref{thm:isotropic} below.  

Similarly, when we deal with the deterministic $n \times n$ matrices $A_{n,1}, \ldots, A_{n,m}$, it will be useful to consider the analogous $mn \times mn$ block matrix 
\begin{equation} \label{def:A}
	\mathcal{A}_n := \begin{bmatrix} 
				0 &  {A}_{n,1} &     &            & 0       \\
                         		0 & 0    & {A}_{n,2} &            & 0        \\      
                           		&      & \ddots & \ddots     &         \\
                         		0 &      &     &          0 & {A}_{n,m-1} \\
                       		{A}_{n,m} &      &     &            &  0      
			\end{bmatrix}. 
\end{equation}

\subsection{Matrix notation}
Here and in the sequel, we will deal with matrices of various sizes.  The most common dimensions are $n \times n$ and $N \times N$, where we take $N := mn$.  Unless otherwise noted, we denote $n \times n$ matrices by capital letters (such as $M, X, A$) and larger $N \times N$ matrices using calligraphic symbols (such as $\mathcal{M}$, $\mathcal{Y}$, $\mathcal{A}$).  

If $M$ is an $n \times n$ matrix and $1 \leq i,j \leq n$, we let $M_{ij}$ and $M_{(i,j)}$ denote the $(i,j)$-entry of $M$.  Similarly, if $\mathcal{M}$ is an $N \times N$ matrix, we let $\mathcal{M}_{ij}$ and $\mathcal{M}_{(i,j)}$ denote the $(i,j)$-entry of $\mathcal{M}$ for $1 \leq i,j \leq N$.  However, in many instances, it is best to view $N \times N$ matrices as block matrices with $n \times n$ entries.  To this end, we introduce the following notation.  Let $\mathcal{M}$ be an $N \times N$ matrix.  For $1 \leq a, b \leq m$, we let $\mathcal{M}^{[a,b]}$ denote the $n \times n$ matrix which is the $(a,b)$-block of $\mathcal{M}$.    
For convenience, we extend this notation to include the cases where $a = m+1$ or $b= m+1$ by taking the value $m+1$ to mean $1$ (i.e., modulo $m$).  For instance, $\mathcal{M}^{[m+1, m]} = \mathcal{M}^{[1,m]}$.  
For $1 \leq i,j \leq n$, the notation $\mathcal{M}^{[a,b]}_{ij}$ or $\mathcal{M}^{[a,b]}_{(i,j)}$ denotes the $(i,j)$-entry of $\mathcal{M}^{[a,b]}$.  

Sometimes we will deal with $n \times n$ matrices notated with a subscript such as $M_n$.  In this case, for $1 \leq i,j \leq n$, we write $(M_n)_{ij}$ or $M_{n,(i,j)}$ to denote the $(i,j)$-entry of $M_n$.  Similarly, if $\mathcal{M}_n$ is an $N \times N$ matrix, we write $\mathcal{M}^{[a,b]}_{n, (i,j)}$ to denote the $(i,j)$-entry of the block $\mathcal{M}_n^{[a,b]}$. 

In the special case where we deal with a vector, the notation is the same, but only one index or block will be specified. In particular, if $v$ is a vector in $\C^{N}$, then $v_{i}$ denotes the $i$-th entry of $v$. If we consider $v$ to be a block vector with $m$ blocks of size $n$, then $v^{[1]}, \ldots, v^{[m]}$ denote these blocks, i.e., each $v^{[a]}$ is an $n$-vector, and 
$$ v = \begin{pmatrix} v^{[1]} \\ \vdots \\ v^{[m]} \end{pmatrix}. $$
In addition, $v^{[a]}_{i}$ denotes the $i$-th entry in block $a$.

\subsection{Singular Value Inequalities}

For an $n \times n$ matrix $M$, recall that $s_{1}(M)\geq\dots\geq s_{n}(M)$ denote its ordered singular values.  We will need the following elementary bound concerning the largest and smallest singular values.  
\begin{proposition}
	Let $M$ be an $n\times n$ matrix and assume $E\subseteq\C$ such that $$\inf_{z\in E} s_{n}(M-zI)\geq c$$ for some constant $c>0$. Then
	$$\sup_{z\in E} \|G(z) \|\leq \frac{1}{c}$$
	where $G(z)=(M-zI)^{-1}$.
	\label{Prop:LargeAndSmallSingVals} 
\end{proposition}
\begin{proof}
	First, observe that for any $z\in E$, we have that $z$ is not an eigenvalue of $M$ and so $M-zI$ is invertible and $G(z)$ exists. Recall that if $s$ is a singular value of $M-zI$ and $M - zI$ is invertible, then $1/s$ is a singular value of $(M-zI)^{-1}$. Thus, we conclude that
	\begin{align*}
	\sup_{z\in E} s_{1}(G(z))&=\sup_{z\in E}\frac{1}{s_{n}(M-zI)}\\
	&=\frac{1}{\inf_{z\in E}s_{n}(M-zI)}\\
	&\leq \frac{1}{c},
	\end{align*} 
	as desired.  
\end{proof}

\section{Eigenvalue criterion lemma and an overview of the proof} \label{sec:overview}

Let us now briefly overview the proofs of our main results.  One of the key ingredients is the eigenvalue criterion lemma presented below (Lemma \ref{lemma:eigenvalue}), which is based on Sylvester's determinant theorem:
\begin{equation} \label{eq:sylvester}
	\det (I + AB) = \det (I + BA) 
\end{equation}
whenever $A$ is an $n \times k$ matrix and $B$ is a $k \times n$ matrix.  In particular, the left-hand side of \eqref{eq:sylvester} is an $n \times n$ determinant and the right-hand side is a $k \times k$ determinant.  

For concreteness, let us focus on the proof of Theorem \ref{thm:outliers}.  That is, we wish to study the eigenvalues of
$$ P_n := \prod_{k=1}^m \left( \frac{1}{\sqrt{n}} X_{n,k} + A_{n,k} \right) $$
outside the disk $\{z \in \mathbb{C} : |z| \leq \sigma + 2\eps\}$.  We first linearize the problem by invoking Proposition \ref{prop:linear} with the matrix $\frac{1}{\sqrt{n}} \mathcal{Y}_n + \mathcal{A}_n$, where $\mathcal{Y}_n$ and $\mathcal{A}_n$ are defined in \eqref{def:Y} and \eqref{def:A}.  Indeed, by Proposition \ref{prop:linear}, it suffices to study the eigenvalues of $\frac{1}{\sqrt{n}} \mathcal{Y} _n + \mathcal{A}_n$ outside of the disk $\{z \in \mathbb{C} : |z| \leq \sigma^{1/m} + \delta\}$ for some $\delta > 0$ (depending on $\sigma$, $\eps$, and $m$). Let us suppose that $\mathcal{A}_n$ is rank one.  In other words, assume $\mathcal{A}_n = v u^\ast$ for some $u, v \in \mathbb{C}^{mn}$.  In order to study the outlier eigenvalues, we will need to solve the equation 
\begin{equation} \label{eq:firstdet}
	\det \left( \frac{1}{\sqrt{n}} \mathcal{Y}_n + \mathcal{A}_n - zI \right) = 0 
\end{equation}
for $z \in \mathbb{C}$ with $|z| > \sigma^{1/m} + \delta$.  Assuming $z$ is not eigenvalue of $\mathcal{Y}_n$, we can rewrite \eqref{eq:firstdet} as
$$ \det \left( I + \mathcal{G}_n(z) \mathcal{A}_n \right) = 0, $$
where the resolvent $\mathcal{G}_n(z)$ is defined in \eqref{def:G}.  
From \eqref{eq:sylvester} and the fact that $\mathcal{A}_n = v u^\ast$, we find that this reduces to solving
$$ 1 + u^\ast \mathcal{G}_n(z) v = 0. $$
Thus, the problem of locating the outlier eigenvalues reduces to studying the resolvent $\mathcal{G}_n(z)$.  In particular, we develop an isotropic limit law in Section \ref{sec:isotropic} to compute the limit of $u^\ast \mathcal{G}_n(z) v$.  This limit law is inspired by the isotropic semicircle law developed by Knowles and Yin in \cite{KY, KY2} for Wigner random matrices as well as the isotropic law verified in \cite{OR} for elliptic matrices.  

The general case, when $\mathcal{A}_n$ is not necessarily rank one, is similar.  In this case, we will exploit the following criterion to characterize the outlier eigenvalues.  

\begin{lemma}[Eigenvalue criterion] \label{lemma:eigenvalue}
Let $Y$ and $A$ be $n \times n$ matrices, and assume $A = BC$, where $B$ is an $n \times k$ matrix and $C$ is a $k \times n$ matrix.  Let $z$ be a complex number which is not an eigenvalue of $Y$.  Then $z$ is an eigenvalue of $Y + A$ if and only if 
$$ \det \left(I_k + C \left( Y - z I_n \right)^{-1} B \right) = 0. $$
\end{lemma}
\begin{remark} \label{rem:arg}
The proof of Lemma \ref{lemma:eigenvalue} actually reveals that
\begin{equation} \label{eq:iddet}
	\det \left(I_k + C \left( Y - z I_n \right)^{-1} B \right) = \frac{\det \left( Y + A - zI \right) }{\det \left( Y - z I \right) } 
\end{equation}
provided the denominator does not vanish.  Versions of this identity have appeared in previous publications including \cite{AGG,BGM,BGM2,BR,CDF1,CDF, KY, KY2, OR, OW, PRS, RS, Tout}.  
\end{remark}
\begin{proof}[Proof of Lemma \ref{lemma:eigenvalue}]
Assume $z$ is not an eigenvalue of $Y$. Then $\det (Y - z I) \neq 0$ and
\begin{align*}
	\det (Y + A - z I) &= \det (Y - z I) \det ( I + (Y - z I)^{-1} A) \\
	&= \det (Y- z I) \det (I + (Y - zI)^{-1} BC) .
\end{align*}
Thus, by \eqref{eq:sylvester}, $z$ is an eigenvalue of $Y +A$ if and only if 
$$ \det (I + C (Y-zI)^{-1} B) = 0, $$
as desired.  
\end{proof}

Another identity we will make use of is the Resolvent Identity, which states that \begin{equation}
A^{-1}-B^{-1}=A^{-1}(B-A)B^{-1}
\label{Equ:ResolventIndentity}
\end{equation}
whenever $A$ and $B$ are invertible.

\section{Isotropic limit law and the proofs of the main theorems}
\label{sec:isotropic}

This section is devoted to the proofs of Theorems \ref{thm:nooutlier}, \ref{Thm:NoOutlierInProductPert}, \ref{thm:nomixed}, \ref{thm:nonzeromean}, and \ref{thm:outliers}.  The key ingredient is the following result concerning the properties of the resolvent $\mathcal{G}_n(z)$.   

\begin{theorem}[Isotropic limit law] \label{thm:isotropic}
Let $m \geq 1$ be a fixed integer, and assume $\xi_1, \ldots, \xi_m$ are complex-valued random variables with mean zero, unit variance, finite fourth moments, and independent real and imaginary parts.  For each $n \geq 1$, let $X_{n,1}, \ldots, X_{n,m}$ be independent $n \times n$ iid random matrices with atom variables $\xi_1, \ldots, \xi_m$, respectively. Recall that $\mathcal{Y}_n$ is defined in \eqref{def:Y} and its resolvent $\mathcal{G}_n(z)$ is defined in \eqref{def:G}. Then, for any fixed $\delta > 0$, the following statements hold.  
\begin{enumerate}[label=(\roman*)]
\item \label{item:invertible} Almost surely, for $n$ sufficiently large, the eigenvalues of $\frac{1}{\sqrt{n}} \mathcal{Y}_n$ are contained in the disk $\{z \in \mathbb{C} : |z| \leq 1 + \delta \}$.  In particular, this implies that almost surely, for $n$ sufficiently large, the matrix $\frac{1}{\sqrt{n}} \mathcal{Y}_n - z I$ is invertible for every $z \in \mathbb{C}$ with $|z| > 1 + \delta$.
\item \label{item:invtbnd} There exists a constant $c > 0$ (depending only on $\delta$) such that almost surely, for $n$ sufficiently large, 
$$ \sup_{z \in \mathbb{C} : |z| > 1 + \delta} \| \mathcal{G}_n(z) \| \leq c. $$
\item \label{item:isotropic} For each $n \geq 1$, let $u_n, v_n \in \mathbb{C}^{mn}$ be deterministic unit vectors.  Then 
$$ \sup_{z \in \mathbb{C} : |z| > 1 + \delta} \left| u_n^\ast \mathcal{G}_n(z) v_n + \frac{1}{z} u_n^\ast v_n \right| \longrightarrow 0 $$
almost surely as $n \to \infty$. 
\end{enumerate}
\end{theorem}

We conclude this section with the proofs of Theorems \ref{thm:nooutlier}, \ref{Thm:NoOutlierInProductPert}, \ref{thm:nomixed}, \ref{thm:nonzeromean}, and \ref{thm:outliers} assuming Theorem \ref{thm:isotropic}. Sections \ref{Sec:TruncAndTools}--\ref{sec:combin} are devoted to the proof of Theorem \ref{thm:isotropic}.

\subsection{Proof of Theorem \ref{thm:nooutlier}}

Consider 
$$P_{n}:=n^{-m/2}\iidmat{X}{n}{1}\cdots\iidmat{X}{n}{m}$$ 
and note that by rescaling by $\frac{1}{\sigma}$, it is sufficient to assume that $\sigma_{i}=1$ for all $1\leq i\leq m$. By Proposition \ref{prop:linear}, the eigenvalues of $P_{n}$ are precisely the eigenvalues of $n^{-m/2}\blmat{Y}{n}^{m}$, each with multiplicity $m$. Additionally, the eigenvalues of $n^{-m/2}\blmat{Y}{n}^{m}$ are exactly the $m$-th powers of the eigenvalues of $n^{-1/2}\blmat{Y}{n}$.
Thus, it is sufficient to study the spectral radius of $n^{-1/2}\blmat{Y}{n}$.  By part  \ref{item:invertible}  of Theorem \ref{thm:isotropic}, we conclude that almost surely, 
\begin{equation*} 
\limsup_{n\rightarrow\infty}\rho\left(\frac{1}{\sqrt{n}}\blmat{Y}{n}\right)\leq 1
\end{equation*}
where $\rho(M)$ denotes the spectral radius of the matrix $M$.  This completes the proof of Theorem \ref{thm:nooutlier}.  

\subsection{Proof of Theorem~\ref{Thm:NoOutlierInProductPert}} 

	By rescaling by $\frac{1}{\sigma}$, it is sufficient to assume that $\sigma_{i}=1$ for $1\leq i\leq m$. Observe that if two deterministic terms 
	$$I+\iidmat{A}{n}{i}\text{ and }I+ \iidmat{A}{n}{j}$$
	appeared consecutively in the product $P_{n}$, then they could be rewritten
	$$\left(I+\iidmat{A}{n}{i}\right)\cdot\left(I+ \iidmat{A}{n}{j}\right)=I+\iidmat{A}{n}{j}'$$
	where all non-identity terms are lumped into the new deterministic matrix $\iidmat{A}{n}{j}'$ which still satisfies the assumptions on rank and norm. Additionally, if two random matrices $\iidmat{X}{n}{i}$ and $\iidmat{X}{n}{j}$ appeared in the product $P_{n}$ consecutively, then we could write 
	$$\left(\frac{1}{\sqrt{n}}\iidmat{X}{n}{i}\right)\left(\frac{1}{\sqrt{n}}\iidmat{X}{n}{j}\right)=\left(\frac{1}{\sqrt{n}}\iidmat{X}{n}{i}\right)(I_{n}+0_{n})\left(\frac{1}{\sqrt{n}}\iidmat{X}{n}{j}\right)$$
	where $0_{n}$ denotes the $n\times n$ zero matrix. Therefore, it is sufficient to consider products in which terms alternate between a random term $\frac{1}{\sqrt{n}}\iidmat{X}{n}{i}$ and a deterministic term $I+\iidmat{A}{n}{j}$. Next, observe that the eigenvalues of the product $P_{n}$ remain the same when the matrices in the product are cyclically permuted. Thus, without loss of generality and up to reordering the indices, we may assume that the product $P_{n}$ appears as
	\begin{equation}
	P_{n}= \left(I+\iidmat{A}{n}{1}\right)\left(\frac{1}{\sqrt{n}}\iidmat{X}{n}{1}\right)\left(I+\iidmat{A}{n}{2}\right)\left(\frac{1}{\sqrt{n}}\iidmat{X}{n}{2}\right)\cdots\left(I+\iidmat{A}{n}{m}\right)\left(\frac{1}{\sqrt{n}}\iidmat{X}{n}{m}\right).
	\label{equ:ProdAnyOrder}
	\end{equation} 
	Next, define the $2mn\times 2mn$ matrix 
	\begin{equation}
	\mathcal{L}_{n} :=\left[\begin{array}{cc}
	0_{mn} & I_{mn}+\diagA
	\\
	\frac{1}{\sqrt{n}}\blmat{Y}{n} & 0_{mn}\\ 
	\end{array} \right],
	\label{Def:Ln}
	\end{equation}
	where $0_{mn}$ denotes the $mn\times mn$ zero matrix, $\blmat{Y}{n}$ is as defined in \eqref{def:Y}, and $\diagA$ is defined so that \begin{equation*}
	I_{mn}+\diagA
	:=\left[\begin{array}{cccc}
	I+\iidmat{A}{n}{1} & 0 & \dots & 0\\
	0 & I+\iidmat{A}{n}{2} & \dots & 0\\
	\vdots & \vdots & \ddots & \vdots\\
	0 & 0 & \dots & I+\iidmat{A}{n}{m}\\
	\end{array}\right].
	\end{equation*}
	Here, we have to slightly adjust our notation to deal with the fact that $\blmat{L}{n}$ is a $2mn \times 2mn$ matrix instead of $mn \times mn$.  For the remainder of the proof, we view $\blmat{L}{n}$ as a $2\times 2$ matrix with entries which are $mn\times mn$ matrices, and we denote the $mn \times mn$ blocks as $\blmat{L}{n}^{[a,b]}$ for $1 \leq a,b \leq 2$.  We will use analogous notation for
	other $2mn \times 2mn$ matrices and $2mn$-vectors.  
	
	For $1 \le k \le 2m$, let $W_{k}$ denote the product $P_{n}$, but with terms cyclically permuted so that the product starts on the $k$th term. Note that there are $2m$ such products and each results in an $n\times n$ matrix. For instance, 
	$$W_{2}=\left(\frac{1}{\sqrt{n}}\iidmat{X}{n}{1}\right)(I+\iidmat{A}{n}{2})\left(\frac{1}{\sqrt{n}}\iidmat{X}{n}{2}\right)(I+\iidmat{A}{n}{3})\cdots\left(\frac{1}{\sqrt{n}}\iidmat{X}{n}{m}\right)(I+\iidmat{A}{n}{1})$$
	and 
	$$W_{3}=(I+\iidmat{A}{n}{2})\left(\frac{1}{\sqrt{n}}\iidmat{X}{n}{2}\right)(I+\iidmat{A}{n}{3})\cdots\left(\frac{1}{\sqrt{n}}\iidmat{X}{n}{m}\right)(I+\iidmat{A}{n}{1})\left(\frac{1}{\sqrt{n}}\iidmat{X}{n}{1}\right).$$
	A simple computation reveals that 
	\begin{equation*}
	\blmat{L}{n}^{2m}=\left[\begin{array}{cc}
	\mathcal{W} & 0_{mn}\\
	0_{mn} & \tilde{\mathcal{W}}\\
	\end{array}\right]
	\end{equation*}
	where 
	$$\mathcal{W}=\left[\begin{array}{cccc}
	W_{1} & 0_{n} & \dots & 0_{n}\\
	0_{n} & W_{3} & \dots & 0_{n}\\
	\vdots & \vdots & \ddots & \vdots\\
	0_{n} & 0_{n} & \dots & W_{2m-1}\\
	\end{array}\right]\;\;\;\;\text{and}\;\;\;\;\tilde{\mathcal{W}}=\left[\begin{array}{cccc}
	W_{2} & 0_{n} & \dots & 0_{n}\\
	0_{n} & W_{4} & \dots & 0_{n}\\
	\vdots & \vdots & \ddots & \vdots\\
	0_{n} & 0_{n} & \dots & W_{2m}\\
	\end{array}\right].$$
	Thus, the eigenvalues of $\blmat{L}{n}^{2m}$ are precisely the eigenvalues of the product $P_{n}$, each with multiplicity $2m$.

\newcommand\Acorner{\blmat{A}{n}^{\square}}
	Define
	\begin{equation*}
	\blmat{X}{n}:=\left[\begin{array}{cc}
	0_{mn} & I_{mn}\\
	\frac{1}{\sqrt{n}}\blmat{Y}{n} & 0_{mn}\\
	\end{array}\right]\;\;\;\;\text{ and }\;\;\;\;
	\Acorner:=\left[\begin{array}{cc}
	0_{mn} & \diagA\\
	0_{mn} & 0_{mn}\\
	\end{array}\right].
	\end{equation*}
	Then we can rewrite $\blmat{L}{n}=\blmat{X}{n}+\Acorner$. 
	
	For $1 \leq k \leq m$, let $Z_{k}:=\iidmat{X}{n}{k}\iidmat{X}{n}{k+1}\cdots\iidmat{X}{n}{m}\iidmat{X}{n}{1}\cdots\iidmat{X}{n}{k-1}$.  Then
	\begin{equation} \label{eq:XnZn}
	\blmat{X}{n}^{2m}=n^{-m/2}
	\left[\begin{array}{cccccccc}
	\mathcal{Z}_{n} & 0_{mn}\\
	0_{mn} & \mathcal{Z}_{n}\\
	\end{array}\right]
	\end{equation}
	where 
	$$\mathcal{Z}_{n}=\left[\begin{array}{ccc}
	Z_{1} & \dots & 0_{n}\\
	\vdots & \ddots & \vdots\\
	0_{n} & \dots & Z_{m}\\
	\end{array}\right].
	$$
	Thus, the eigenvalues of $\blmat{X}{n}^{2m}$ are precisely the eigenvalues of $n^{-m/2}\iidmat{X}{n}{1}\cdots\iidmat{X}{n}{m}$, each with multiplicity $2m$.
	
	
	Fix $\delta > 0$. By part \ref{item:invertible} of Theorem \ref{thm:isotropic} and Proposition \ref{prop:linear}, we find that almost surely, for $n$ sufficiently large, the eigenvalues of $n^{-m/2}X_{n,1}\cdots X_{n,m}$ are contained in the disk $\{ z \in \C : |z| \leq 1 + \delta\}$.   By \eqref{eq:XnZn}, we conclude that almost surely, for $n$ sufficiently large, the eigenvalues of $\blmat{X}{n}$ are contained in the disk $\{z \in \C : |z| \leq 1 + \delta\}$.  Therefore, almost surely, for $n$ sufficiently large $\blmat{X}{n} - zI$ is invertible for all $|z| > 1 + \delta$.  For all such values of $z$, we define
	\begin{equation*}
	\mathcal{R}_{n}(z):= \left(\blmat{X}{n}-zI_{2mn}\right)^{-1}.
	\end{equation*} 
	Since $\Acorner$ has rank $O(1)$ and operator norm $O(1)$, we can decompose (by the singular value decomposition) $\Acorner = \blmat{B}{n} \blmat{C}{n}$, where $\blmat{B}{n}$ is a $mn \times k$ matrix, $\blmat{C}{n}$ is a $k \times mn$ matrix, $k = O(1)$, and both $\blmat{B}{n}$ and $\blmat{C}{n}$ have rank $O(1)$ and operator norm $O(1)$.  
	
	Thus, for $|z| > 1 + \delta$, almost surely, for $n$ sufficiently large, 
	\begin{equation*}
	\det(\blmat{L}{n}-zI_{2mn})=\det(\blmat{X}{n}+\Acorner-zI_{2mn})=0
	\end{equation*}
	if and only if 
	\begin{equation}  \label{eq:proddecompcb}
	\det(I_{k}+\blmat{C}{n}\mathcal{R}_{n}(z)\blmat{B}{n})=0
	\end{equation}
	by Lemma \ref{lemma:eigenvalue}.
	
	Using Schur's Compliment to calculate the $mn\times mn$ blocks of $\mathcal{R}_n(z)$, we can see that 
$\mathcal{R}_n(z)=\begin{pmatrix}
z\mathcal{G}_{n}(z^{2}) &  \mathcal{G}_{n}(z^{2})\\ 
I+z^{2}\mathcal{G}_{n}(z^{2}) & z\mathcal{G}_{n}(z^{2})
\end{pmatrix}$
, where $\mathcal{G}_{n}(z):=\left(\frac{1}{\sqrt{n}}\mathcal{Y}_{n}-zI\right)^{-1}$ (which is defined for $|z| > 1 +\delta$ by Theorem~\ref{thm:isotropic}); hence
	\begin{equation*}
	\mathcal{R}_{n}(z)^{[a,b]}=\begin{cases}
	z\mathcal{G}_{n}(z^{2}) & \text{ if } a=b\\
	\mathcal{G}_{n}(z^{2}) & \text{ if } a=1,\;b=2\\
	I+z^{2}\mathcal{G}_{n}(z^{2}) & \text{ if } a=2,\;b=1.\\
	\end{cases}
	\end{equation*}
	Note that for $u=u_{2mn}$ and $v=v_{2mn}$ in $\C^{2mn}$, we have
	\begin{equation*}
	u^{*}\mathcal{R}_{n}(z)v = \sum_{1\leq a,b\leq 2}\left(u^{*}\right)^{[a]}\mathcal{R}_{n}(z)^{[a,b]}v^{[b]}
	\end{equation*}
	where $v^{[1]}, v^{[2]}$ denote the $mn\times 1$ sub-blocks of the vector $v$ and $\left(u^{*}\right)^{[1]}, \left(u^{*}\right)^{[2]}$ denote the $1 \times mn$ sub-blocks of the vector $u^*$. Additionally, if $u$ and $v$ have uniformly bounded norm for all $n$,  then by Theorem \ref{thm:isotropic}, almost surely
	\begin{align*}
	&\sup_{|z|>1+\delta}\left|\left(u^{*}\right)^{[1]}z\mathcal{G}_{n}(z^{2})v^{[1]}-z\left(-\frac{1}{z^{2}}\right)\left(u^{*}\right)^{[1]}v^{[1]}\right|=o(1),\\
	&\sup_{|z|>1+\delta}\left|\left(u^{*}\right)^{[2]}z\mathcal{G}_{n}(z^{2})v^{[2]}-z\left(-\frac{1}{z^{2}}\right)\left(u^{*}\right)^{[2]}v^{[2]}\right|=o(1),\\
	&\sup_{|z|>1+\delta}\left|(u^{*})^{[1]}\mathcal{G}_{n}(z^{2})v^{[2]}-\left(-\frac{1}{z^{2}}(u^{*})^{[1]}v^{[2]}\right)\right|=o(1), \text{ and }\\
	&\sup_{|z|>1+\delta}\left|(u^{*})^{[2]}(I+z^{2}\mathcal{G}_{n}(z^{2}))v^{[1]}-\left((u^{*})^{[2]}v^{[1]}-z^{2}\frac{1}{z^{2}}(u^{*})^{[2]}v^{[1]}\right)\right|=o(1).\\
	\end{align*}
	We note that the off-diagonal blocks of $\mathcal{R}_n(z)$ have a much different behavior than $\mathcal{G}_n(z)$. In order to keep track of this behavior, define 
	$$\mathcal{H}_{n}=\left[\begin{array}{cc}
	0_{mn} & I_{mn}\\
	0_{mn} & 0_{mn}\\ 
	\end{array}\right].$$
	Then, for deterministic $u$ and $v$ in $\C^{2mn}$ with uniformly bounded norm for all $n$, almost surely
	\begin{align*}
	\sup_{|z|>1+\delta}\left|u^{*}\mathcal{R}_{n}(z)v-\left(-\frac{1}{z}u^{*}v -\frac{1}{z^{2}}u^{*}\mathcal{H}_{n}v\right)\right|=o(1).
	\end{align*} 
	Applying this to \eqref{eq:proddecompcb}, we obtain that
	$$\sup_{|z|>1+\delta}\lnorm \blmat{C}{n}\mathcal{R}_{n}(z)\blmat{B}{n}-\left(-\frac{1}{z}\blmat{C}{n}\blmat{B}{n}-\frac{1}{z^{2}}\blmat{C}{n} \mathcal{H}_{n} \blmat{B}{n}\right)\rnorm=o(1)$$
	almost surely.  Hence, Lemma \ref{Lem:normtodet} reveals that almost surely
	$$\sup_{|z|>1+\delta}\left|\det\left(I-\mathcal{C}_{n}\mathcal{R}_{n}\mathcal{B}_{n}\right)-\det\left(I-\frac{1}{z}\blmat{C}{n}\blmat{B}{n}-\frac{1}{z^{2}}\blmat{C}{n}\mathcal{H}_{n}\blmat{B}{n}\right)\right|=o(1).$$
	By another application of Sylvester's determinant formula \eqref{eq:sylvester} and by noticing that $\mathcal{H}_{n}\Acorner$ is the zero matrix, this can be rewritten as
	$$\sup_{|z|>1+\delta}\left|\det\left(I+\mathcal{R}_{n}(z)\Acorner\right) - \det\left(I-\frac{1}{z}\Acorner\right)\right|=o(1).$$
	Finally, since the eigenvalues of $\Acorner$ are all zero, $\det\left(I-\frac{1}{z}\Acorner\right)=1$ so we can write the above statement as 
	$$\sup_{|z|>1+\delta}\left|\det \left( \mathcal R_n\right)\det\left(\mathcal{L}_{n}-zI\right) -1\right|=o(1)$$
	almost surely.  Almost surely, for $n$ sufficiently large, and for $|z|> 1+\delta$, we know that $\det(R_n(z))$ is finite, and thus	
	$\det\left(\mathcal{L}_{n}-zI\right)$ is nonzero for all $|z| > 1 + \delta$, implying that $\mathcal{L}_n$ has no eigenvalues outside the disk $\{z \in \C : |z| \leq 1 + \delta\}$.  By the previous observations, this implies the same conclusion for $P_n$ (since $\delta$ is arbitrary), and the proof is complete.

\subsection{Proof of Theorem \ref{thm:outliers}}

Recall that 
$$P_{n} :=\prod_{k=1}^{m}\left(\frac{1}{\sqrt{n}}\iidmat{X}{n}{k}+A_{n,k}\right), \qquad A_n := \prod_{k=1}^m A_{n,k}. $$ 
By rescaling $P_{n}$ by $\frac{1}{\sigma}$, we may assume that $\sigma_{i}=1$ for $1\leq i\leq n$. Let $\varepsilon>0$, and assume that for sufficiently large $n$, no eigenvalues of $A_{n}$ fall in the band $\{z\in\C\;:\;1+\varepsilon <|z|<1+3\varepsilon\}$. Assume that for some $j=O(1)$, there are $j$ eigenvalues $\lambda_{1}(A_{n}),\ldots,\lambda_{j}(A_{n})$ that lie in the region $\{z\in \C\;:\;|z|\geq 1+3\varepsilon\}$. 

Let $\mathcal{Y}_n$ and $\mathcal{A}_n$ be defined as in \eqref{def:Y} and \eqref{def:A}.  Using Proposition \ref{prop:linear}, it will suffice to study the eigenvalues of $\frac{1}{\sqrt{n}}\blmat{Y}{n}+\blmat{A}{n}$.  In particular, let $\eps' > 0$ such that $(1 + 2 \eps)^{1/m} = 1 + \eps'$.  Then we want to find solutions to
\begin{equation}
\det\left(\frac{1}{\sqrt{n}}\blmat{Y}{n}+\blmat{A}{n}-zI\right)=0
\label{Equ:detZero}
\end{equation}
for $|z|\geq 1+\eps'$. By part \ref{item:invertible} of Theorem \ref{thm:isotropic}, almost surely, for $n$ sufficiently large $\frac{1}{\sqrt{n}}\blmat{Y}{n} - zI$ is invertible for all $|z| \geq 1 + \eps'$.  By supposition, we can decompose (using the singular value decomposition) $\mathcal{A}_n = \mathcal{B}_n \mathcal{C}_n$, where $\mathcal{B}_n$ is $mn \times k$, $\mathcal{C}_n$ is $k \times mn$, $k = O(1)$, and both $\mathcal{B}_n$ and $\mathcal{C}_n$ have rank $O(1)$ and operator norm $O(1)$.  Thus, by Lemma \ref{lemma:eigenvalue}, we need to investigate the values of $z \in \C$ with $|z| \geq 1 + \eps'$ such that
\begin{equation*}
\det(I_{k}+\mathcal{C}_{n}\mathcal{G}_{n}(z)\mathcal{B}_{n})=0,  
\end{equation*}
where $\mathcal{G}_n(z)$ is defined in \eqref{def:G}.  

Since $k = O(1)$, Theorem \ref{thm:isotropic} implies that
\begin{equation}
\sup_{|z| \geq 1 + \eps'} \lnorm \mathcal{C}_{n}\mathcal{G}_{n}(z)\mathcal{B}_{n}-\left(-\frac{1}{z}\right)\mathcal{C}_{n}\mathcal{B}_{n}\rnorm \longrightarrow 0
\end{equation}
almost surely as $n\rightarrow\infty$. By Lemma \ref{Lem:normtodet}, this implies that 
\begin{equation*}
\sup_{|z| \geq 1 + \eps'} \left|\det\left(I_{k}+ \mathcal{C}_{n}\mathcal{G}_{n}(z)\mathcal{B}_{n}\right)-\det\left(I_{k}-\frac{1}{z}\mathcal{C}_{n}\mathcal{B}_{n}\right)\right|\longrightarrow 0
\end{equation*}
almost surely. By an application of Sylvester's determinant theorem \eqref{eq:sylvester}, this is equivalent to 
\begin{equation} \label{eq:rouchedet}
\sup_{|z| \geq 1 + \eps'} \left| \det\left(I_{k}+ \mathcal{C}_{n}\mathcal{G}_{n}\mathcal{B}_{n}\right) - \det\left(I_{n}-\frac{1}{z}\blmat{A}{n}\right)\right|\longrightarrow 0
\end{equation} 
almost surely as $n\rightarrow\infty$. Define
$$ g(z) := \det\left(I_{n}-\frac{1}{z}\blmat{A}{n}\right)=\prod_{i=1}^{k}\left(1-\frac{\lambda_{i}(\blmat{A}{n})}{z}\right). $$
Since the eigenvalues of $\blmat{A}{n}^m$ are precisely the eigenvalues of $A_n$, each with multiplicity $m$, it follows that $g$ has precisely $l := jm$ roots $\lambda_1(\blmat{A}{n}), \ldots, \lambda_l(\blmat{A}{n})$ outside the disk $\{ z \in \C : |z| < 1 + \eps'\}$.  Thus, by \eqref{eq:rouchedet} and Rouch\'{e}'s theorem, almost surely, for $n$ sufficiently large, 
\[ f(z):=\det\left(I_{k}+ \mathcal{C}_{n}\mathcal{G}_{n}\mathcal{B}_{n}\right) \]
has exactly $l$ roots outside the disk $\{ z \in \C : |z| < 1 + \eps' \}$ and these roots take the values $\lambda_i(\blmat{A}{n}) + o(1)$ for $1 \leq i \leq l$.  

Returning to \eqref{Equ:detZero}, we conclude that almost surely for $n$ sufficiently large, $\frac{1}{\sqrt{n}} \blmat{Y}{n} + \blmat{A}{n}$ has exactly $l$ roots outside the disk $\{z \in \mathbb{C} : |z| < 1 + \eps'\}$, and after possibly reordering the eigenvalues, these roots take the values
\[ \lambda_i \left( \frac{1}{\sqrt{n}} \blmat{Y}{n} + \blmat{A}{n} \right) = \lambda_i ( \blmat{A}{n} ) + o(1) \]
for $1 \leq i \leq l$.

We now relate these eigenvalues back to the eigenvalues of $P_n$.  Recall that $\left( \frac{1}{\sqrt{n}} \blmat{Y}{n} + \blmat{A}{n} \right)^m$ has the same eigenvalues as $P_n$, each with multiplicity $m$; and $\blmat{A}{n}^m$ has the same eigenvalues of $A_n$, each with multiplicity $m$.  Taking this additional multiplicity into account and using the fact that
\[ \left( \lambda_i( \blmat{A}{n} ) + o(1) \right)^m = \lambda_i( \blmat{A}{n}^m ) + o(1) \]
since $\blmat{A}{n}$ has spectral norm $O(1)$, we conclude that almost surely, for $n$ sufficiently large, $P_n$ has exactly $j$ eigenvalues in the region $\{ z \in \mathbb{C} : |z| \geq 1 + 2 \eps \}$, and after reordering the indices correctly 
\[ \lambda_i(P_n) = \lambda_i(A_n) + o(1) \]
for $1 \leq i \leq j$.  This completes the proof of Theorem \ref{thm:outliers}.

\subsection{Proof of Theorems \ref{thm:nomixed} and \ref{thm:nonzeromean}}
In the proofs of Theorems \ref{thm:nomixed} and \ref{thm:nonzeromean} we will make use of an alternative isotropic law which is a corollary of Theorem \ref{thm:isotropic}. We state and prove the result now.

\begin{corollary}[Alternative Isotropic Law] \label{cor:ProductIsotropic}
	Let $m \geq 1$ be a fixed integer, and assume $\xi_1, \ldots, \xi_m$ are complex-valued random variables with mean zero, unit variance, finite fourth moments, and independent real and imaginary parts.  For each $n \geq 1$, let $X_{n,1}, \ldots, X_{n,m}$ be independent $n \times n$ iid random matrices with atom variables $\xi_1, \ldots, \xi_m$, respectively. Then, for any fixed $\delta > 0$, the following statements hold.  
	\begin{enumerate}[label=(\roman*)]
		\item \label{item:alt:invertible} Almost surely, for $n$ sufficiently large, the eigenvalues of $n^{-m/2}\iidmat{X}{n}{1}\cdots\iidmat{X}{n}{m}$ are contained in the disk $\{z \in \mathbb{C} : |z| \leq 1 + \delta \}$.  In particular, this implies that almost surely, for $n$ sufficiently large, the matrix\\ $n^{-m/2}\iidmat{X}{n}{1}\cdots\iidmat{X}{n}{m}- z I$ is invertible for every $z \in \mathbb{C}$ with $|z| > 1 + \delta$.
		\item There exists a constant $c > 0$ (depending only on $\delta$) such that almost surely, for $n$ sufficiently large, 
		$$ \sup_{z \in \mathbb{C} : |z| > 1 + \delta} \lnorm \left(n^{-m/2}\iidmat{X}{n}{1}\cdots\iidmat{X}{n}{m}-zI\right)^{-1} \rnorm \leq c. $$
		\item For each $n \geq 1$, let $u_n, v_n \in \mathbb{C}^{n}$ be deterministic unit vectors.  Then 
		$$ \sup_{z \in \mathbb{C} : |z| > 1 + \delta} \left| u_n^\ast \left(n^{-m/2}\iidmat{X}{n}{1}\cdots\iidmat{X}{n}{m}-zI\right)^{-1} v_n + \frac{1}{z} u_n^\ast v_n \right| \longrightarrow 0 $$
		almost surely as $n \to \infty$. 
	\end{enumerate}
\end{corollary}

\begin{proof}
Part \ref{item:alt:invertible} follows from Theorem \ref{thm:nooutlier}.  Let $\blmat{Y}{n}$ be defined by \eqref{def:Y}, and let $\blmat{G}{n}(z)$ be defined by \eqref{def:G}.  Then the last two parts of Corollary \ref{cor:ProductIsotropic} follow from the last two parts of Theorem \ref{thm:isotropic} due to the fact that
\[ \mathcal{G}_n^{[1,1]}(z) = z^{m-1} ( n^{-m/2} X_{n,1} \cdots X_{n,m} - z^m I)^{-1}. \] 
\end{proof}

With this result in hand, we proceed to the remainder of the proofs.

\begin{proof}[Proof of Theorem \ref{thm:nomixed}.]
By rescaling by $\frac{1}{\sigma}$ it suffices to assume that $\sigma_i = 1$ for $1 \leq i \leq m$. Let $\varepsilon >0$. Assume that for sufficiently large $n$ there are no eigenvalues of $A_{n}$ in the band $\{z\in\C\;:\;1+\varepsilon < |z| < 1+3\varepsilon\}$ and there are $j$ eigenvalues, $\lambda_{1}(A_{n})$,\ldots,$\lambda_{j}(A_{n})$, in the region $\{z\in\C\;:\;|z|\geq 1+3\varepsilon\}$. 

By Corollary \ref{cor:ProductIsotropic}, almost surely, for $n$ sufficiently large, $n^{-m/2}\iidmat{X}{n}{1}\cdots\iidmat{X}{n}{m} - zI$ is invertible for all $|z| > 1 + \eps$. We decompose (using the singular value decomposition) $A_n = B_n C_n$, where $B_n$ is $n \times k$, $C_n$ is $k \times n$, $k = O(1)$, and both $B_n$ and $C_n$ have rank $O(1)$ and spectral norm $O(1)$.  By Lemma \ref{lemma:eigenvalue}, the eigenvalues of $P_n$ outside $\{ z \in \C : |z| < 1 + 2 \eps \}$ are precisely the values of $z \in \mathbb{C}$ with $|z| > 1 + 2 \eps$ such that
\[ \det\left(I_{k}+C_{n}\left(n^{-m/2}\iidmat{X}{n}{1}\cdots\iidmat{X}{n}{m}-zI_{n}\right)^{-1}B_{n}\right) = 0. \]

By Corollary \ref{cor:ProductIsotropic},
\begin{equation*}
\sup_{|z|>1+\eps}\lnorm C_{n}\left(n^{-m/2}\iidmat{X}{n}{1}\cdots\iidmat{X}{n}{m}-zI_{n}\right)^{-1}B_{n}+\frac{1}{z}C_{n}B_{n}\rnorm\longrightarrow 0
\end{equation*}
almost surely as $n\rightarrow\infty$. By applying \eqref{eq:sylvester} and Lemma \ref{Lem:normtodet}, this gives 
\begin{equation*}
\sup_{|z|>1+\eps}\left|\det\left(I_{k}+C_{n}\left(n^{-m/2}\iidmat{X}{n}{1}\cdots\iidmat{X}{n}{m}-zI_{n}\right)^{-1}B_{n}\right) - \det\left(I_{n}-\frac{1}{z}A_{n}\right)\right|\longrightarrow 0
\end{equation*}
almost surely as $n\rightarrow\infty$.  By recognizing the roots of $\det\left(I_{n}-\frac{1}{z}A_{n}\right)$ as the eigenvalues of $A_{n}$, Rouch\'{e}'s Theorem implies that almost surely for $n$ sufficiently large $P_n$ has exactly $j$ eigenvalues in the region $\{z \in \C : |z| > 1 + 2 \eps \}$, and after labeling the eigenvalues properly, 
\begin{equation*}
\lambda_{i}(P_{n})=\lambda_{i}(A_{n})+o(1)
\end{equation*}
for $1 \leq i \leq j$.  
\end{proof}

The proof of Theorem \ref{thm:nonzeromean} will require the following corollary of \cite[Lemma 2.3]{Tout}.  

\begin{lemma} \label{Lem:ProductToZero}
Let $\varphi_{n}$ and $X_{n,1}$,\ldots,$X_{n,m}$ be as in Theorem \ref{thm:nonzeromean}. Then almost surely, 
	$$n^{-m/2} \left| \varphi_{n}^{*}X_{n,1}X_{n,2}\cdots X_{n,m}\varphi_{n} \right| = o(1).$$
\end{lemma}
\begin{proof}
Let $u := n^{-(m-1)/2}\varphi_{n}^{*}X_{n,1}X_{n,2}\cdots X_{n,m-1}$.  In view of \cite[Theorem 1.4]{Tout}, it follows that $\|u \| = O(1)$ almost surely.  We now condition on $X_{n,1}, \ldots, X_{n,m-1}$ so that $\| u \| = O(1)$.  As $X_{n,m}$ is independent of $u$, we apply \cite[Lemma 2.3]{Tout} to conclude that 
$$ u \left(\frac{1}{\sqrt{n}}X_{n,m}\right)\varphi_{n}=o(1) $$
almost surely, concluding the proof. 
\end{proof}

With this result, we may proceed to the proof of Theorem \ref{thm:nonzeromean}.

\begin{proof}[Proof of Theorem \ref{thm:nonzeromean}]
By rescaling by $\frac{1}{\sigma}$ it suffices to assume that $\sigma_i = 1$ for $1 \leq i \leq m$. Fix $\gamma > 0$ and let $\varepsilon>0$. By Corollary \ref{cor:ProductIsotropic} and Lemma \ref{lemma:eigenvalue}, almost surely, for $n$ sufficiently large, the only eigenvalues of 
	\begin{equation*}
	P_{n}:= n^{-m/2}\iidmat{X}{n}{1}\cdots\iidmat{X}{n}{m}+\mu n^{\gamma}\phi_{n}\phi_{n}^{*}
	\end{equation*}
	in the region $\{z \in \C : |z| > 1 + \eps \}$ are the values of $z \in \C$ with $|z| > 1 + \eps$ such that 
	\begin{equation}
	1+\mu n^{\gamma} \phi^{*}_{n}\left(n^{-m/2}\iidmat{X}{n}{1}\cdots\iidmat{X}{n}{m}-zI_{n}\right)^{-1}\phi_{n}=0.
	\label{equ:detInNonZeroMean}
	\end{equation}

	Define the functions
	\begin{align*}
	&f(z):=1+\mu n^{\gamma}\phi^{*}_{n}\left(n^{-m/2}\iidmat{X}{n}{1}\cdots\iidmat{X}{n}{m}-zI_{n}\right)^{-1}\phi_{n}\\
	&g(z):=1-\frac{\mu n^{\gamma}}{z}.
	\end{align*}
	Observe that $g(z)$ has one zero located at $\mu n^{\gamma}$ which will be outside the disk $\{z\in \C\;:\;|z| \leq 1+\varepsilon\}$ for large enough $n$. By Corollary \ref{cor:ProductIsotropic}, it follows that almost surely
	\begin{equation*}
	\sup_{|z|>1+\eps}|f(z)-g(z)| = o(n^{\gamma}).
	\end{equation*}
	Thus, almost surely 
	\begin{align*}
	f(z)=g(z)+o(n^{\gamma})
	\end{align*}
	for all $z \in \C$ with $|z| > 1+\eps$. 
	
	Observe that if $z$ is a root of $f$ with $|z| > 1 + \eps$, then $|g(z)|=\left|1-\frac{\mu n^{\gamma}}{z}\right|=o(n^{\gamma})$. We conclude that if $z$ is a root of $f$ outside the disk $\{z \in \C : |z| > 1 + \eps \}$, then the root must tend to infinity with $n$ almost surely.  We will return to this fact shortly.  
	
	For the next step of the proof, we will need to bound the spectral norm of $n^{-m/2} X_{n,1} \cdots X_{n,m}$.  To do so, we apply \cite[Theorem 1.4]{Tout} and obtain that, almost surely, for $n$ sufficiently large, 
	\begin{equation} \label{eq:tout:normbound}
		n^{-m/2} \| X_{n,1} \cdots X_{n,m} \| \leq (2.1)^m.
	\end{equation} 
	Thus, by a Neumann series expansion, for all $|z| > (2.5)^m$, we have
	\begin{align*}
	f(z)&=1-\frac{\mu n^{\gamma}}{z}+\frac{\mu n^{\gamma}}{z^{2}}\phi^{*}_{n}\left(n^{-m/2}\iidmat{X}{n}{1}\cdots\iidmat{X}{n}{m}\right)\phi_{n}+O\left(\frac{ n^{\gamma}}{|z|^{3}}\right)\\
	&= g(z)+\frac{\mu n^{\gamma}}{z^{2}}\phi^{*}_{n}\left(n^{-m/2}\iidmat{X}{n}{1}\cdots\iidmat{X}{n}{m}\right)\phi_{n}+O\left(\frac{n^{\gamma}}{|z|^{3}}\right).
	\end{align*}
	By Lemma \ref{Lem:ProductToZero},  $\varphi^{*}_{n}\left(n^{-m/2}\iidmat{X}{n}{1}\cdots\iidmat{X}{n}{m}\right)\varphi_{n}=o(1)$ almost surely, so one can see
	\begin{equation*}
	f(z)=g(z)+o\left(\frac{n^{\gamma}}{|z|^{2}}\right)+O\left(\frac{n^{\gamma}}{|z|^{3}}\right)
	\end{equation*}
	uniformly for all $|z|>(2.5)^{m}$.  In particular, when $|z| \to \infty$, we obtain
	\begin{equation} \label{eq:fztendinfty}
	|z| |f(z) - g(z)| = o \left( \frac{n^\gamma}{|z|} \right)
	\end{equation}
	almost surely.  Since any root of $zf(z)$ outside $\{z \in \C : |z| \leq 1 + \eps \}$ must tend to infinity with $n$, it follows from Rouch\'{e}'s theorem that almost surely, for $n$ sufficiently large, $zf(z)$ has precisely one root outside the disk $\{ z \in \C : |z| \leq 1 + \eps\}$ and that root takes the value $\mu n^\gamma + o(n^\gamma)$.  
	
	It remains to reduce the error from $o(n^\gamma)$ to $o(1)$.  Fix $\delta > 0$, and let $\Gamma$ be the circle around $\mu n^\gamma$ with radius $\delta$.  Then from \eqref{eq:fztendinfty} we see that almost surely
	\[ \sup_{z \in \Gamma} |z||f(z) - g(z)| = o(1). \]
	Hence, almost surely, for $n$ sufficiently large,
	\[ |z||f(z) - g(z)| < \delta = |z||g(z)| \]
	for all $z \in \Gamma$.  Therefore, by another application of Rouch\'{e}'s theorem, we conclude that almost surely, for $n$ sufficiently large, $zf(z)$ contains precisely one root outside of $\{z \in \C : |z| \leq 1 + \eps \}$ and that root is located in the interior of $\Gamma$.  Since $\delta$ was arbitrary, this completes the proof.  
\end{proof}

\section{Truncation and useful tools}
\label{Sec:TruncAndTools}

We now turn to the proof of Theorem \ref{thm:isotropic}.   We will require the following standard truncation results for iid random matrices.  

\begin{lemma}
	Let $\xi$ be a complex-valued random variable with mean zero, unit variance, finite fourth moment, and independent real and imaginary parts. Let $\Re(\xi)$ and $\Im(\xi)$ denote the real and imaginary parts of $\xi$ respectively, and let $\sqrt{-1}$ denote the imaginary unit. For $L > 0$, define 
	\begin{align*}
	\tilde{\xi}&:=\Re(\xi)\indicator{|\Re(\xi)|\leq L/\sqrt{2}}-\E\left[\Re(\xi)\indicator{|\Re(\xi)|\leq L/\sqrt{2}}\right]\\
	&\quad \quad \quad +\sqrt{-1}\left(\Im(\xi)\indicator{|\Im(\xi)|\leq L/\sqrt{2}}-\E\left[\Im(\xi)\indicator{|\Im(\xi)|\leq L/\sqrt{2}}\right]\right)
	\end{align*}
	and
	\begin{equation*}
		\hat{\xi}:=\frac{\tilde{\xi}}{\sqrt{\emph{Var}(\tilde{\xi})}}.
	\end{equation*}
	Then there exists a constant $L_0 > 0$ (depending only on $\E|\xi|^{4}$) such that the following statements hold for all $L > L_0$.  
	\begin{enumerate}[label=\emph{(\roman*)}]
		\item \label{item:truncation:i} $\emph{Var}(\tilde{\xi})\geq \frac{1}{2}$
		\item \label{item:truncation:ii} $|1-\emph{Var}(\tilde{\xi})|\leq\frac{4}{L^2}\E|\xi|^{4}$
		\item \label{item:truncation:iii} Almost surely, $|\hat{\xi}|\leq 4L$
		\item \label{item:truncation:iv} $\hat{\xi}$ has mean zero, unit variance, $\E|\hat{\xi}|^{4}\leq C\E|\xi|^{4}$ for some absolute constant $C>0$, and the real and imaginary parts of $\hat{\xi}$ are independent.  
	\end{enumerate}
	\label{lem:Truncate}
\end{lemma}
The proof of this theorem is a standard truncation argument. The full details of the proof can be found in Appendix \ref{Sec:ProofOfTruncation}.

Let $X$ be an $n\times n$ random matrix filled with iid copies of a random variable $\xi$ which has mean zero, unit variance, finite fourth moment, and independent real and imaginary parts.  We define matrices $\tilde{X}$ and $\hat{X}$ to be the $n\times n$ matrices with entries defined by
\begin{align}
\tilde{X}_{(i,j)}&:=\Re(X_{(i,j)})\indicator{|\Re(X_{(i,j)})|\leq L/\sqrt{2}}-\E\left[\Re(X_{(i,j)})\indicator{|\Re(X_{(i,j)})|\leq L/\sqrt{2}}\right]\notag\\
& \quad  +\sqrt{-1}\left(\Im(X_{(i,j)})\indicator{|\Im(X_{(i,j)})|\leq L/\sqrt{2}}-\E\left[\Im(X_{(i,j)})\indicator{|\Im(X_{(i,j)})|\leq L/\sqrt{2}}\right]\right)\label{def:Xtilde}
\end{align}
and
\begin{equation}
\hat{X}_{(i,j)}:= \frac{\tilde{X}_{(i,j)}}{\sqrt{\Var(\tilde{X}_{(i,j)})}}
\label{def:Xhat}
\end{equation}
for $1 \leq i,j \leq n$.  

\begin{lemma}
Let $X$ be an iid random matrix with atom variable $\xi$ which has mean zero, unit variance,  $m_{4}:=\E|\xi|^4< \infty$, and independent real and imaginary parts.  Let $\hat{X}$ be as defined in \eqref{def:Xhat}. Then, there exist constants $C, L_0 > 0$ (depending only on $m_4$) such that for all $L > L_0$ 
\begin{equation*}
\limsup_{n\rightarrow \infty} \frac{1}{\sqrt n}  \lnorm X-\hat{X}\rnorm \leq \frac{C}{L}
\end{equation*}
almost surely.  
\label{lem:XcloseXhat}
\end{lemma}

\begin{proof}
	Let $\tilde{X}$ be defined as in \eqref{def:Xtilde}, and let $L_0$ be the value from Lemma \ref{lem:Truncate}.  
	Begin by noting that 
	$$\lnorm X-\hat{X}\rnorm \leq \lnorm X-\tilde{X}\rnorm+\lnorm \tilde{X}-\hat{X}\rnorm$$
	and thus it suffices to show that
	$$\underset{n\rightarrow\infty}{\limsup}\frac{1}{\sqrt{n}}\lnorm X-\tilde{X}\rnorm \leq \frac{C}{L}$$
	and 
	$$\underset{n\rightarrow\infty}{\limsup}\frac{1}{\sqrt{n}}\lnorm \tilde{X}-\hat{X}\rnorm \leq \frac{C}{L}$$
	almost surely. First, by Lemma \ref{lem:Truncate},
	\begin{align}
	\frac{1}{\sqrt{n}}\lnorm \tilde{X}-\hat{X}\rnorm &=
	\frac{1}{\sqrt{n}}\lnorm \tilde{X}\rnorm\left|1-\frac{1}{\sqrt{\Var(\tilde{\xi})}}\right|\notag\\
	&\leq \frac{1}{\sqrt{n}}\lnorm \tilde{X}\rnorm\sqrt{2}\left|\sqrt{\Var(\tilde{\xi})}-1\right|\notag\\
	&\leq \frac{1}{\sqrt{n}}\lnorm \tilde{X}\rnorm\sqrt{2}\left|\Var(\tilde{\xi})-1\right|\notag\\
	&\leq \frac{1}{\sqrt{n}}\lnorm \tilde{X}\rnorm\sqrt{2}\left(\frac{4}{L^2}\E|\xi|^{4}\right).
	\label{Equ:XsCloseStep}
	\end{align}
	By \cite[Theorem 1.4]{Tout}, we find that almost surely
	\begin{equation*}
	\limsup_{n\rightarrow\infty}\frac{1}{\sqrt{n}}\lnorm \tilde{X}\rnorm \leq 2,
	\end{equation*}
	and thus by \eqref{Equ:XsCloseStep}
	\begin{equation*}
	\underset{n\rightarrow\infty}{\limsup}\frac{1}{\sqrt{n}}\lnorm \tilde{X}-\hat{X}\rnorm \leq\frac{C}{L}
	\end{equation*}
	almost surely for all $L \geq \max\{1, L_0\}$.
	
	Next consider $\underset{n\rightarrow\infty}{\limsup}\frac{1}{\sqrt{n}}\lnorm X-\tilde{X}\rnorm$.  Note that $X - \tilde{X}$ is an iid matrix with atom variable 
	\begin{align*}\Re(X_{(i,j)})&\indicator{|\Re(X_{(i,j)})|> L/\sqrt{2}}-\E\left[\Re(X_{(i,j)})\indicator{|\Re(X_{(i,j)})|> L/\sqrt{2}}\right]\\
	&+\sqrt{-1}\left(\Im(X_{(i,j)})\indicator{|\Im(X_{(i,j)})|> L/\sqrt{2}}-\E\left[\Im(X_{(i,j)})\indicator{|\Im(X_{(i,j)})|> L/\sqrt{2}}\right]\right).
	\end{align*}  
	Thus, each entry has mean zero, variance
	\begin{equation*}
	\Var((X-\tilde{X})_{i,j}) \leq \frac{8}{L^{2}}\E|\xi|^{4},
	\end{equation*}
	and finite fourth moment.  Thus, again by \cite[Theorem 1.4]{Tout},  
	\begin{equation*}
	\underset{n\rightarrow\infty}{\limsup}\frac{1}{\sqrt{n}}\lnorm X-\tilde{X}\rnorm \leq \frac{C}{L}
	\end{equation*}
	almost surely, and the proof is complete.  
\end{proof}

We now consider the iid random matrices $X_{n,1}, \ldots, X_{n,m}$ from Theorem \ref{thm:isotropic}.  For each $1\leq k\leq m$, define the truncation $\triidmat{X}{n}{k}$ as was done above for $\hat{X}$ in \eqref{def:Xhat}. Define $\trblmat{Y}{n}$ by
\begin{equation}
\trblmat{Y}{n}:=\left[\begin{matrix}
0 & \triidmat{X}{n}{1} & \dots & 0\\
\vdots & \vdots & \ddots & \vdots \\
0 & 0 & \dots & \triidmat{X}{n}{m-1}\\
\triidmat{X}{n}{m} & 0 & \dots & 0\\
\end{matrix}\right].
\label{Equ:trBlockMat}
\end{equation}
Using Theorem \ref{lem:Truncate}, we have the following corollary.  

\begin{corollary}
	Let $X_{n,1},\ldots,X_{n,m}$ be independent iid random matrices with atom variables $\xi_{1},\ldots,\xi_{m}$, each of which has mean zero, unit variance, finite fourth moment, and independent real and imaginary parts.  Let $\hat{X}_{n,1},\ldots,\hat{X}_{n,m}$ be the truncations of $X_{n,1}, \ldots, X_{n,m}$ as was done in \eqref{def:Xhat}.  In addition, let $\blmat{Y}{n}$ be as defined in \eqref{def:Y} and $\trblmat{Y}{n}$ be as defined in \eqref{Equ:trBlockMat}. Then there exist constants $C, L_0 > 0$ (depending only on the atom variables $\xi_1, \ldots, \xi_m$) such that
	\begin{equation}
	\limsup_{n\rightarrow\infty}\frac{1}{\sqrt{{n}}}\lnorm \blmat{Y}{n}-\trblmat{Y}{n}\rnorm \leq \frac{C}{L}
	\end{equation}
	almost surely for all $L > L_0$. 
	\label{Cor:YsClose}
\end{corollary}	
\begin{proof}
Due to the block structure of $\blmat{Y}{n}$ and $\trblmat{Y}{n}$, it follows that 
\begin{equation*}
\lnorm \blmat{Y}{n}-\trblmat{Y}{n}\rnorm \leq \max_{k}\lnorm \iidmat{X}{n}{k}-\triidmat{X}{n}{k}\rnorm .
\end{equation*}
Therefore, the claim follows from Lemma \ref{lem:XcloseXhat}.  
\end{proof}

\section{Least singular value bounds}
\label{Sec:LeastSingVal}

In this section, we study the least singular value of $\frac{1}{\sqrt{n}}\blmat{Y}{n}-zI$.  We begin by recalling Weyl's inequality for the singular values (see, for example, \cite[Problem III.6.5]{Bhatia}), which states that for $n\times n$ matrices $A$ and $B$, 
\begin{equation} \label{eq:weyl}
\max_{1 \leq i \leq n}\left|s_{i}(A)-s_{i}(B)\right|\leq \lnorm A-B\rnorm.
\end{equation}


We require the following theorem, which is based on \cite[Theorem 2]{N}.

\begin{theorem} \label{Thm:LeastTruncSingValNonZero} 
Fix $L > 0$, and let $\xi_1, \ldots, \xi_m$ be complex-valued random variables, each having mean zero, unit variance, independent real and imaginary parts, and which satisfy 
\[	\sup_{1 \leq k \leq m} |\xi_k| \leq L \]
almost surely.  Let $X_{n,1}, \ldots, X_{n,m}$ be independent iid random matrices with atom variables $\xi_1, \ldots, \xi_m$, respectively.  Define $\mathcal{Y}_n$ as in \eqref{def:Y}, and fix $\delta > 0$.  Then there exists a constant $c > 0$ (depending only on $\delta$) such that 
\begin{equation*}
	 \inf_{|z|\geq 1+\delta}s_{mn}\left(\frac{1}{\sqrt{n}}\blmat{Y}{n}-zI\right)\geq c
 \end{equation*} 
with overwhelming probability.  
\end{theorem}

A similar statement was proven in \cite[Theorem 2]{N}, where the same conclusion was shown to hold almost surely rather than with overwhelming probability.  However, many of the intermediate steps in \cite{N} are proven to hold with overwhelming probability.  We use these intermediate steps to prove Theorem \ref{Thm:LeastTruncSingValNonZero} in Appendix \ref{sec:singoutlier}.

\begin{lemma} \label{Lem:LeastSingValAwayFromZero}
	Let $X_{n,1},\ldots,X_{n,m}$ satisfy the assumptions of Theorem \ref{thm:isotropic}, and let $\blmat{Y}{n}$ be as defined in \eqref{def:Y}. Fix $\delta > 0$.  Then there exists a constant $c>0$ such that almost surely, for $n$ sufficiently large,
	\begin{equation*}
	\inf_{|z|\geq 1+\delta}s_{mn}\left(\frac{1}{\sqrt{n}}\blmat{Y}{n}-zI\right)\geq c.
	\end{equation*}
\end{lemma}

\begin{proof}
	Let $L > 0$ be a large constant to be chosen later.  Let $\hat{X}_{n,1}, \ldots, \hat{X}_{n,m}$ be defined as in \eqref{def:Xhat}, and let $\hat{\mathcal{Y}}_n$ be defined as in \eqref{Equ:trBlockMat}.  By Theorem \ref{Thm:LeastTruncSingValNonZero} and the Borel--Cantelli lemma, there exists a constant $c'>0$ such that almost surely, for $n$ sufficiently large, 
	\begin{equation*}
	\inf_{|z|\geq 1+\delta}s_{mn}\left(\frac{1}{\sqrt{n}}\trblmat{Y}{n}-zI\right)\geq c'.
	\end{equation*}
	By Corollary \ref{Cor:YsClose} and 
	\eqref{eq:weyl} we may choose $L$ sufficiently large to ensure that almost surely, for $n$ sufficiently large, 
	\begin{equation*}
	\left|s_{mn}\left(\frac{1}{\sqrt{n}}\blmat{Y}{n}-zI\right)-s_{mn}\left(\frac{1}{\sqrt{n}}\trblmat{Y}{n}-zI\right)\right|\leq \frac{c'}{2},
	\end{equation*}
	uniformly in $z$. We conclude that almost surely, for $n$ sufficiently large,  
	\begin{equation*}
	\inf_{|z|\geq 1+\delta}s_{mn}\left(\frac{1}{\sqrt{n}}\blmat{Y}{n}-zI\right)\geq \frac{c'}{2}, 
	\end{equation*}	
	and the proof is complete.  
\end{proof}

With this result we can prove the following lemma.

\begin{lemma} \label{Cor:LeastSingValOfProductAwayFromZero}
Let $X_{n,1},\ldots,X_{n,m}$ satisfy the assumptions of Theorem \ref{thm:isotropic}, and fix $\delta > 0$.  Then there exists a constant $c>0$ such that almost surely, for $n$ sufficiently large, 
\begin{equation*}
	\inf_{|z|\geq 1+\delta}s_{mn}\left(n^{-m/2}\iidmat{X}{n}{1}\cdots\iidmat{X}{n}{m}-zI\right)\geq c.  
	\end{equation*}
\end{lemma}
\begin{proof}
Let $\mathcal{Y}_n$ be defined as in \eqref{def:Y}.  Then Lemma \ref{Lem:LeastSingValAwayFromZero} implies that almost surely, for $n$ sufficiently large, $\frac{1}{\sqrt{n}} \mathcal{Y}_n - z I$ is invertible for all $|z| \geq 1 + \delta$.  By computing the block inverse of this matrix, we find
\[ \left(\left(\frac{1}{\sqrt{n}}\blmat{Y}{n}-zI\right)^{-1}\right)^{[1,1]}=z^{m-1}\left(n^{-m/2}\iidmat{X}{n}{1}\cdots\iidmat{X}{n}{m}-z^{m}I\right)^{-1}. \]
Thus, for $|z| \geq 1 + \delta$, 
\begin{align*}
	\left\| \left(n^{-m/2}\iidmat{X}{n}{1}\cdots\iidmat{X}{n}{m}-z^{m}I\right)^{-1} \right\| &\leq |z|^{m-1} \left\| \left(n^{-m/2}\iidmat{X}{n}{1}\cdots\iidmat{X}{n}{m}-z^{m}I\right)^{-1} \right\| \\
	&\leq \left\| \left( \frac{1}{\sqrt{n}} \mathcal{Y}_n - z I \right)^{-1} \right \|, 
\end{align*}
where the last step used the fact that the operator norm of a matrix bounds above the operator norm of any sub-matrix.  

Recall that if $M$ is an invertible $N \times N$ matrix, then $s_N(M) = \| M^{-1} \|^{-1}$.  Applying this fact to the matrices above, we conclude that
\[ s_{mn}\left(n^{-m/2}\iidmat{X}{n}{1}\cdots\iidmat{X}{n}{m}-z^{m}I\right) \geq s_{mn}\left( \frac{1}{\sqrt{n}} \mathcal{Y}_n - z I \right), \]
and the claim follows from Lemma \ref{Lem:LeastSingValAwayFromZero}.  
\end{proof}

\begin{remark}
Observe that $z$ is an eigenvalue of $\frac{1}{\sqrt{n}}\blmat{Y}{n}$ if and only if $\det\left(\frac{1}{\sqrt{n}}\blmat{Y}{n}-zI\right)=0$.  Also, recall that $\left|\det\left(\frac{1}{\sqrt{n}}\blmat{Y}{n}-zI\right)\right|=\prod_i s_{i}\left(\frac{1}{\sqrt{n}}\blmat{Y}{n}-zI\right)$; this product is zero if and only if the smallest singular value is zero. Since Lemma \ref{Lem:LeastSingValAwayFromZero} and Lemma \ref{Cor:LeastSingValOfProductAwayFromZero} bound the least singular values of $\frac{1}{\sqrt{n}} \mathcal{Y}_n - z I$ and $n^{-m/2}X_{n,1} \cdots X_{n,m} - zI$ away from zero for $|z| \geq 1 + \delta$, we can conclude that such values of $z$ are almost surely, for $n$ sufficiently large, not eigenvalues of these matrices.  
\label{Rem:ZnotEval}
\end{remark}

\section{Reductions to the proof of Theorem \ref{thm:isotropic}}
\label{Sec:Reductions}

In this section, we will prove that it is sufficient to reduce the proof of Theorem \ref{thm:isotropic} to the case in which the entries of each matrix are truncated and where we restrict $z$ to the band $5\leq |z|\leq 6$.  

\begin{theorem}  
	Let $X_{n,1}, \ldots, X_{n,m}$ be as in Theorem \ref{thm:isotropic}.  Then there exists a constant $L_0$ such that the following holds for all $L > L_0$.  Let $\hat{X}_{n,1}, \ldots, \hat{X}_{n,m}$ be defined as in \eqref{def:Xhat}, let $\hat{\mathcal{Y}}_n$ be given by \eqref{Equ:trBlockMat}, and let $\hat{\mathcal{G}}_n(z) := \left( \frac{1}{\sqrt{n}} \hat{\mathcal{Y}}_n - z I \right)^{-1}$.  
	\begin{enumerate}[label=(\roman*)]
		\item \label{item:tr:invertible} For any fixed $\delta > 0$, almost surely, for $n$ sufficiently large, the eigenvalues of $\frac{1}{\sqrt{n}} \trblmat{Y}{n}$ are contained in the disk $\{z \in \mathbb{C} : |z| \leq 1 + \delta \}$.  In particular, $\frac{1}{\sqrt{n}} \trblmat{Y}{n}-zI$ is almost surely invertible for every $z\in\C$ with $|z|>1+\delta$.
		\item \label{item:tr:invtbnd}For any fixed $\delta > 0$, there exists a constant $c > 0$ (depending only on $\delta$) such that almost surely, for $n$ sufficiently large 
		$$ \sup_{z \in \mathbb{C} : |z| > 1 + \delta} \| \hat{\mathcal{G}}_n(z) \| \leq c. $$
		\item \label{item:tr:isotropic} For each $n \geq 1$, let $u_n, v_n \in \mathbb{C}^{mn}$ be deterministic unit vectors.  Then
		$$ \sup_{z \in \mathbb{C} :5\leq |z|\leq 6} \left| u_n^\ast \hat{\mathcal{G}}_n(z) v_n + \frac{1}{z} u_n^\ast v_n \right| \longrightarrow 0 $$
		almost surely as $n \to \infty$. 
	\end{enumerate}
	\label{Thm:reductions}
\end{theorem}

We now prove Theorem \ref{thm:isotropic} assuming Theorem \ref{Thm:reductions}.

\begin{proof}[Proof of Theorem \ref{thm:isotropic}]
	Part \ref{item:invertible} of Theorem \ref{thm:isotropic} follows from Lemma \ref{Lem:LeastSingValAwayFromZero} (see Remark \ref{Rem:ZnotEval}).  In addition, part \ref{item:invtbnd} of Theorem \ref{thm:isotropic} follows from an application of Lemma \ref{Lem:LeastSingValAwayFromZero} and Proposition \ref{Prop:LargeAndSmallSingVals}.  
	
	We now turn to the proof of part \ref{item:isotropic}.  Fix $0 < \delta < 1$.  Let $\mathcal{Y}_n$ be given by \eqref{def:Y}, and let $\mathcal{G}_n$ be given by \eqref{def:G}.  Let $\eps, \eps' > 0$, and observe that there exists a positive constant $M_{1}$ such that for $|z|\geq M_{1}$, 
	\begin{equation*}
	\lnorm \left(-\frac{1}{z}\right)u^{*}v\rnorm \leq \left|\frac{1}{z}\right|\lnorm u^{*}\rnorm\lnorm v\rnorm \leq \frac{\varepsilon}{2}.
	\end{equation*}
	Also, by \cite[Theorem 1.4]{Tout} and Lemma \ref{Lem:specnorm}, there exists a constant $M_{2} > 0$ such that almost surely, for $n$ sufficiently large, 
	\begin{equation*}
	\sup_{|z| \geq M_2} \lnorm u^{*}\mathcal{G}_{n}(z)v\rnorm \leq \sup_{|z| \geq M_2} \| \mathcal{G}_n(z) \| \leq \frac{\eps}{2}. 
	\end{equation*}
	Let $M:=\max\{M_{1},M_{2}, 6\}$. Then, almost surely, for $n$ sufficiently large, 
	\begin{equation}
	\sup_{|z| \geq M} \left|u^{*}\mathcal{G}_{n}(z)v+\frac{1}{z}u^{*}v\right| \leq \varepsilon.
	\label{equ:MtoInf}
	\end{equation}
	
	We now work on the region where $1 + \delta < |z| \leq M$.  By the resolvent identity \eqref{Equ:ResolventIndentity}, we note that
	\[ \lnorm \mathcal{G}_{n}(z)-\hat{\mathcal{G}}_{n}(z)\rnorm  \leq \lnorm \mathcal{G}_n(z) \rnorm \lnorm \hat{\mathcal{G}}_n(z) \rnorm \frac{1}{\sqrt{n}} \lnorm \trblmat{Y}{n} - \blmat{Y}{n} \rnorm. \]
	Thus, by part \ref{item:invtbnd} of Theorem \ref{thm:isotropic} (proven above), Theorem \ref{Thm:reductions}, and Corollary \ref{Cor:YsClose}, there exist constants $C,c > 0$ such that
	\begin{equation}
	 \limsup_{n\rightarrow\infty} \sup_{|z| > 1 + \delta} \lnorm \mathcal{G}_{n}(z)-\hat{\mathcal{G}}_{n}(z)\rnorm 
	\leq \limsup_{n\rightarrow\infty} c^2\frac{1}{\sqrt{n}}\lnorm \trblmat{Y}{n}-\blmat{Y}{n}\rnorm
	\leq c^2 \frac{C}{L} \leq \frac{\eps'}{2} 
	\label{equ:ResolventsColse}
	\end{equation}
	almost surely for $L$ sufficiently large.  
	
	From \eqref{equ:ResolventsColse} and Theorem \ref{Thm:reductions}, almost surely, for $n$ sufficiently large, 
	\begin{align*}
	\sup_{5 \leq |z| \leq 6} \left|u^{*}\mathcal{G}_{n}(z)v+\frac{1}{z}u^{*}v\right| &\leq \sup_{5 \leq |z| \leq 6} \left|u^{*}\mathcal{G}_{n}(z)v-u^{*}\hat{\mathcal{G}}_{n}(z)v\right| \\
	&\qquad \qquad + \sup_{5 \leq |z| \leq 6} \left|u^{*}\hat{\mathcal{G}}_{n}(z)v+\frac{1}{z}u^{*}v\right| \\
	&\leq \varepsilon'.  
	\end{align*}
	Since $\eps' > 0$ was arbitrary, this implies that 
	\[ \limsup_{n \to \infty} \sup_{5 \leq |z| \leq 6} \left|u^{*}\mathcal{G}_{n}(z)v+\frac{1}{z}u^{*}v\right|  = 0 \]
	almost surely.  Since the region $\{z\in\C\;:\;1+\delta\leq |z|\leq M\}$ is compact and contains the region $\{z\in\C\;:\;5\leq |z|\leq 6\}$, Vitali's Convergence Theorem\footnote{The hypothesis of Vitali's Convergence Theorem are satisfied almost surely, for $n$ sufficiently large, by parts \ref{item:invertible} and \ref{item:invtbnd} of Theorem \ref{thm:isotropic}.  In addition, one can check that $(u_n)^\ast \mathcal{G}_n(z) v_n$ is holomorphic in the region $\{z \in \C : |z| > 1 + \delta\}$ almost surely, for $n$ sufficiently large, using the resolvent identity \eqref{Equ:ResolventIndentity}.} (see, for instance, \cite[Lemma 2.14]{BSbook}) implies that we can extend this convergence to the larger region, and we obtain 
	\[ \limsup_{n \to \infty} \sup_{1 + \delta \leq |z| \leq M} \left|u^{*}\mathcal{G}_{n}(z)v+\frac{1}{z}u^{*}v\right|  = 0 \]
	almost surely. In particular, this implies that, almost surely, for $n$ sufficiently large, 
	\[ \sup_{1 + \delta \leq |z| \leq M} \left|u^{*}\mathcal{G}_{n}(z)v+\frac{1}{z}u^{*}v\right| \leq \eps. \]
	Combined with \eqref{equ:MtoInf}, this completes the proof of Theorem \ref{thm:isotropic} (since $\eps > 0$ was arbitrary).  
\end{proof}

It remains to prove Theorem \ref{Thm:reductions}.  We prove parts \ref{item:tr:invertible} and \ref{item:tr:invtbnd} of Theorem \ref{Thm:reductions} now. The proof of part \ref{item:tr:isotropic} is lengthy and will be addressed in the forthcoming sections.  
 
\begin{proof}[Proof of Theorem \ref{Thm:reductions} \ref{item:tr:invertible} and \ref{item:tr:invtbnd}]
	Let $\delta>0$, and observe that by Theorem \ref{Thm:LeastTruncSingValNonZero} and the Borel--Cantelli lemma, there exists a constant $c>0$ (depending only on $\delta$) such that almost surely, for $n$ sufficiently large, 
	\begin{equation} \label{eq:lsvtrbnd}
		\inf_{|z| > 1+\delta}s_{mn}\left(\frac{1}{\sqrt{n}}\trblmat{Y}{n}-zI\right)\geq c.  
		\end{equation}
		This implies (see Remark \ref{Rem:ZnotEval}) that almost surely, for $n$ sufficiently large, the eigenvalues $\frac{1}{\sqrt{n}}\trblmat{Y}{n}$ are contained in the disk $\{z \in \mathbb{C} : |z| \leq 1 + \delta \}$, proving \ref{item:tr:invertible}. From \eqref{eq:lsvtrbnd} and Proposition \ref{Prop:LargeAndSmallSingVals}, we conclude that almost surely, for $n$ sufficiently large, 
	\begin{equation*}
	\sup_{|z|>1+\delta}\lnorm \hat{\mathcal{G}}_{n}(z)\rnorm \leq \frac{1}{c}, 
	\end{equation*}
	proving \ref{item:tr:invtbnd}. 
\end{proof}

\section{Concentration of bilinear forms involving the resolvent $\mathcal{G}_n$}
\label{Sec:Concentration}

Sections \ref{Sec:Concentration} and \ref{sec:combin} are devoted to the proof of part \ref{item:tr:isotropic} of Theorem \ref{Thm:reductions}.  Let $\hat{X}_{n,1}, \ldots, \hat X_{n,m}$ be the truncated matrices from Theorem \ref{Thm:reductions}, and let $u_n, v_n \in \C^{mn}$ be deterministic unit vectors.    For ease of notation, in Sections \ref{Sec:Concentration} and \ref{sec:combin}, we drop the decorations and write $X_{n,1}, \ldots, X_{n,m}$ for $\hat{X}_{n,1}, \ldots, \hat X_{n,m}$.  Similarly, we write $\blmat{Y}{n}$ for $\trblmat{Y}{n}$ and $\blmat{G}{n}$ for $\trblmat{G}{n}$.  Recall that all constants and asymptotic notation may depend on $m$ without explicitly showing the dependence.

Define the following event:
\begin{equation}
\Omega_{n}:=\left\{ \frac{1}{\sqrt{n}} \lnorm \blmat{Y}{n} \rnorm\leq 4.5\right\}.
\label{Equ:condition}
\end{equation}

\begin{lemma}[Spectral Norm Bound for $\blmat{Y}{n}$] Under the assumptions above, the event $\Omega_n$ holds with overwhelming probability.  
	\label{Lem:specnormbound}
\end{lemma}
\begin{proof}
Based on the block structure of $\blmat{Y}{n}$, it follows that 
\[ \lnorm \blmat{Y}{n} \rnorm \leq \max_{i} \lnorm X_{n,i} \rnorm. \]
Therefore, the claim follows from \cite[Theorem 5.9]{BSbook} (alternatively, \cite[Theorem 1.4]{Tout}).  In fact, the constant $4.5$ can be replaced with any constant strictly larger than $2$; $4.5$ will suffice for what follows.  
\end{proof}

By Lemma \ref{Lem:specnormbound}, $\Omega_{n}$ holds with overwhelming probability, i.e., for every $A>0$,
\begin{equation} \label{eq:owpOmega_n}
\P\left(\Omega_{n}\right)= 1-O_{A}(n^{-A}).
\end{equation}
We will return to this fact several times throughout the proof. The remainder of this section is devoted to proving the following lemma.

\begin{lemma}[Concentration]
Let $u_{n}, v_{n} \in \mathbb{C}^{mn}$ be deterministic unit vectors.  Then, under the assumptions above, for any $\varepsilon >0$, almost surely
	\begin{equation}
	\sup_{5 \leq |z| \leq 6} \left|u_{n}^{*}\mathcal{G}_{n}(z)v_{n}\oindicator{\Omega_{n}}-\E\left[u_{n}^{*}\mathcal{G}_{n}(z)v_{n}\oindicator{\Omega_{n}}\right]\right|<\varepsilon
	\end{equation}
	for $n$ sufficiently large.
	\label{Lem:concentration}
\end{lemma}
The proof of Lemma \ref{Lem:concentration} follows the arguments of \cite{BP, OR}. Before we begin the proof, we present some notation. Define $\blmat{Y}{n}^{(k)}$ to be the matrix $\blmat{Y}{n}$ with  all entries in the $k$th row and the $k$th column filled with zeros. Note that $\blmat{Y}{n}^{(k)}$ is still an $mn\times mn$ matrix. Define 
\begin{equation}
\mathcal{G}_{n}^{(k)}:=\left(\frac{1}{\sqrt{n}}\blmat{Y}{n}^{(k)}-zI\right)^{-1}, 
\label{Equ:resolventmodified}
\end{equation} 
let $r_{k}$ denote the $k$th row of $\blmat{Y}{n}$, and let $c_{k}$ denote $k$th column of $\blmat{Y}{n}$. Also define the $\sigma$-algebra 
\begin{equation*}
\mathcal{F}_{k}:=\sigma(r_{1},\ldots,r_{k},c_{1},\ldots,c_{k})
\end{equation*}
generated by the first $k$ rows and the first $k$ columns of $\blmat{Y}{n}$. Note that $\mathcal{F}_{0}$ is the trivial $\sigma$-algebra and $\mathcal{F}_{mn}$ is the $\sigma$-algebra generated by all rows and columns. Next define 
\begin{equation}
\label{def:conditionalExpectation}
\E_{k}[\;\cdot\;]:=\E[\;\cdot\;|\;\mathcal{F}_{k}]
\end{equation}
to be the conditional expectation given the first $k$ rows and columns, and
\begin{equation} 
\Omega_{n}^{(k)}:=\left\{\frac{1}{\sqrt{n}}\lVert \blmat{Y}{n}^{(k)}\rVert\leq 4.5\right\}.
\end{equation}
Observe that $\Omega_{n}\subseteq\Omega_{n}^{(k)}$ and therefore, by Lemma \ref{Lem:specnormbound}, $\Omega_{n}^{(k)}$ holds with overwhelming probability as well. 
\begin{remark}
	By Lemma \ref{Lem:specnorm}, we have that
	\begin{equation*}
	\sup_{5 \leq |z| \leq 6} \lnorm \mathcal{G}_{n}(z) \rnorm\leq 2, \qquad \sup_{5 \leq |z| \leq 6} \lnorm \mathcal{G}_{n}^{(k)}(z) \rnorm\leq 2
	\end{equation*}
	on the event $\Omega_{n}$, and $\sup_{5 \leq |z| \leq 6} \lnorm \mathcal{G}_{n}^{(k)}(z) \rnorm\leq 2$ on $\Omega_{n}^{(k)}$.
	\label{Equ:ResolventNormBdd} 
\end{remark} 

 We will now collect some preliminary calculations and lemmata that will be required for the proof of Lemma \ref{Lem:concentration}. 
 
\begin{lemma}[Rosenthal's Inequality; \cite{Bmart}] Let $\{x_{k}\}$ be a complex martingale difference sequence with respect to the filtration $\mathcal{F}_{k}$. Then for $p\geq 2$,
	\begin{equation*}
	\E\left|\sum x_{k}\right|^{p}\leq K_{p}\left(\E\left(\sum \E\left[|x_{k}|^{2}\;|\;\mathcal{F}_{k-1}\right]\right)^{p/2}+\E\sum |x_{k}|^{p}\right)
	\end{equation*}
	for a constant $K_{p}>0$ which depends only on $p$.
	\label{Lem:rosenthals}
\end{lemma}
\begin{lemma}
	Let $A$ be an $n\times n$  Hermitian positive semidefinite matrix, and let $S \subset [n]$. Then $\sum_{i \in S} A_{ii} \leq \tr A$.  
	\label{Lem:patrialtrace}
\end{lemma}
\begin{proof}
The claim follows from the fact that, by definition of $A$ being Hermitian positive semidefinite, the diagonal entries of $A$ are real and non-negative.
\end{proof}

\begin{lemma} \label{Lem:conjLessThanConstant}
	Let $A$ be an $N\times N$ Hermitian positive semidefinite matrix with rank at most one. Suppose that $\xi$ is a complex-valued random variable with mean zero, unit variance, and which satisfies $|\xi| \leq L$ almost surely for some constant $L > 0$. Let $S\subseteq [N]$, and let $w = (w_i)_{i=1}^N$ be a vector with the following properties:  
	\begin{enumerate}[label=(\roman*)]
		\item $\{w_i : i \in S \}$ is a collection of iid copies of $\xi$, 
		\item $w_{i}=0$ for $i\not\in S$.
	\end{enumerate} 
	Then for any $p\geq 1$,
	\begin{equation}
	\E\left|w^{*}Aw\right|^{p} \ll_{L,p}\lnorm A\rnorm ^{p}.
	\end{equation} 
\end{lemma}

\begin{proof}
	Let $w_{S}$ denote the $|S|$-vector which contains entries $w_{i}$ for $i\in S$, and let $A_{S\times S}$ denote the $|S|\times|S|$ matrix which has entries $A_{(i,j)}$ for $i,j\in S$. Then we observe
	\begin{equation*}
	w^{*}Aw = \sum_{i,j}\bar{w}_{i}A_{(i,j)}w_{j}=  w_{S}^{*}A_{S\times S}w_{S}.
	\end{equation*}
	By Lemma \ref{Lem:BilinearForms}, we get
	\begin{equation*}
		\E\left|w^{*}Aw\right|^{p} \ll_{p}\left( \tr A_{S\times S}\right)^{p}+\E|\xi|^{2p}\tr A^{p}_{S\times S} \leq \left(\tr A_{S \times S} \right)^{p}+L^{2p}\text{tr}A_{S \times S}^{p}.
	\end{equation*}
	Since the rank of $A_{S \times S}$ is at most one, we find
	\[ \tr A_{S \times S} \leq \| A \| \]
	and
	\[ \tr A_{S \times S}^p \leq \| A \|^p, \]
	where we used the fact that the operator norm of a matrix bounds the operator norm of any sub-matrix.  We conclude that 
	\begin{equation*}
	\E\left|w^{*}Aw\right|^{p} \ll_{p}\lnorm A\rnorm^{p}+L^{2p}\lnorm A\rnorm^p  \ll_{L,p}\lnorm A\rnorm ^{p},
	\end{equation*}
	as desired.
\end{proof}

\begin{lemma} \label{lem:BilinearFormWithDifferentVectorsBig}
	Let $A$ be a deterministic complex $N\times N$ matrix. Suppose that $\xi$ is a complex-valued random variable with mean zero, unit variance, and which satisfies $|\xi| \leq L$ almost surely for some constant $L > 0$. Let $S,R\subseteq [N]$, and let $w = (w_i)_{i=1}^N$ and $t = (t_i)_{i=1}^N$ be independent vectors with the following properties:  
	\begin{enumerate}[label=(\roman*)]
		\item $\{w_i : i \in S \}$ and $\{t_j : j \in R \}$ are collections of iid copies of $\xi$, 
		\item $w_{i}=0$ for $i\not\in S$, and  $t_{j}=0$ for $j\not\in R$.
	\end{enumerate} 
	Then for any $p\geq 1$,
	\begin{equation}
	\E\left|w^{*}At\right|^{p}\ll_{L,p}(\tr(A^{*}A))^{p/2}.
	\end{equation}
\end{lemma}
\begin{proof}
Let $w_{S}$ denote the $|S|$-vector which contains entries $w_{i}$ for $i\in S$, and let $t_{R}$ denote the $|R|$-vector which contains entries $t_{j}$ for $j\in R$.  For an $N \times N$ matrix $B$, we let $B_{S \times S}$ denote the $|S| \times |S|$ matrix with entries $B_{(i,j)}$ for $i,j \in S$.  Similarly, we let $B_{R \times R}$ denote the $|R| \times |R|$ matrix with entries $B_{(i,j)}$ for $i,j \in R$.  

We first note that, by the Cauchy--Schwarz inequality, it suffices to assume $p$ is even.  In this case, since $w$ is independent of $t$, Lemma \ref{Lem:BilinearForms} implies that
\begin{align*}
\E|w^{*}At|^{p} &= \E|w^{*}Att^{*}A^{*}w|^{p/2}\\
&= \E\left|w^{*}_{S}(Att^{*}A^{*})_{S\times S}w_{S}\right|^{p/2}\\
&\ll_{p}\E\left[\left(\tr(Att^{*}A^{*})_{S\times S}\right)^{p/2}+L^{p}\tr(Att^{*}A^{*})_{S\times S}^{p/2}\right]. 
\end{align*}
Recall that for any matrix $B$, $\tr(B^{*}B)^{p/2}\leq (\tr(B^{*}B))^{p/2}$. By this fact and by Lemma \ref{Lem:patrialtrace}, we observe that 
\begin{equation*} 
\E\left[\left(\tr(Att^{*}A^{*})_{S\times S}\right)^{p/2}+L^{p}\tr(Att^{*}A^{*})_{S\times S}^{p/2}\right]\ll_{L,p}\E\left[(\tr(Att^{*}A^{*}))^{p/2}\right].
\end{equation*}
By a cyclic permutation of the trace, we have 
\begin{equation*}
\E\left[(\tr(Att^{*}A^{*}))^{p/2}\right] = \E\left[(t^{*}A^{*}At)^{p/2}\right]\leq\E\left|t^{*}A^{*}At\right|^{p/2}.
\end{equation*}
By Lemma \ref{Lem:BilinearForms}, Lemma \ref{Lem:patrialtrace}, and a similar argument as above, we obtain
\begin{align*}
\E\left|t^{*}A^{*}At\right|^{p/2}&=\E\left|t_{R}^{*}(A^{*}A)_{R\times R}t_{R}\right|^{p/2}\\
&\ll_{p}(\tr(A^{*}A)_{R\times R})^{p/2}+L^{p}\tr(A^{*}A)_{R\times R}^{p/2}\\
&\ll_{L,p}(\tr(A^{*}A))^{p/2}, 
\end{align*}
completing the proof.
\end{proof}

\begin{lemma} Let $r_{k}$ be the $k$th row of $\blmat{Y}{n}$, $c_k$ be the $k$th column of $\blmat{Y}{n}$, $\mathcal{G}_{n}^{(k)}(z)$ be the resolvent of $\blmat{Y}{n}^{(k)}$, and $u_{n}\in\C^{mn}$ be a deterministic unit vector.  Then, under the assumptions above, 
	\begin{equation} \label{eq:concsimp1}
	\sup_{5 \leq |z| \leq 6} \E_{k-1}\left|\frac{1}{n}r_{k}\mathcal{G}_{n}^{(k)}(z)u_{n}u_{n}^{*}\mathcal{G}_{n}^{(k)*}(z)\oindicator{\Omega_{n}^{(k)}}r_{k}^{*}\right|^{p}\ll_{L,p}n^{-p}
	\end{equation}
	and 
	\begin{equation} \label{eq:concsimp2}
	\sup_{5 \leq |z| \leq 6} \E_{k-1}\left|\frac{1}{n}c_{k}^{*}\mathcal{G}_{n}^{(k)*}(z)u_{n}u_{n}^{*}\mathcal{G}_{n}^{(k)*}(z)\oindicator{\Omega_{n}^{(k)}}c_{k}\right|^{p}\ll_{L,p}n^{-p}.
	\end{equation}
	almost surely. 
	\label{Lem:concentrationSimplification}
\end{lemma}

\begin{proof}
	We will only prove the bound in \eqref{eq:concsimp1} as the proof of \eqref{eq:concsimp2} is identical. For each fixed $z$ in the band $5 \leq |z| \leq 6$, we will show that 
	\[ \E_{k-1} \left|\frac{1}{n}r_{k}\mathcal{G}_{n}^{(k)}(z)u_{n}u_{n}^{*}\mathcal{G}_{n}^{(k)*}(z)\oindicator{\Omega_{n}^{(k)}}r_{k}^{*}\right|^{p}\ll_{L,p}n^{-p} \]
	surely, where the implicit constant does not depend on $z$.  The claim then follows by taking the supremum over all $z$ in the band $5 \leq |z| \leq 6$.  
	
	To this end, fix $z$ with $5\leq |z|\leq 6$. Throughout the proof, we drop the dependence on $z$ in the resolvent as it is clear from context. Note that
	\begin{equation*}
	\E_{k-1}\left|\frac{1}{n}r_{k}\mathcal{G}_{n}^{(k)}u_{n}u_{n}^{*}\mathcal{G}_{n}^{(k)*}\oindicator{\Omega_{n}^{(k)}}r_{k}^{*}\right|^{p}=\frac{1}{n^{p}}\E_{k-1}\left|r_{k}\left(\mathcal{G}_{n}^{(k)}u_{n}u_{n}^{*}\mathcal{G}_{n}^{(k)*}\oindicator{\Omega_{n}^{(k)}}\right)r_{k}^{*}\right|^{p}, 
	\end{equation*}
	 and $\mathcal{G}_{n}^{(k)}u_{n}u_{n}^{*}\mathcal{G}_{n}^{(k)*}\oindicator{\Omega_{n}^{(k)}}$ is independent of $r_{k}$.  In addition, observe that $\mathcal{G}_{n}^{(k)}u_{n}u_{n}^{*}\mathcal{G}_{n}^{(k)*}\oindicator{\Omega_{n}^{(k)}}$ is Hermitian positive semidefinite matrix with rank at most one.  Applying Lemma \ref{Lem:conjLessThanConstant} and Remark \ref{Equ:ResolventNormBdd}, we obtain   
		\begin{equation*}
		\E_{k-1}\left|\frac{1}{n}r_{k}\mathcal{G}_{n}^{(k)}u_{n}u_{n}^{*}\mathcal{G}_{n}^{(k)*}\oindicator{\Omega_{n}^{(k)}}r_{k}^{*}\right|^{p}\ll_{L,p} \frac{1}{n^{p}}\E_{k-1}\lnorm \mathcal{G}_{n}^{(k)}u_{n}u_{n}^{*}\mathcal{G}_{n}^{(k)*}\oindicator{\Omega_{n}^{(k)}}\rnorm^{p}\ll_{L,p}n^{-p}
	\end{equation*}
	surely, and the proof is complete.  
\end{proof}

\begin{lemma} \label{lemma:indicator}
Let $\zeta_1, \ldots, \zeta_{mn}$ be complex-valued random variables (not necessarily independent) which depend on $\mathcal{Y}_{n}$.  Assume
\[ \sup_{k} |\zeta_{k}|\oindicator{\Omega_n^{(k)}} = O(1) \]
almost surely.  Then, under the assumptions above, for any $p \geq 1$, 
\[ \E \left| \sum_{k=1}^{mn} \left( \zeta_k \oindicator{\Omega_n^{(k)}} - \zeta_k \oindicator{\Omega_n \cap \Omega^{(k)}_n}\right) \right|^p = O_p(n^{-p}). \]
\end{lemma}
\begin{proof}
We will exploit the fact that $\Omega_n \subseteq \Omega_n^{(k)}$ for any $1 \leq k \leq mn$.  Indeed, we have
\begin{align*}
	\E \left| \sum_{k=1}^{mn} \left( \zeta_k \oindicator{\Omega_n^{(k)}} - \zeta_k \oindicator{\Omega_n \cap \Omega^{(k)}_n}\right) \right|^p & \leq \E \left| \sum_{k=1}^{mn} \zeta_k \oindicator{\Omega_n^{(k)} \cap \Omega_n^c}  \right|^p \\
	& \ll_p \E \left( \sum_{k=1}^{mn} \oindicator{\Omega_n^{(k)} \cap \Omega_n^c} \right)^p \\
	&\ll_p n^p \Prob( \Omega_n^c),
\end{align*}
and the claim follows from \eqref{eq:owpOmega_n}.  
\end{proof}

We will made use of the Sherman--Morrison rank one perturbation formula (see \cite[Section 0.7.4]{HJ}). Suppose $A$ is an invertible square matrix, and let $u$, $v$ be vectors. If $1+v^{*}A^{-1}u\neq 0$, then
\begin{equation}
(A+uv^{*})^{-1}=A^{-1}-\frac{A^{-1}uv^{*}A^{-1}}{1+v^{*}A^{-1}u}
	\label{equ:ShermanMorrison1}
\end{equation}
and 
\begin{equation}
(A+uv^{*})^{-1}u=\frac{A^{-1}u}{1+v^{*}A^{-1}u}.
\label{equ:ShermanMorrison2}
\end{equation}

Now, we proceed to prove the main result of this section, Lemma \ref{Lem:concentration}. 
\begin{proof}[Proof of Lemma \ref{Lem:concentration}]
Define
\begin{equation}
\blmat{Y}{n}^{(k1)}:= \blmat{Y}{n}^{(k)}+e_{k}r_{k},\quad\blmat{Y}{n}^{(k2)}:= \blmat{Y}{n}^{(k)}+c_{k}e_{k}^{*}
\label{def:Y(kj)}
\end{equation}  
where $e_{1},\ldots,e_{mn}$ are the standard basis elements of $\C^{mn}$. Also define 
\begin{equation}
\mathcal{G}_{n}^{(kj)}:=\left(\frac{1}{\sqrt{n}}\blmat{Y}{n}^{(kj)}-zI\right)^{-1},\quad j=1,2, 
\end{equation}
and set 
\begin{align*}
& \alpha_{n}^{(k)} := \frac{1}{1+z^{-1}n^{-1}r_{k}\mathcal{G}_{n}^{(k)}c_{k}\oindicator{\Omega_{n}}},\\
& \zeta_{n}^{(k)} :=n^{-1}r_{k}\mathcal{G}_{n}^{(k)}c_{k},\\
&\eta_{n}^{(k)} := n^{-1}r_{k}\mathcal{G}_{n}^{(k)}v_{n}u_{n}^{*}\mathcal{G}_{n}^{(k)}c_{k}.  
\end{align*} 
Using these definitions, we make the following observations.  
\begin{enumerate}[label=(\roman*)]
	\item \label{item:F1} Since the only nonzero element in the $k$-th row and $k$-th column of $\blmat{Y}{n}^{(k)}-zI$ is on the diagonal, 
	\begin{equation*} 
	e_{k}^{*}\mathcal{G}_{n}^{(k)}e_{k}=-z^{-1},\quad e_{k}^{*}\mathcal{G}_{n}^{(k)}v_{n}= -z^{-1}v_{n,k},\quad u_{n}^{*}\mathcal{G}_{n}^{(k)}e_{k}=-z^{-1}\bar{u}_{n,k}
	\end{equation*}
	where $u_{n}=(u_{n,k})_{k=1}^{mn}$ and $v_{n}=(v_{n,k})_{k=1}^{mn}$. 
	\item \label{item:F2} Since the $k$-th elements of $c_{k}$ and $r_{k}$ are zero,
	\begin{equation*}
	e_{k}^{*}\mathcal{G}_{n}^{(k)}c_{k}=0,\quad r_{k}\mathcal{G}_{n}^{(k)}e_{k}=0.
	\end{equation*} 
	\item \label{item:F3} By \eqref{equ:ShermanMorrison1}, \ref{item:F1}, and \ref{item:F2}, 
	\begin{align*}
	e_{k}^{*}\mathcal{G}_{n}^{(k1)}n^{-1/2}c_{k} & = e_{k}^{*}\mathcal{G}_{n}^{(k)}n^{-1/2}c_{k}-\frac{e_{k}^{*}\mathcal{G}_{n}^{(k)}e_{k}n^{-1/2}r_{k}\mathcal{G}_{n}^{(k)}n^{-1/2}c_{k}}{1+n^{-1/2}r_{k}\mathcal{G}_{n}^{(k)}e_{k}}\\
	&=z^{-1}n^{-1}r_{k}\mathcal{G}_{n}^{(k)}c_{k}, 
	\end{align*}
	so that 
	\begin{equation*}
	\frac{1}{1+e_{k}^{*}\mathcal{G}_{n}^{(k1)}\oindicator{\Omega_{n}}n^{-1/2}c_{k}}=\alpha_{n}^{(k)}.
	\end{equation*}
	\item \label{item:F4} By the same argument as \ref{item:F3}, 
	\begin{equation*}
	n^{-1/2}r_{k}\mathcal{G}_{n}^{(k2)}e_{k}=z^{-1}n^{-1}r_{k}\mathcal{G}_{n}^{(k)}c_{k}, 
	\end{equation*}
	so that 
	\begin{equation*}
	\frac{1}{1+n^{-1/2}r_{k}\mathcal{G}_{n}^{(k2)}e_{k}\oindicator{\Omega_{n}}}=\alpha_{n}^{(k)}.
	\end{equation*}
	\item \label{item:F5} By Schur's compliment, 
	\begin{equation*}
	(\mathcal{G}_{n})_{(k,k)} = -\frac{1}{z+n^{-1}r_{k}\mathcal{G}_{n}^{(k)}c_{k}}
	\end{equation*}
	provided the necessary inverses exist (which is the case on the event $\Omega_n$).  Thus, on $\Omega_{n}=\Omega_{n}\cap\Omega_{n}^{(k)}$ and uniformly for $5\leq |z|\leq 6$, by Remark \ref{Equ:ResolventNormBdd},
	\begin{align*}
	\left|\alpha_{n}^{(k)}\right| &=\left|\frac{z}{z+n^{-1}r_{k}\mathcal{G}_{n}^{(k)}c_{k}}\right|\\
	&=|z|\left|(\mathcal{G}_{n})_{(k,k)}\right|\\
	&\leq 12.
	\end{align*}
	On $\Omega_{n}^{c}$, $\alpha_{n}^{(k)}=1$, so we have that, almost surely,  
	\begin{equation*}
		\left|\alpha_{n}^{(k)}\right|\leq 12.
	\end{equation*}
	\item \label{item:F6} By \eqref{equ:ShermanMorrison2} and \ref{item:F3}, 
	\begin{align*}
	u_{n}^{*}\mathcal{G}_{n}n^{-1/2}c_{k}\oindicator{\Omega_{n}} & =\frac{u_{n}^{*}\mathcal{G}_{n}^{(k1)}n^{-1/2}c_{k}\oindicator{\Omega_{n}}}{1+e_{k}^{*}\mathcal{G}_{n}^{(k1)}n^{-1/2}c_{k}}\\
	& =\frac{u_{n}^{*}\mathcal{G}_{n}^{(k1)}n^{-1/2}c_{k}\oindicator{\Omega_{n}}}{1+e_{k}^{*}\mathcal{G}_{n}^{(k1)}n^{-1/2}c_{k}\oindicator{\Omega_{n}}}\\
	& =u_{n}^{*}\mathcal{G}_{n}^{(k1)}n^{-1/2}c_{k}\oindicator{\Omega_{n}}\alpha_{n}^{(k)}. 
	\end{align*}
	\item \label{item:F7} By \eqref{equ:ShermanMorrison2} and \ref{item:F4}, a similar argument as above gives 
	\begin{equation*}
	u_{n}^{*}\mathcal{G}_{n}e_{k}\oindicator{\Omega_{n}} = u_{n}^{*}\mathcal{G}_{n}^{(k2)}e_{k}\oindicator{\Omega_{n}}\alpha_{n}^{(k)}. 
	\end{equation*} 
	\item \label{item:F8} By \eqref{equ:ShermanMorrison1} and \ref{item:F2},
	\begin{align*}
	\mathcal{G}_{n}^{(k1)} &=\mathcal{G}_{n}^{(k)}-\frac{\mathcal{G}_{n}^{(k)}e_{k}n^{-1/2}r_{k}\mathcal{G}_{n}^{(k)}}{1+n^{-1/2}r_{k}\mathcal{G}_{n}^{(k)}e_{k}}\\ 
	&=\mathcal{G}_{n}^{(k)}-\mathcal{G}_{n}^{(k)}e_{k}n^{-1/2}r_{k}\mathcal{G}_{n}^{(k)}.
	\end{align*}
	\item \label{item:F9} By \eqref{equ:ShermanMorrison1} and \ref{item:F2}, and by the same calculation as in \ref{item:F8}, 
	\begin{equation*}
	\mathcal{G}_{n}^{(k2)}=\mathcal{G}_{n}^{(k)}-\mathcal{G}_{n}^{(k)}n^{-1/2}c_{k}e_{k}^{*}\mathcal{G}_{n}^{(k)}. 
	\end{equation*}
	\item \label{item:F10} By definition of $\alpha_{n}^{(k)}$, 
	\begin{equation*}
	z^{-1}(\E_{k}-\E_{k-1})[n^{-1}r_{k}\mathcal{G}_{n}^{(k)}c_{k}\alpha_{n}^{(k)}]=-(\E_{k}-\E_{k-1})[\alpha_{n}^{(k)}].
	\end{equation*}
	\item \label{item:F11} By definition of $\alpha_{n}^{(k)}$ and $\zeta_{n}^{(k)}$, 
	\begin{align*}
	\alpha_{n}^{(k)}-1&=\frac{-z^{-1}n^{-1}r_{k}\mathcal{G}_{n}^{(k)}\oindicator{\Omega_{n}}c_{k}}{1+z^{-1}n^{-1}r_{k}\mathcal{G}_{n}^{(k)}\oindicator{\Omega_{n}}c_{k}}\\
	&=-z^{-1}\zeta_{n}^{(k)}\alpha_{n}^{(k)}\oindicator{\Omega_{n}}.
	\end{align*}
	\item \label{item:F12} The entries of $c_{k}$ and $r_{k}$ have mean zero, unit variance, and are bounded by $4L$ almost surely. In addition, $(r_{k}^{T},c_{k})$ and $\mathcal{G}_{n}^{(k)}\oindicator{\Omega_{n}^{(k)}}$ are independent. By Remark \ref{Equ:ResolventNormBdd}, $\lnorm\mathcal{G}_{n}^{(k)*}u_{n}v_{n}^{*}\mathcal{G}_{n}^{(k)*}\mathcal{G}_{n}^{(k)}v_{n}u_{n}^{*}\mathcal{G}_{n}^{(k)}\rnorm\leq 16$ on $\Omega_{n}^{(k)}$. Thus, by Lemma \ref{lem:BilinearFormWithDifferentVectorsBig}, for any $p\geq2$, 
	\begin{align*}
	\sup_{1\leq k\leq n}\E_{k-1}[|\eta_{n}^{(k)}|^{p}\oindicator{\Omega_{n}}] &\leq \sup_{1\leq k\leq n}\E_{k-1}[|\eta_{n}^{(k)}|^{p}\oindicator{\Omega_{n}^{(k)}}]\\
	&\ll_{L,p}\sup_{1\leq k\leq n} n^{-p}\E_{k-1}\left[\left(\tr(\mathcal{G}_{n}^{(k)*}u_{n}v_{n}^{*}\mathcal{G}_{n}^{(k)*}\mathcal{G}_{n}^{(k)}v_{n}u_{n}^{*}\mathcal{G}_{n}^{(k)})\right)^{p/2}\oindicator{\Omega_{n}^{(k)}}\right]\\
	&\ll_{L,p} n^{-p}
	\end{align*}
	since
	\[ \mathcal{G}_{n}^{(k)*}u_{n}v_{n}^{*}\mathcal{G}_{n}^{(k)*}\mathcal{G}_{n}^{(k)}v_{n}u_{n}^{*}\mathcal{G}_{n}^{(k)} \]
	is at most rank one.  Similarly, Remark \ref{Equ:ResolventNormBdd} give the almost sure bound $\lnorm\mathcal{G}_{n}^{(k)*}\mathcal{G}_{n}^{(k)}\rnorm\leq 4$ on $\Omega_{n}^{(k)}$, which gives
	\begin{align*}
	\sup_{1\leq k\leq n}\E_{k-1}[|\zeta_{n}^{(k)}|^{p}\oindicator{\Omega_{n}}] &\leq \sup_{1\leq k\leq n}\E_{k-1}[|\zeta_{n}^{(k)}|^{p}\oindicator{\Omega_{n}^{(k)}}]\\
	&\ll_{L,p}\sup_{1\leq k\leq n} n^{-p}\E_{k-1}\left[\left(\tr(\mathcal{G}_{n}^{(k)*}\mathcal{G}_{n}^{(k)})\right)^{p/2}\oindicator{\Omega_{n}^{(k)}}\right]\\
	&\ll_{L,p} n^{-p/2}. 
	\end{align*}
\end{enumerate}

With the above observations in hand, we now complete the proof.  By the resolvent identity \eqref{Equ:ResolventIndentity} and Remark \ref{Equ:ResolventNormBdd}, it follows that the function
\[ z \mapsto \left|u_{n}^{*}\mathcal{G}_{n}(z)v_{n}\oindicator{\Omega_{n}}-\E\left[u_{n}^{*}\mathcal{G}_{n}(z)v_{n}\oindicator{\Omega_{n}}\right]\right| \]
is Lipschitz continuous in the region $\{z \in \C : 5 \leq |z| \leq 6 \}$.  Thus, by a standard net argument, it suffices to prove that almost surely, for $n$ sufficiently large, 
\[ \left|u_{n}^{*}\mathcal{G}_{n}(z)v_{n}\oindicator{\Omega_{n}}-\E\left[u_{n}^{*}\mathcal{G}_{n}(z)v_{n}\oindicator{\Omega_{n}}\right]\right|<\varepsilon \]
for each fixed $z \in \mathbb{C}$ with $5 \leq |z| \leq 6$.  To this end, fix such a value of $z$.  Throughout the proof, we drop the dependence on $z$ in the resolvent as it is clear from context. Note that by Markov's inequality and the Borel--Cantelli lemma, it is sufficient to prove that
	\begin{equation}
	\E\left|u_{n}^{*}\mathcal{G}_{n}v_{n}\oindicator{\Omega_{n}}-\E\left[u_{n}^{*}\mathcal{G}_{n}v_{n}\oindicator{\Omega_{n}}\right]\right|^{p}=O_{L,p}(n^{-p/2})
	\end{equation}
	for all $p> 2$.  We now rewrite the above expression as a martingale difference sequence: 
	\begin{equation*}
	u_{n}^{*}\mathcal{G}_{n}v_{n}\oindicator{\Omega_{n}}-\E[u_{n}^{*}\mathcal{G}_{n}v_{n}\oindicator{\Omega_{n}}]	=\sum_{k=1}^{mn}\left(\E_{k}-\E_{k-1}\right)[u_{n}^{*}\mathcal{G}_{n}v_{n}\oindicator{\Omega_{n}}].
	\end{equation*}
	Since $u_{n}^{*}\mathcal{G}_{n}^{(k)}v_{n}\oindicator{\Omega_{n}^{(k)}}$ is independent of $r_k$ and $c_{k}$, one can see that
	$$(\E_{k}-\E_{k-1})[u_{n}^{*}\mathcal{G}_{n}^{(k)}v_{n}\oindicator{\Omega_{n}^{(k)}}]=0,$$ 
	and so
	\begin{equation*}
	u_{n}^{*}\mathcal{G}_{n}v_{n}\oindicator{\Omega_{n}}-\E[u_{n}^{*}\mathcal{G}_{n}v_{n}\oindicator{\Omega_{n}}] =\sum_{k=1}^{mn}\left(\E_{k}-\E_{k-1}\right)\left[u_{n}^{*}\mathcal{G}_{n}v_{n}\oindicator{\Omega_{n}}-u_{n}^{*}\mathcal{G}_{n}^{(k)}v_{n}\oindicator{\Omega_{n}^{(k)}}\right].
	\end{equation*}
	 In view of Lemma \ref{lemma:indicator} and since $\Omega_{n}\cap\Omega_{n}^{(k)}=\Omega_{n}$, it suffices to prove that
	\[ \E \left| \sum_{k=1}^{mn}\left(\E_{k}-\E_{k-1}\right)\left[u_{n}^{*}\mathcal{G}_{n}v_{n}\oindicator{\Omega_{n} }-u_{n}^{*}\mathcal{G}_{n}^{(k)}v_{n}\oindicator{\Omega_n}\right] \right|^p = O_{L,p}(n^{-p/2}) \]
	for all $p > 2$.  Define 
	\begin{equation}
	W_{k}:=\left(\E_{k}-\E_{k-1}\right)\left[u_{n}^{*}\mathcal{G}_{n}v_{n}\oindicator{\Omega_{n}}-u_{n}^{*}\mathcal{G}_{n}^{(k)}v_{n}\oindicator{\Omega_{n}}\right]
	\label{def:W_k}
	\end{equation}
	and observe that $\{W_{k}\}_{k=1}^{mn}$ is a martingale difference sequence with respect to $\{ \mathcal{F}_{k} \}$.
	
	From the resolvent identity \eqref{Equ:ResolventIndentity}, we observe that
	\begin{align*}
	\sum_{k=1}^{mn}W_{k}&=\sum_{k=1}^{mn}\left(\E_{k}-\E_{k-1}\right)\left[u_{n}^{*}\mathcal{G}_{n}\frac{1}{\sqrt{n}} (\blmat{Y}{n}^{(k)}-\blmat{Y}{n})\mathcal{G}_{n}^{(k)}v_{n}\oindicator{\Omega_{n}}\right]\\
	&=-\sum_{k=1}^{mn}\left(\E_{k}-\E_{k-1}\right)\left[u_{n}^{*}\mathcal{G}_{n}\left(\frac{1}{\sqrt{n}}e_{k}r_{k}+\frac{1}{\sqrt{n}}c_{k}e_{k}^{*}\right)\mathcal{G}_{n}^{(k)}v_{n}\oindicator{\Omega_{n}}\right]\\
	&:=-\sum_{k=1}^{mn}(W_{k1}+W_{k2})
	\end{align*}
	where we define
	\begin{equation*} 
	W_{k1}:=\left(\E_{k}-\E_{k-1}\right)\left[u_{n}^{*}\mathcal{G}_{n}n^{-1/2}e_{k}r_{k}\mathcal{G}_{n}^{(k)}v_{n}\oindicator{\Omega_n }\right], 
	\label{def:W_k1}
	\end{equation*}
	and 
	\begin{equation*} 
	W_{k2}:=\left(\E_{k}-\E_{k-1}\right)\left[u_{n}^{*}\mathcal{G}_{n}n^{-1/2}c_{k}e_{k}^{*}\mathcal{G}_{n}^{(k)}v_{n}\oindicator{\Omega_n}\right].
	\label{def:W_k2}
	\end{equation*}
	
	By \ref{item:F1}, \ref{item:F7}, and \ref{item:F9}, we can further decompose
	\begin{align*}
	\sum_{k=1}^{mn} W_{k1} &=\sum_{k=1}^{mn}\left(\E_{k}-\E_{k-1}\right)\left[u_{n}^{*}\mathcal{G}_{n}n^{-1/2}e_{k}r_{k}\mathcal{G}_{n}^{(k)}v_{n}\oindicator{\Omega_n }\right]\\
	&=\sum_{k=1}^{mn}\left(\E_{k}-\E_{k-1}\right)\left[u_{n}^{*}\mathcal{G}_{n}^{(k2)}n^{-1/2}e_{k}r_{k}\mathcal{G}_{n}^{(k)}v_{n}\alpha_{n}^{(k)}\oindicator{\Omega_n }\right]\\
	&=-\sum_{k=1}^{mn}\left(\E_{k}-\E_{k-1}\right)\left[z^{-1}(\bar{u}_{n,k}-u_{n}^{*}\mathcal{G}_{n}^{(k)}n^{-1/2}c_{k})n^{-1/2}r_{k}\mathcal{G}_{n}^{(k)}v_{n}\alpha_{n}^{(k)}\oindicator{\Omega_n }\right]\\
	&=-\sum_{k=1}^{mn}(W_{k11}+W_{k12})
	\end{align*}
	where
	\begin{equation*}
	W_{k11}:= \left(\E_{k}-\E_{k-1}\right)\left[z^{-1}\bar{u}_{n,k}n^{-1/2}r_{k}\mathcal{G}_{n}^{(k)}v_{n}\alpha_{n}^{(k)}\oindicator{\Omega_n }\right]
	\label{def:W_k11}
	\end{equation*}
	and
	\begin{equation*}
	W_{k12}:= -\left(\E_{k}-\E_{k-1}\right)\left[z^{-1}u_{n}^{*}\mathcal{G}_{n}^{(k)}n^{-1}c_{k}r_{k}\mathcal{G}_{n}^{(k)}v_{n}\alpha_{n}^{(k)}\oindicator{\Omega_n }\right].
	\label{def:W_k112}
	\end{equation*}
	Similarly, by \ref{item:F1},\ref{item:F6}, \ref{item:F8}, and \ref{item:F10}, 
	\begin{align*}
	\sum_{k=1}^{mn}W_{k2}&=\sum_{k=1}^{mn}\left(\E_{k}-\E_{k-1}\right)\left[u_{n}^{*}\mathcal{G}_{n}n^{-1/2}c_{k}e_{k}^{*}\mathcal{G}_{n}^{(k)}v_{n}\oindicator{\Omega_n}\right]\\
	&=\sum_{k=1}^{mn}\left(\E_{k}-\E_{k-1}\right)\left[u_{n}^{*}\frac{\mathcal{G}_{n}^{(k1)}n^{-1/2}c_{k}}{1+r_{k}^{*}\mathcal{G}_{n}^{(k1)}c_{k}}e_{k}^{*}\mathcal{G}_{n}^{(k)}v_{n}\oindicator{\Omega_n}\right]\\
	&=\sum_{k=1}^{mn}\left(\E_{k}-\E_{k-1}\right)\left[u_{n}^{*}\frac{\mathcal{G}_{n}^{(k1)}n^{-1/2}c_{k}}{1+r_{k}^{*}\mathcal{G}_{n}^{(k1)}c_{k}\oindicator{\Omega_{n}}}e_{k}^{*}\mathcal{G}_{n}^{(k)}v_{n}\oindicator{\Omega_n}\right]\\
	&=\sum_{k=1}^{mn}\left(\E_{k}-\E_{k-1}\right)\left[u_{n}^{*}\mathcal{G}_{n}^{(k1)}n^{-1/2}c_{k}e_{k}^{*}\mathcal{G}_{n}^{(k)}v_{n}\alpha_{n}^{(k)}\oindicator{\Omega_n}\right]\\
	&=-\sum_{k=1}^{mn}\left(\E_{k}-\E_{k-1}\right)\left[z^{-1}v_{n,k}u_{n}^{*}\mathcal{G}_{n}^{(k1)}n^{-1/2}c_{k}\alpha_{n}^{(k)}\oindicator{\Omega_n}\right]\\
	&=-\sum_{k=1}^{mn}\left(\E_{k}-\E_{k-1}\right)\left[z^{-1}v_{n,k}u_{n}^{*}(\mathcal{G}_{n}^{(k)}-\mathcal{G}_{n}^{(k)}e_{k}n^{-1/2}r_{k}\mathcal{G}_{n}^{(k)})n^{-1/2}c_{k}\alpha_{n}^{(k)}\oindicator{\Omega_n}\right]\\
	&=-\sum_{k=1}^{mn}\left(\E_{k}-\E_{k-1}\right)\left[z^{-1}v_{n,k}(u_{n}^{*}\mathcal{G}_{n}^{(k)}n^{-1/2}c_{k}\alpha_{n}^{(k)}-u_{n}^{*}\mathcal{G}_{n}^{(k)}e_{k}n^{-1}r_{k}\mathcal{G}_{n}^{(k)}c_{k}\alpha_{n}^{(k)})\oindicator{\Omega_n}\right]\\
	&=-\sum_{k=1}^{mn}\left(\E_{k}-\E_{k-1}\right)\left[z^{-1}v_{n,k}(u_{n}^{*}\mathcal{G}_{n}^{(k)}n^{-1/2}c_{k}\alpha_{n}^{(k)}+\bar{u}_{n,k}z^{-1}n^{-1}r_{k}\mathcal{G}_{n}^{(k)}c_{k}\alpha_{n}^{(k)})\oindicator{\Omega_n}\right]\\
	&=-\sum_{k=1}^{mn}\left(\E_{k}-\E_{k-1}\right)\left[z^{-1}v_{n,k}u_{n}^{*}\mathcal{G}_{n}^{(k)}n^{-1/2}c_{k}\alpha_{n}^{(k)}\oindicator{\Omega_n}-z^{-1}\bar{u}_{n,k}v_{n,k}\alpha_{n}^{(k)}\oindicator{\Omega_n}\right]\\
	&=-\sum_{k=1}^{mn}(W_{k21}+W_{k22})
	\end{align*}
	where
	\begin{equation*}
	W_{k21}:=\left(\E_{k}-\E_{k-1}\right)\left[z^{-1}v_{n,k}u_{n}^{*}\mathcal{G}_{n}^{(k)}n^{-1/2}c_{k}\alpha_{n}^{(k)}\oindicator{\Omega_n}\right]
	\label{def:W_k21}
	\end{equation*}
	and 
	\begin{equation*}
	W_{k22}:=-\left(\E_{k}-\E_{k-1}\right)\left[z^{-1}\bar{u}_{n,k}v_{n,k}\alpha_{n}^{(k)}\oindicator{\Omega_n}\right].
	\label{def:W_k22}
	\end{equation*}
	Thus, in order to complete the proof, it suffices to show that for all $p> 2$,
	\begin{equation} \label{eq:concshow}
	\E\left|\sum_{k=1}^{mn}W_{k11}\right|^{p}+\E\left|\sum_{k=1}^{mn}W_{k12}\right|^{p}+\E\left|\sum_{k=1}^{mn}W_{k21}\right|^{p}+\E\left|\sum_{k=1}^{mn}W_{k22}\right|^{p}=O_{L,p}\left(n^{-p/2}\right).
	\end{equation}
	We bound each term individually.
	To begin, observe that by Lemma \ref{Lem:rosenthals}, Lemma \ref{Lem:concentrationSimplification}, and \ref{item:F5}, for any $p>2$, 
	\begin{align*}
	\E\left|\sum_{k=1}^{mn}W_{k11}\right|^{p} &=\E\left|\sum_{k=1}^{mn}\left(\E_{k}-\E_{k-1}\right)\left[z^{-1}\bar{u}_{n,k}n^{-1/2}r_{k}\mathcal{G}_{n}^{(k)}v_{n}\alpha_{n}^{(k)}\oindicator{\Omega_n }\right]\right|^{p}\\
	&\ll_{p} \E\left[\sum_{k=1}^{mn}\E_{k-1}\left|(\E_{k}-\E_{k-1})[z^{-1}\bar{u}_{n,k}n^{-1/2}r_{k}\mathcal{G}_{n}^{(k)}v_{n}\alpha_{n}^{(k)}\oindicator{\Omega_{n}}]\right|^{2}\right]^{p/2}\\
	&\quad \quad \quad \quad +\E\left[\sum_{k=1}^{mn}\left|(\E_{k}-\E_{k-1})[z^{-1}\bar{u}_{n,k}n^{-1/2}r_{k}\mathcal{G}_{n}^{(k)}v_{n}\alpha_{n}^{(k)}\oindicator{\Omega_{n}}]\right|^{p}\right]\\
	&\ll_{p} \E\left[\sum_{k=1}^{mn}|\bar{u}_{n,k}|^{2}\E_{k-1}\left|n^{-1/2}r_{k}\mathcal{G}_{n}^{(k)}v_{n}\oindicator{\Omega_{n}}\right|^{2}\right]^{p/2}\\
	&\quad \quad \quad \quad +\sum_{k=1}^{mn}|\bar{u}_{n,k}|^{p}\E\left|n^{-1/2}r_{k}\mathcal{G}_{n}^{(k)}v_{n}\oindicator{\Omega_{n}}\right|^{p}\\
	&\ll_{p} \E\left[\sum_{k=1}^{mn}|\bar{u}_{n,k}|^{2}\E_{k-1}\left|n^{-1/2}r_{k}\mathcal{G}_{n}^{(k)}v_{n}\oindicator{\Omega_{n}^{(k)}}\right|^{2}\right]^{p/2}\\
	&\quad \quad \quad \quad +\sum_{k=1}^{mn}|\bar{u}_{n,k}|^{p}\E\left|n^{-1/2}r_{k}\mathcal{G}_{n}^{(k)}v_{n}\oindicator{\Omega_{n}^{(k)}}\right|^{p}\\
	&\ll_{L,p}\left(\E\left[\sum_{k=1}^{mn}|\bar{u}_{n,k}|^{2}\cdot n^{-1}\right]^{p/2}+\sum_{k=1}^{mn}|\bar{u}_{n,k}|^{p}\cdot n^{-p/2}\right)\\
	&\ll_{L,p} n^{-p/2},
	\end{align*}
	where we also used Jensen's inequality and the fact that $|z| \geq 5$.  Similarly, by Lemma \ref{Lem:rosenthals}, Lemma \ref{Lem:concentrationSimplification}, and \ref{item:F5}, for any $p>2$, 
	\begin{align*}
	\E\left|\sum_{k=1}^{mn}W_{k21}\right|^{p} &=\E\left|\sum_{k=1}^{mn}\left(\E_{k}-\E_{k-1}\right)\left[z^{-1}v_{n,k}u_{n}^{*}\mathcal{G}_{n}^{(k)}n^{-1/2}c_{k}\alpha_{n}^{(k)}\oindicator{\Omega_n }\right]\right|^{p}\\
	&\ll_{p} \E\left[\sum_{k=1}^{mn}\E_{k-1}\left|(\E_{k}-\E_{k-1})[z^{-1}v_{n,k}u_{n}^{*}\mathcal{G}_{n}^{(k)}n^{-1/2}c_{k}\alpha_{n}^{(k)}\oindicator{\Omega_n }]\right|^{2}\right]^{p/2}\\
	&\quad \quad \quad \quad +\E\left[\sum_{k=1}^{mn}\left|(\E_{k}-\E_{k-1})[z^{-1}v_{n,k}u_{n}^{*}\mathcal{G}_{n}^{(k)}n^{-1/2}c_{k}\alpha_{n}^{(k)}\oindicator{\Omega_n }]\right|^{p}\right]\\
	&\ll_{p} \E\left[\sum_{k=1}^{mn}|v_{n,k}|^{2}\E_{k-1}\left|u_{n}^{*}\mathcal{G}_{n}^{(k)}n^{-1/2}c_{k}\oindicator{\Omega_{n}}\right|^{2}\right]^{p/2}\\
	&\quad \quad \quad \quad +\sum_{k=1}^{mn}|v_{n,k}|^{p}\E\left|u_{n}^{*}\mathcal{G}_{n}^{(k)}n^{-1/2}c_{k}\oindicator{\Omega_{n}}\right|^{p}\\
	&\ll_{p} \E\left[\sum_{k=1}^{mn}|v_{n,k}|^{2}\E_{k-1}\left|u_{n}^{*}\mathcal{G}_{n}^{(k)}n^{-1/2}c_{k}\oindicator{\Omega_{n}^{(k)}}\right|^{2}\right]^{p/2}\\
	&\quad \quad \quad \quad +\sum_{k=1}^{mn}|v_{n,k}|^{p}\E\left|u_{n}^{*}\mathcal{G}_{n}^{(k)}n^{-1/2}c_{k}\oindicator{\Omega_{n}^{(k)}}\right|^{p}\\
	&\ll_{L,p}\left(\E\left[\sum_{k=1}^{mn}|v_{n,k}|^{2}\cdot n^{-1}\right]^{p/2}+\sum_{k=1}^{mn}|v_{n,k}|^{p}\cdot n^{-p/2}\right)\\
	&\ll_{L,p} n^{-p/2}. 
	\end{align*}
	Next, by Lemma \ref{Lem:rosenthals}, \ref{item:F5}, and \ref{item:F12}, for any $p>2$, 
	\begin{align*}
	\E\left|\sum_{k=1}^{mn}W_{k12}\right|^{p} &=\E\left|\sum_{k=1}^{mn}(E_{k}-\E_{k-1})[z^{-1}u_{n}^{*}\mathcal{G}_{n}^{(k)}n^{-1}c_{k}r_{k}\mathcal{G}_{n}^{(k)}v_{n}\alpha_{n}^{(k)}\oindicator{\Omega_{n}}]\right|^{p}\\
	&=\E\left|\sum_{k=1}^{mn}(\E_{k}-\E_{k-1})[z^{-1}\alpha_{n}^{(k)}\eta_{n}^{(k)}\oindicator{\Omega_{n}}]\right|^p\\
	&\ll_{p} \E\left[\sum_{k=1}^{mn}\E_{k-1}\left|(\E_{k}-\E_{k-1})[z^{-1}\alpha_{n}^{(k)}\eta_{n}^{(k)}\oindicator{\Omega_{n}}]\right|^{2}\right]^{p/2}\\
	&\quad \quad \quad \quad +\E\left[\left|(\E_{k}-\E_{k-1})[z^{-1}\alpha_{n}^{(k)}\eta_{n}^{(k)}\oindicator{\Omega_{n}}]\right|^{p}\right]\\
	&\ll_{p} \left(\E\left[\sum_{k=1}^{mn}\E_{k-1}\left|\eta_{n}^{(k)}\oindicator{\Omega_{n}}\right|^2\right]^{p/2}+\sum_{k=1}^{mn}\E\left|\eta_{n}^{(k)}\oindicator{\Omega_{n}}\right|^{p}\right)\\
	&\ll_{L,p} \left(\E\left[\sum_{k=1}^{mn}n^{-2}\right]^{p/2}+\sum_{k=1}^{mn}n^{-p}\right)\\
	&\ll_{L,p} n^{-p/2}+n^{-p+1}\\
	&\ll_{L,p} n^{-p/2}, 
	\end{align*}
	where we also used Jensen's inequality and the fact that $|z| \geq 5$; the last inequality follows from the fact that $p > 2$.  
	
	Finally, moving on to $W_{k22}$, by \ref{item:F11} we can decompose this further as
	\begin{align*}
	\sum_{k=1}^{mn}W_{k22}&=\sum_{k=1}^{mn}(\E_{k}-\E_{k-1})[z^{-1}\bar{u}_{n,k}v_{n,k}\alpha_{n}^{(k)}\oindicator{\Omega_{n}}]\\
	&=\sum_{i=1}^{mn}(\E_{k}-\E_{k-1})[z^{-1}\bar{u}_{n,k}v_{n,k}(1-z^{-1}\zeta_{n}^{(k)}\alpha_{n}^{(k)}\oindicator{\Omega_{n}})]\\
	&=\sum_{k=1}^{mn}(W_{k221}+W_{k222})
	\end{align*}
	where
	\begin{equation*}
	W_{k221}:=(\E_{k}-\E_{k-1})[z^{-1}\bar{u}_{n,k}v_{n,k}]
	\label{def:W_k221}
	\end{equation*}
	and
	\begin{equation*}
	W_{k222}:=-(\E_{k}-\E_{k-1})[z^{-2}\bar{u}_{n,k}v_{n,k}\zeta_{n}^{(k)}\alpha_{n}^{(k)}\oindicator{\Omega_{n}}].
	\label{def:W_k222}
	\end{equation*}
	Since $z^{-1}\bar{u}_{n,k}v_{n,k}$ is deterministic, it follows that 
	\[ \left|\sum_{k=1}^{mn}W_{k221}\right|^{p}=0. \]
	Thus, it suffices to show that $\E\left|\sum_{k=1}^{mn}W_{k222}\right|^p=O_L,p(n^{-p/2})$ for $p>2$. 
	By Lemma \ref{Lem:rosenthals}, \ref{item:F5}, and \ref{item:F12}, we have that for any $p>2$,
	\begin{align*}
	\E\left|\sum_{k=1}^{mn}W_{k222}\right|^{p}&=\E\left|\sum_{k=1}^{mn}(\E_{k}-\E_{k-1})[z^{-2}\bar{u}_{n,k}v_{n,k}\zeta_{n}^{(k)}\alpha_{n}^{(k)}\oindicator{\Omega_{n}}]\right|^{p}\\
	&\ll_{p}\E\left[\sum_{k=1}^{mn}\E_{k-1}\left|(\E_{k}-\E_{k-1})[z^{-2}\bar{u}_{n,k}v_{n,k}\zeta_{n}^{(k)}\alpha_{n}^{(k)}\oindicator{\Omega_{n}}]\right|^{2}\right]^{p/2}\\
	&\quad \quad \quad \quad +\E\left[\sum_{k=1}^{mn}\left|(\E_{k}-\E_{k-1})[z^{-2}\bar{u}_{n,k}v_{n,k}\zeta_{n}^{(k)}\alpha_{n}^{(k)}\oindicator{\Omega_{n}}]\right|^{p}\right]\\
	&\ll_{p} \E\left[\sum_{k=1}^{mn}|\bar{u}_{n,k}|^{2}|v_{n,k}|^{2}\E_{k-1}\left|\zeta_{n}^{(k)}\oindicator{\Omega_{n}}\right|^{2}\right]^{p/2}+\sum_{k=1}^{mn}|\bar{u}_{n,k}|^{p}|v_{n,k}|^{p}\E\left|\zeta_{n}^{(k)}\oindicator{\Omega_{n}}\right|^{p}\\
	&\ll_{L,p} \left[\sum_{k=1}^{mn}|\bar{u}_{n,k}|^{2}|v_{n,k}|^{2}n^{-1}\right]^{p/2}+\sum_{k=1}^{mn}|\bar{u}_{n,k}|^{p}|v_{n,k}|^{p}n^{-p/2}\\
	&\ll_{L,p} n^{-p/2}, 
	\end{align*}
	where we again used Jensen's inequality, the bound $|z| \geq 5$, and the fact that $u_n$ and $v_n$ are unit vectors.  This completes the proof of \eqref{eq:concshow}, and hence the proof of Lemma \ref{Lem:concentration} is complete.  
\end{proof}

\section{Proof of Theorem \ref{Thm:reductions}} \label{sec:combin}
In this section we complete the proof of Theorem \ref{Thm:reductions}.  We continue to work under the assumptions and notation introduced in Section \ref{Sec:Concentration}.  

It remains to prove part \ref{item:tr:isotropic} of Theorem \ref{Thm:reductions}.  In view of Lemma \ref{Lem:concentration} and \eqref{eq:owpOmega_n}, it suffices to show that
\begin{equation} \label{eq:tr:isotropic:show}
	\sup_{5 \leq |z| \leq 6} \left| \E [u_n^\ast \mathcal{G}_n(z) v_n \oindicator{\Omega_n}] + \frac{1}{z} u_n^\ast v_n \right| = o(1). 
\end{equation}
\subsection{Neumann series}
 We will rewrite the resolvent, $\mathcal{G}_{n}$, as a Neumann series.  Indeed, for $|z|\geq 5$,  
\begin{equation} \label{eq:neumannnormbnd}
\frac{1}{\sqrt{n}}\lnorm \frac{\blmat{Y}{n}}{z}\rnorm\leq \frac{9}{10} < 1
\end{equation}
on the event $\Omega_n$, so we may write
\begin{equation*}
 \mathcal{G}_n(z) \oindicator{\Omega_{n}}
=-\frac{1}{z}\left(I\oindicator{\Omega_{n}} + \sum_{k=1}^{\infty}\left(\frac{1}{\sqrt{n}}\frac{\blmat{Y}{n}\oindicator{\Omega_{n}}}{z}\right)^{k}\right)
=-\frac{1}{z}I\oindicator{\Omega_{n}}-\sum_{k=1}^{\infty}\frac{\left(\frac{\blmat{Y}{n}}{\sqrt{n}}\oindicator{\Omega_{n}}\right)^{k}}{z^{k+1}}
\end{equation*}
almost surely. Therefore, using \eqref{eq:owpOmega_n}, we obtain
\begin{align}
\E\left[u_n^{*}\mathcal{G}_{n}v_n\oindicator{\Omega_{n}}\right]&= \E\left[-\frac{1}{z}u_n^{*}v_n\oindicator{\Omega_{n}}-u_n^{*}\sum_{k=1}^{\infty}\frac{\left(\frac{\blmat{Y}{n}}{\sqrt{n}}\right)^{k}}{z^{k+1}}v_n\oindicator{\Omega_{n}}\right]\notag\\
&=-\frac{1}{z}u_n^{*}v_n\P(\Omega_{n})-\sum_{k=1}^{\infty}\frac{1}{z^{k+1}}\E\left[u_n^{*}\left(\frac{\blmat{Y}{n}}{\sqrt{n}}\right)^{k}v_n\oindicator{\Omega_{n}}\right]\notag\\
&= -\frac{1}{z}u_n^{*}v_n + o(1) -\sum_{k=1}^{\infty}\frac{1}{z^{k+1}}\E\left[u_n^{*}\left(\frac{\blmat{Y}{n}}{\sqrt{n}}\right)^{k}v_n\oindicator{\Omega_{n}}\right].\label{Equ:Neumann}
\end{align}
 Showing the sum in \eqref{Equ:Neumann} converges to zero uniformly in the region $\{z \in \C : 5 \leq |z| \leq 6 \}$ will complete the proof of \eqref{eq:tr:isotropic:show}.

\subsection{Removing the indicator function}

In this subsection, we prove the following.  
\begin{lemma} \label{lem:removeindicator}
	Under the assumptions above, for any integer $k \geq 1$,  
	\begin{equation*}
	\left|\E\left[u_n^{*}\left(\frac{1}{\sqrt{n}}\blmat{Y}{n}\right)^{k}v_n\right]-\E\left[u_n^{*}\left(\frac{1}{\sqrt{n}}\blmat{Y}{n}\right)^{k}v_n\oindicator{\Omega_{n}}\right]\right| = o_{k,L}(1).
	\end{equation*}
\end{lemma}

\begin{proof}
Since the entries of $\mathcal{Y}_n$ are truncated, it follows that 
\[ \lnorm \mathcal{Y}_n \rnorm \leq \lnorm \mathcal{Y}_n \rnorm_2 \ll_{L} n \]
almost surely.  Therefore, as $\Prob(\Omega_n^c) = O_A(n^{-A})$ for any $A > 0$, we obtain
\begin{align*}
	\left|\E\left[u_n^{*}\left(\frac{1}{\sqrt{n}}\blmat{Y}{n}\right)^{k}v_n\right]-\E\left[u_n^{*}\left(\frac{1}{\sqrt{n}}\blmat{Y}{n}\right)^{k}v_n\oindicator{\Omega_{n}}\right]\right| &\leq \E \left[ \lnorm \mathcal{Y}_n \rnorm^k \oindicator{\Omega_n^c} \right] \\
	&\ll_{L} n^{k}  \Prob(\Omega_n^c) \\
	&\ll_{L,A} n^{k-A}.
\end{align*}
Choosing $A$ sufficiently large (in terms of $k$), completes the proof.
\end{proof}

\subsection{Combinatorial arguments}

In this section, we will show that 
\begin{equation}
\sup_{5 \leq |z| \leq 6} \left| \sum_{k=1}^{\infty}\frac{\E\left[u_n^{*}\left(\frac{1}{\sqrt{n}}\blmat{Y}{n}\right)^{k}v_n\oindicator{\Omega_{n}}\right]}{z^{k+1}} \right| =o_{L}(1).
\label{Equ:sumtozero}
\end{equation}
In view of \eqref{eq:neumannnormbnd}, the tail of the series can easily be controlled.  Thus, it suffices to show that, for each integer $k \geq 1$, 
\begin{equation*}
\E\left[u_n^{*}\left(\frac{1}{\sqrt{n}}\blmat{Y}{n}\right)^{k}v_n\oindicator{\Omega_{n}}\right] = o_{L,k}(1). 
\end{equation*}
By Lemma \ref{lem:removeindicator}, it suffices to prove the statement above without the indicator function. In particular, the following lemma completes the proof of \eqref{eq:tr:isotropic:show} (and hence completes the proof of Theorem \ref{Thm:reductions}).  

\begin{lemma}[Moment Calculations] \label{Lem:momentstozero}
	Under the assumptions above, for any integer $k \geq 1$, 
	\begin{equation*}
	\E\left[u_{n}^{*}\left(\frac{1}{\sqrt{n}}\blmat{Y}{n}\right)^{k}v_{n}\right] = o_{L,k}(1).
	\end{equation*}
\end{lemma}

To prove Lemma \ref{Lem:momentstozero}, we will expand the above expression in terms of the entries of the random matrices $\iidmat{X}{n}{1}$,\ldots,$\iidmat{X}{n}{m}$. For brevity, in this section we will drop the subscript $n$ from our notation and just write $X_{1}, \ldots, X_m$ for $X_{n,1}, \ldots, X_{n,m}$. Similarly, we write the vectors $u_{n}$ and $v_{n}$ as $u$ and $v$, respectively. 

To begin, we exploit the block structure of $\blmat{Y}{n}$ and write
\begin{align}
\E\left[u^{*}\left(\frac{1}{\sqrt{n}}\blmat{Y}{n}\right)^{k}v\right]&=n^{-k/2}\E\left[\sum_{1\leq a,b\leq\notag m}(u^{*})^{[a]}\left(\blmat{Y}{n}^{k}\right)^{[a,b]}v^{[b]}\right]\\
&=n^{-k/2}\sum_{1\leq a,b\leq m}(u^{*})^{[a]}\E\left[\left(\blmat{Y}{n}^{k}\right)^{[a,b]}\right]v^{[b]}.\label{Equ:BlockSum}
\end{align}
Due to the block structure of $\blmat{Y}{n}$, for each $1\leq a\leq m$, there exists some $1\leq b\leq m$ which depends on $a$ and $k$ such that the $[a,b]$ block of $\blmat{Y}{n}^{k}$ is nonzero and all other blocks, $[a,c]$ for $c\neq b$, are zero. Additionally, the nonzero block entry $(\blmat{Y}{n}^{k})^{[a,b]}$ is $X_{a}X_{a+1}\cdots X_{a+k}$ where the subscripts are reduced modulo $m$ and we use modular class representatives $\{1,\ldots,m\}$ (as opposed to the usual $\{0,..,m-1\}$). To show that the expectation of this sum is $o_{L,k}(1)$, we need to systematically count all terms which have a nonzero expectation. To do so, we use graphs to characterize the terms which appear in the sum. In particular we develop path graphs, each of which corresponds uniquely to a term in the expansion of \eqref{Equ:BlockSum}. For each term, the corresponding path graph will record the matrix entries which appear, the order in which they appear, and the matrix from which they come. We begin with the following definitions. 

\begin{definition}
	We consider graphs where each vertex is specified by the ordered pair $(t,i_{t})$ for integers $1\leq t\leq k+1$ and $1\leq i_{t}\leq n$. We call $t$ the \textit{time} coordinate of the vertex and $i_{t}$ the \textit{height} coordinate of the vertex. Let $(i_{1},i_{2},\ldots,i_{k+1})$ be a $k+1$ tuple of integers from the set $\{1,2,\ldots,n\}$ and let $1\leq a\leq m$. We define an \textit{$m$-colored $k$-path graph} $G^{a}(i_{1},i_{2},\ldots,i_{k+1})$ to be the edge-colored directed graph with vertex set $V=\{(1,i_{1}),\;(2,i_{2}),\ldots,(k+1,i_{k+1})\}$ and directed edges from $(t,i_{t})$ to $(t+1,i_{t+1})$ for $1\leq t\leq k$, where the edge from $(t,i_{t})$ to $(t+1,i_{t+1})$ is color $a+t-1$, with the convention that colors are reduced modulo $m$, and the modulo class representatives are $\{1,\ldots,m\}$. The graph is said to \textit{visit} a vertex $(t,i_{t})$ if there exists an edge which begins or terminates at that vertex, so that $(t,i_{t})\in V$. We say an edge of $G^{a}(i_{1},\ldots,i_{k+1})$ is of \textit{type I} if it terminates on a vertex $(t,i_{t})$ such that $i_{t}\neq i_{s}$ for all $s<t$. We say an edge is of \textit{type II} if it terminates on a vertex $(t,i_{t})$ such that there exists some $s<t$ with $i_{t} = i_{s}$.
\end{definition}
	
	We make some observations about this definition.
	\begin{remark}
		We view $k$ and $n$ as specified parameters.  Once these parameters are specified, the notation $G^{a}(i_{1},\ldots,i_{k+1})$ completely determines the graph. The vertex set $V$ is a subset of the vertices of the $(k+1)\times n$ integer lattice and each graph has exactly $k$ directed edges. In each graph, there is an edge which begins on vertex $(1,i_{1})$ for some $1\leq i_{1}\leq n$ and there is an edge which terminates on vertex $(k+1,i_{k+1})$ for some $1\leq i_{k+1}\leq n$. Each edge begins at $(t,i_{t})$ and terminates on $(t+1,i_{t+1})$ for some integers $1\leq t\leq k$ and $1\leq i_{t},i_{t+1}\leq n$. Additionally, each edge is one of $m$ possible colors. 
		Note that if we think about the edges as ordered by the time coordinate, then the order of the colors is a cyclic permutation of the coloring $1,2,\ldots,m$, beginning with $a$. This cycle is repeated as many times as necessary in order to cover all edges. 
		
		Notice that we call $G^{a}(i_{1},i_{2},\ldots,i_{k+1})$ a path graph because it can be thought of as a path through the integer lattice from vertex $(1,i_{1})$ to $(k+1,i_{k+1})$ for some $1\leq i_{1},i_{k+1}\leq n$. Indeed, by the requirements in the definition, this graph must be one continuous path and no vertex can be visited more than once. We may call an $m$-colored $k$-path graph a \emph{path graph} when $m$ and $k$ are clear from context.
		
		Finally, we can think of an edge as type I if it terminates at a height not previously visited. It is of type II if it terminates at a height that has been previously visited. 
	\end{remark}

\begin{definition}
	For a given $m$-colored $k$-path graph $G^{a}(i_{1},i_{2},\ldots,i_{k+1})$, we say two edges $e_{1}$ and $e_{2}$ of $G^{a}(i_{1},i_{2},\ldots,i_{k+1})$ are \textit{time-translate parallel} if $e_{1}$ begins at vertex $(t,i_{t})$ and terminates at vertex $(t+1,i_{t+1})$ and edge $e_{2}$ begins at vertex $(t',i_{t'})$ and terminates at vertex $(t'+1,i_{t'+1})$ where $i_{t}=i_{t'}$ and $i_{t+1}=i_{t'+1}$ for some $1\leq t,t'\leq k$ with $t\neq t'$. 
\end{definition}

\begin{remark}
	Intuitively, two edges are time-translate parallel if they span the same two hight coordinates at different times. Throughout this section, we shorten the term ``time-translate parallel" and refer to edges with this property as ``parallel" for brevity. We warn the reader that by this, we mean that the edges must span the same heights at two different times. For instance, the edge from $(1,2)$ to $(2,4)$ is not parallel to the edge from $(3,1)$ to $(4,3)$ since they don't span the same heights, although these edges might appear parallel in the geometric interpretation of the word. Also note that two parallel edges need not have the same color. See Figure \ref{Fig:graph8_2M1} for examples of parallel and non parallel edges.
\end{remark}

\begin{definition}
	For a fixed $k$, we say two $m$-colored $k$-path graphs $G^{a}(i_{1},\ldots,i_{k+1})$ and $G^{a}(i'_{1},\ldots,i'_{k+1})$ are \textit{equivalent}, denoted $G^{a}(i_{1},\ldots,i_{k+1})\sim G^{a}(i'_{1},\ldots,i'_{k+1})$, if there exists some permutation $\sigma$ of $\{1,\ldots,n\}$ such that $(i_{1},\ldots,i_{k+1})=(\sigma(i'_{1}),\ldots,\sigma(i'_{k+1}))$. Note here that for two path graphs to be equivalent, the color of the first edge, and hence the color of all edges sequentially, must be the same in both.
\end{definition}

One can check that the above definition of equivalent $m$-colored $k$-path graphs is an equivalence relation. Thus, the set of all $m$-colored $k$-path graphs can be split into equivalence classes. 

\begin{definition}
	For each equivalence class of graphs, the \textit{canonical $m$-colored $k$-path graph} is the unique graph from an equivalence class which satisfies the following condition: If $G^{a}(i_{1},\ldots,i_{k+1})$ visits vertex $(t,i_{t})$, then for every $0 < i<i_{t}$ there exists $s<t$ such that $G^{a}(i_{1},\ldots,i_{k+1})$ visits $(s,i)$.
\end{definition} 

\begin{remark}
	Observe that, intuitively, the canonical representation for $G^{a}(i_{1},\ldots,i_{k+1})$ is the graph which does not ``skip over" any height coordinates. Namely, the canonical graph necessarily begins at vertex $(1,1)$ and at each time step the height of the next vertex can be a most one larger than the maximum height of all previous vertices.
\end{remark}

\begin{figure}[h]
	\includegraphics[scale=0.4]{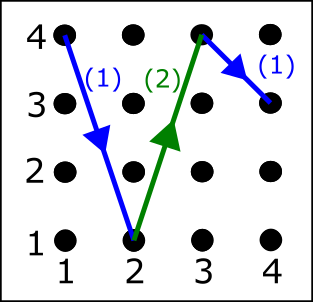}
	\includegraphics[scale=0.4]{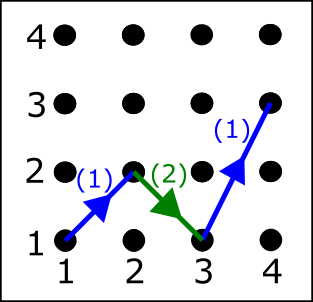}
	\caption{The graphs featured here are two possible 2-colored 3-path graphs, $G^{1}(4,1,4,3)$ and $G^{1}(1,2,1,3)$, which correspond to the terms $X_{1,(4,1)}X_{2,(1,4)}X_{1,(4,3)}$ and $X_{1,(1,2)}X_{2,(2,1)}X_{1,(1,3)}$, respectively. These two graphs are equivalent. The first graph is not a canonical graph while the second graph is a canonical graph.  For ease of notation and clarity, we have drawn all vertices on the integer lattice as dots, with the time axis appearing horizontally and the height axis vertically.  While each dot represents a vertex in the integer latter, not all of the dots represent vertices in the path graphs.  Only dots from which an edge begins or terminates are vertices of the path graph.  We have also colored the edges with blue and green to represent the two colors.  The colors of each edge are also represented by the number in parenthesis, e.g., $(1)$ or $(2)$.  In either graph, the first and third edges are type I edges, while the second edge in each graph is type II.  See Example \ref{Example:TwoEquiv} for further discussion.}
	\label{Fig:3path1and2}
\end{figure}

Now, we return to the task at hand: the proof of Lemma \ref{Lem:momentstozero}. We fix a positive integer $k$, and we expand as in \eqref{Equ:BlockSum}. 
For a fixed value $1 \leq a \leq m$, consider the nonzero block $\left(\blmat{Y}{n}^{k} \right)^{[a, b]}$. We further expand to see 
\begin{equation} \label{Equ:ExpandedTerm}
(u^{*})^{[a]}\E\left[\left(\blmat{Y}{n}^{k}\right)^{[a,b]}\right]v^{[b]} = \sum_{i_{1},\ldots,i_{k+1}}\overline{u}^{[a]}_{i_{1}}\E[X_{a,(i_{1},i_{2})}X_{a+1,(i_{2},i_{3})}\cdots X_{a+k,(i_{k},i_{k+1})}]v^{[b]}_{i_{k+1}},
\end{equation}
where the subscripts $a,\ldots,a+k$ are reduced mod $m$ with representatives $\{1,\ldots,m\}$. Observe that by the structure of $\blmat{Y}{n}$, these subscripts must appear cyclically in the order $a,a+1,\ldots,m,1,\ldots,a-1$, with the order repeating as many times as necessary before ending at $b$.  In particular, the subscripts are uniquely determined by the starting subscript $a$ and the value of $k$.  

We now consider the expectation on the right-hand side of \eqref{Equ:ExpandedTerm}.  Since all entries of each matrix are independent, if an index appears only once in a product, that product will have expectation zero. Therefore, only terms in which every index appears more than once will have a nonzero expectation, and only such terms will contribute to the expected value of the sum. Note that for an entry to appear more than once, we not only need the index of the entries to match but also which of the $m$ matrices the entries came from.   The following definition will assist in encoding each entry on the right-hand side of \eqref{Equ:ExpandedTerm} as a unique $m$-colored $k$-path graph.    

\begin{definition}
	We say the term $\iident{X}{a}{i_{1}}{i_{2}}\iident{X}{a+1}{i_{2}}{i_{3}}\cdots\iident{X}{a+k}{i_{k}}{i_{k+1}}$ from the expansion of \eqref{Equ:ExpandedTerm} \textit{corresponds} to the $m$-colored $k$-path graph $G^{a}(i_{1},\ldots,i_{k+1})$. We use the notation 
\[ x_{G} := X_{a,(i_{1},i_{2})}X_{a+1,(i_{2},i_{3})}\cdots X_{a+k,(i_{k},i_{k+1})} \]
whenever $X_{a,(i_{1},i_{2})}X_{a+1,(i_{2},i_{3})}\cdots X_{a+k,(i_{k},i_{k+1})}$ corresponds to the path graph $G$.  In this case, we also write $u^{*}_{G}$ and $v_{G}$ for $\bar{u}_{i_{1}}^{[a]}$ and $v_{i_{k+1}}^{[a+k]}$, respectively. 
\end{definition}

Each term in the expansion of \eqref{Equ:ExpandedTerm} corresponds uniquely to an $m$-colored $k$-path graph. In terms of the corresponding $m$-colored $k$-path graph, if an edge spans two vertices $(t,i_{t})$ and $(t+1,i_{t+1})$, and has color $a$ then the corresponding matrix product must contain the entry $X_{a,(i_{t},i_{t+1})}$. Thus, the color corresponds to the matrix from which the entry came and the height coordinates correspond to the matrix indices. Repeating indices is analogous to parallel edges, and entries coming from the same matrix corresponds to edges sharing a color. For example, if $X_{4,(3,5)}$ appears at some point in a term, the corresponding $m$-colored $k$-path graph will have an edge from $(t,3)$ to $(t+1,5)$ for some $t$, and the edge will be colored with color $4$.  Thus, a graph corresponds to a term with nonzero expectation if for every edge $e_{1}$, there exists at least one other edge in the graph which is parallel to $e_{1}$ and which has the same color as $e_{1}$. We must systematically count the terms which have nonzero expectation.

Since two equivalent graphs correspond to two terms which differ only by a permutation of indices, and since entries in a given matrix are independent and identically distributed, the expectation of the corresponding terms will be equal. This leads to the following lemma.

\begin{lemma} \label{Lem:EquivGraphsEqualInExpectation}
	If two path graphs $G^{a}(i_{1},\ldots,i_{k+1})$ and $G^{a}(i'_{1},\ldots,i'_{k+1})$ are equivalent, then 
	\begin{equation*}
	\E[x_{G^{a}(i_{1},\ldots,i_{k+1})}] = \E[x_{G^{a}(i'_{1},\ldots,i'_{k+1})}].
	\end{equation*}
\end{lemma}
This lemma allows us to characterize graphs with non-zero expectation based on their canonical representation. Before we begin counting the graphs which correspond to terms with nonzero expectation, we present some examples.   

\begin{example}
	\label{Example:TwoEquiv}
	Consider two $2$-colored $3$-path graphs: $G^{1}(4,1,4,3)$ and $G^{1}(1,2,1,3)$. $G^{1}(4,1,4,3)$ is the leftmost graph in Figure \ref{Fig:3path1and2} and $G^{1}(1,2,1,3)$ is the rightmost graph in Figure \ref{Fig:3path1and2}. They correspond to the terms $X_{1,(4,1)}X_{2,(1,4)}X_{1,(4,3)}$ and $X_{1,(1,2)}X_{2,(2,1)}X_{1,(1,3)}$ respectively from the expansion of $u^{*}\left(\frac{1}{\sqrt{n}}\blmat{Y}{n}\right)^{3}v$ where $m=2$. Note that these graphs are equivalent by the permutation which maps $4\mapsto1\mapsto2\mapsto4$. $G^{1}(4,1,4,3)$ is not a canonical graph, while $G^{1}(1,2,1,3)$ is a canonical graph. Observe that since $X_{1,(4,1)}$ appeared first in the product $X_{1,(4,1)}X_{2,(1,4)}X_{1,(4,3)}$, the first edge in the corresponding path graph is an edge of color $1$ spanning from height coordinate $4$ to height coordinate $1$. 
\end{example}

\begin{example}
	Consider a product of the form
	\begin{equation*} \iident{X}{1}{i_{1}}{i_{2}} \iident{X}{1}{i_{2}}{i_{1}}\iident{X}{1}{i_{1}}{i_{2}}\iident{X}{1}{i_{2}}{i_{1}}\iident{X}{1}{i_{1}}{i_{2}}\iident{X}{1}{i_{2}}{i_{1}},\iident{X}{1}{i_{1}}{i_{2}}
	\end{equation*}
	where $i_1, i_2$ are distinct.  This corresponds to a $1$-colored $7$-path graph whose canonical representative can be seen in Figure \ref{Fig:graph8M1}. Since all entries come from matrix $X_{1}$, we know $m=1$. Since there are 7 terms in the product, $k=7$.  This term has expected value 
	$\E\left[ (\iident{X}{1}{i_{1}}{i_{2}})^{4}\right]\E\left[ (\iident{X}{1}{i_{2}}{i_{1}})^{3}\right]$.  Since $X_1$ is an iid matrix, the particular choice of $i_1 \neq i_2$ is irrelevant to the expected value. 
	\label{Ex:8M1}
\end{example}

\begin{figure}[h]
	\includegraphics[scale= 0.4]{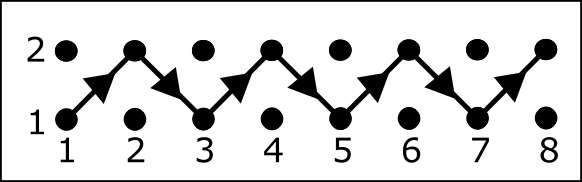}
	\caption{This $1$-colored $7$-path graph corresponds to the product in Example \ref{Ex:8M1}. Since all edges are the same color, they are not labeled with distinct colors. Observe that every edge in this graph is parallel to at least one other edge.  }
	\label{Fig:graph8M1}
\end{figure}

\begin{example}
	Consider a product of the form
	\begin{equation*}
	\iident{X}{1}{i_{1}}{i_{2}}\iident{X}{1}{i_{2}}{i_{3}}\iident{X}{1}{i_{3}}{i_{1}}\iident{X}{1}{i_{1}}{i_{3}}\iident{X}{1}{i_{3}}{i_{4}}\iident{X}{1}{i_{4}}{i_{3}}\iident{X}{1}{i_{3}}{i_{4}}
	\end{equation*}
	where $i_{1},i_{2},i_{3}$ and $i_{4}$ are distinct. This corresponds to a $1$-colored $7$-path graph whose canonical representative is featured in Figure \ref{Fig:graph8_2M1}. Since the only index pair which is appears more than once is $(i_{3},i_{4})$, the corresponding term will have zero expectation. 
	\label{Ex:8_2M1}
\end{example}

\begin{figure}[h]
	\includegraphics[scale= 0.4]{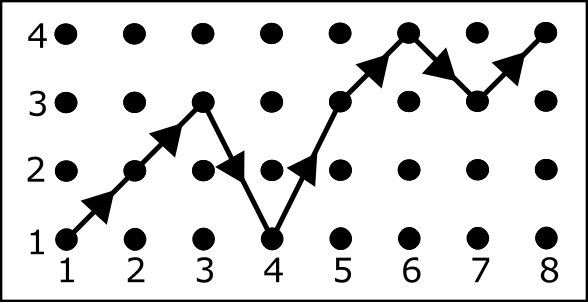}
	\caption{This $1$-colored $7$-path graph is the canonical representative to the graph which corresponds to the product in Example \ref{Ex:8_2M1}. Note that the edge from $(5,3)$ to $(6,4)$ is parallel to the edge from $(7,3)$ to $(8,4)$, but no other edges in the graph are parallel. This implies that the corresponding term will have zero expectation.}
	\label{Fig:graph8_2M1}
\end{figure}

\begin{example} \label{Ex:8M4}
	Let $i_{1},i_{2},i_{3},i_{4}$ be distinct, and consider the product 
	\begin{equation*}
	 \iident{X}{1}{i_{1}}{i_{2}} \iident{X}{2}{i_{2}}{i_{1}}\iident{X}{3}{i_{1}}{i_{2}}\iident{X}{4}{i_{2}}{i_{1}}\iident{X}{1}{i_{1}}{i_{2}}\iident{X}{2}{i_{2}}{i_{1}}\iident{X}{3}{i_{1}}{i_{2}}\iident{X}{4}{i_{2}}{i_{1}}.
	\end{equation*}
	Since there are 8 entries in this product, $k=8$, and as there are entries from 4 matrices, $m=4$. The corresponding canonical $4$-colored $8$-path graph representative is shown in Figure \ref{Fig:8M4}. This term has expectation 
	\[ \E\left[(\iident{X}{1}{i_{1}}{i_{2}})^{2}( \iident{X}{2}{i_{2}}{i_{1}})^{2}(\iident{X}{3}{i_{1}}{i_{2}})^{2}(\iident{X}{4}{i_{2}}{i_{1}})^{2}\right]. \]
\end{example}
\begin{figure}[h]
	\includegraphics[scale= 0.4]{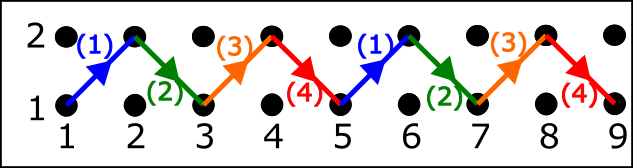}
	\caption{This $4$-colored $8$-path graph corresponds to the product in Example \ref{Ex:8M4}. Since there are eight edges, this is one possible product from $u^\ast\left(n^{-1/2}\blmat{Y}{n}\right)^{8}v $. }
	\label{Fig:8M4}
\end{figure}

\begin{example}
	Consider the product 
	\begin{equation*} 
	\iident{X}{3}{i_{1}}{i_{2}} \iident{X}{1}{i_{2}}{i_{3}}\iident{X}{2}{i_{3}}{i_{4}}\iident{X}{3}{i_{4}}{i_{1}}\iident{X}{1}{i_{1}}{i_{2}}\iident{X}{2}{i_{2}}{i_{4}}\iident{X}{3}{i_{4}}{i_{2}}\iident{X}{1}{i_{2}}{i_{1}}\iident{X}{2}{i_{1}}{i_{5}}.
	\end{equation*}
	Since there are 9 entries in this product, $k=9$ and we can see that there are 3 different matrices so that $m=3$. For any distinct indices $i_{1},i_{2},i_{3},i_{4},i_{5}$, the corresponding canonical $3$-colored $9$-path graph representative is shown in Figure \ref{Fig:9M3}.  In this product, every entry appears only once. Thus the expectation of this product factors, and the term will have expectation zero.
	\label{Ex:9M3}
\end{example}

\begin{figure}[h]
	\includegraphics[scale= 0.4]{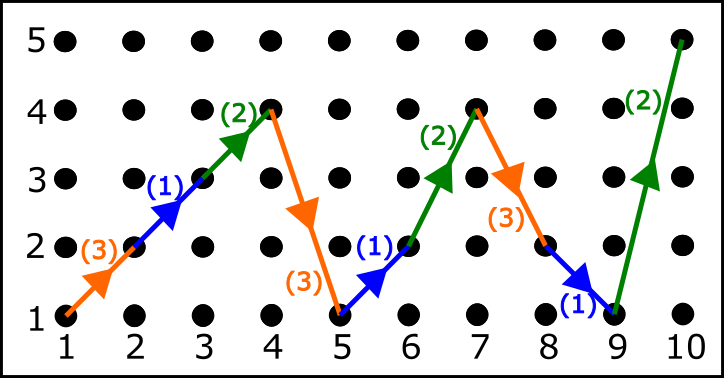}.
	\caption{This $3$-colored $9$-path graph is the canonical representative of the graph corresponding to the term in Example \ref{Ex:9M3}. There are nine edges that appear in this graph and $3$ colors on those edges, indicating that $k=9$ and $m=3$. Since no edges are parallel to another edge of the same color, this path graph corresponds to a term which has zero expectation.}
	\label{Fig:9M3}
\end{figure}

We now complete the proof of Lemma \ref{Lem:momentstozero}, which will occupy the remainder of the section.  
	Let 
	\begin{equation*}
	\Delta_{n,k}^{a} := \{G^{a}(i_{1},\ldots,i_{k+1})\;:\;1\leq i_{1},\ldots,i_{k+1}\leq n\}, 
	\end{equation*}
	and let $\tilde{\Delta}_{n,k}^{a}$ be the set of all canonical graphs in $\Delta_{n,k}^{a}$. We now divide the proof into cases based on the value of $k$.  
	\begin{description}
		\item[Case where $k$ is a multiple of $m$] If $k$ is a multiple of $m$, then by the block structure of $\blmat{Y}{n}$, it follows that $\blmat{Y}{n}^{k}$ is a block diagonal matrix. Since the diagonal blocks are the only nonzero blocks in this case, \eqref{Equ:BlockSum} simplifies to
		\begin{align}
		\E\left[u^{*}\left(\frac{1}{\sqrt{n}}\blmat{Y}{n}\right)^{k}v\right]&=n^{-k/2}\sum_{1\leq a\leq m}(u^{*})^{[a]}\E\left[\left(\blmat{Y}{n}^{k}\right)^{[a,a]}\right]v^{[a]}\notag\\
		&=n^{-k/2}\sum_{1\leq a\leq m}\sum_{G\in\Delta_{n,k}^{a}}u^{*}_{G}\E[x_{G}]v_{G}\label{Equ:FirstSumExpansion}
		\end{align}
		Recall that if $G\in\Delta_{n,k}^{a}$, then $G$ is an $m$-colored $k$-path graph which starts with color, $a$, i.e., $G=G^{a}(i_{1},\ldots,i_{k+1})$ for some $i_{1},\ldots,i_{k+1}\in [n]$. By Lemma \ref{Lem:EquivGraphsEqualInExpectation}, we can reduce the task of counting all terms with nonzero expectation to counting canonical graphs and the cardinality of each equivalence class. 
		
		Observe that if $k=m$, then any term in \eqref{Equ:FirstSumExpansion} will be of the form
		$$\overline{u}_{i_{1}}\E[X_{a,(i_{1},i_{2})}X_{a+1,(i_{2},i_{3})}\cdots X_{a-1,(i_{k},i_{k+1})}]v_{i_{k+1}}$$
		where each matrix contributes only one entry to the above expression. In this case, all terms are independent and the expectation in \eqref{Equ:FirstSumExpansion} is zero. 
		
		Now consider the case where $k=cm$ for some integer $c\geq 2$.  Define 
		\begin{equation}
		h(G)=:\max\{i_{t}\;:\;(t,i_{t})\in V_{\tilde{G}}\},
		\label{Def:h(G)}
		\end{equation}
		where $\tilde{G}$ is the canonical representative for the graph $G$ and $V_{\tilde{G}}$ is the vertex set for the graph $\tilde{G}$. We call $h(G)$ the \textit{maximal height} (or sometimes just \emph{height}) of a graph $G$. Intuitively, $h(G)$ is the number of distinct height coordinates $G$ visits. In terms of the canonical graph, $\tilde{G}$, this is the largest height coordinate visited by an edge in $\tilde{G}$. For each $a$, define
		\begin{align}
		& (\Delta_{n,k}^{a})_{1}:=\{G\in \Delta_{n,k}^{a}\;:\; h(G)>k/2\}\label{Def:Delta1}\\
		& (\Delta_{n,k}^{a})_{2}:=\{G\in \Delta_{n,k}^{a}\;:\;h(G)=k/2\}\label{Def:Delta2}\\
		& (\Delta_{n,k}^{a})_{3}:=\{G\in \Delta_{n,k}^{a}\;:\;h(G)<k/2\}\label{Def:Delta3}.
		\end{align}
		This partitions $\Delta_{n,k}^{a}$ into disjoint subsets. Without loss of generality, we assume that $a=1$ since the argument will be the same for any permutation of the coloring. In this case, all path graphs start with color $1$ and since $k=cm$, the colors $1,2,\ldots,m$ will each repeat $c$ times. We analyze each set of graphs separately.\\
		
		\noindent \underline{Graphs in $(\Delta_{n,k}^{1})_{1}$:}
		
		First, consider the set $(\Delta_{n,k}^{1})_{1}$ and recall that each $G\in(\Delta_{n,k}^{1})_{1}$ must have exactly $k$ edges. Since the expectation of all equivalent graphs is the same, it is sufficient to assume that $G$ is canonical. If $h(G)>k/2$, there must be more than $k/2$ type I edges. If each edge were parallel to at least one other edge in $G$, then there would be more than $k$ edges, a contradiction. Hence there will be at least one edge that is not parallel to any other edge. This implies $\E[x_{G}]=0$ whenever $G\in (\Delta_{n,k}^{1})_{1}$ and thus
		\begin{equation}
		\sum_{G\in(\Delta_{n,k}^{a})_{1}}u^{*}_{G}\E[x_{G}]v_{G}=0.
			\label{Equ:MMultKG1}
		\end{equation}
		
		\noindent \underline{Graphs in $(\Delta_{n,k}^{1})_{2}$:}
		
		Note that if $k$ is odd, then this set will be empty; so assume $k$ is even. Now consider a graph $G\in(\Delta_{n,k}^{1})_{2}$. By Lemma \ref{Lem:EquivGraphsEqualInExpectation}, we can assume that $G$ is canonical. If $G$ has any edges which are not parallel to any other edges, then $\E[x_{G}]=0$ and it does not contribute to the expectation. Thus we can consider only graphs in which every edge is parallel to at least one other edge. Since any $G\in (\Delta_{n,k}^{1})_{2}$ must visit exactly $k/2$ distinct height coordinates and since there must be precisely $k$ edges in $G$, a counting argument reveals that every edge in $G$ must be parallel to exactly one other edge in $G$. This gives way to the following lemma. 
		\begin{lemma}
			Let $k \geq 2$ be any even integer (not necessarily a multiple of $m$). Then there is only one canonical $k$-path graph in $\Delta^{1}_{n,k}$ for which $h(G)=\frac{k}{2}$ and in which each edge is parallel to exactly one other edge.
			\label{Lem:OnlyOne}
		\end{lemma}
		The proof of this lemma, which relies on a counting argument and induction, is detailed in Appendix \ref{Sec:OnlyOne}.  In fact, the proof reveals that this one canonical $m$-colored $k$-path graph starting with color $1$  is 
		$$G^{1}(1,2,\ldots,k/2,1,2,\ldots,k/2,1).$$
		If two edges are parallel but are not the same color then the expectation of terms with corresponding canonical graph will be zero. 
				
		If $c$ is odd and $m$ is even, then the edge from $(k/2,k/2)$ to $(k/2+1,1)$ will have color $\frac{m}{2}$ and thus edge from $(k/2+1,1)$ to $(k/2+2,2)$ will have color $\frac{m}{2}+1$. This edge is necessarily parallel to the edge from $(1,1)$ to $(2,2)$, and it is the only edge parallel to the edge from $(1,1)$ to $(2,2)$. But note that the edge from $(1,1)$ to $(2,2)$ had color 1 and $\frac{m}{2}+1$ is not congruent to 1 $\mod$ $m$. Therefore in the case where $c$ is odd and $m$ is even, the canonical $m$-colored $k$-path graph corresponds to a term in the product which has expectation zero.
		
		Finally, if $c$ is even, then the edge from $(k/2,k/2)$ to $(k/2+1,1)$ must have color $m$. Hence the edge from $(k/2+1,1)$ to $(k/2+2,2)$ will have color 1, which is the same color as the edge from $(1,1)$ to $(2,2)$. This means that when $k=cm$ and $c$ is even, every edge in $G^{1}(1,2,\ldots,k/2,1,2,\ldots, k/2,1)$ will be parallel to exactly one other edge of the same color. In particular, note that for this graph 
		\begin{align}
		\left| u^{*}_{G}\E[x_{G}]v_{G} \right| &\leq \left|\overline{u}_{1}^{[1]} \right| \E \left| X_{1,(1,2)}\cdots X_{m,(k/2,1)} \right| \left| v_{1}^{[m]} \right| \notag\\ 
		&\leq \left|\overline{u}_{1}^{[1]} \right|\E \left|X_{1,(1,2)} \right|^{2}\cdots \E\left|X_{m,(k/2,1)}\right|^{2} \left|v_{1}^{[m]} \right| \notag\\
		&\leq \left|\overline{u}_{1}^{[1]}\right| \left|v_{1}^{[m]}\right|. \label{Equ:SimpleInnerProd}
		\end{align} 
		
		For ease of notation, let $\tilde{G}:=G^{1}(1,2,\ldots,k/2,1,2,\ldots,k/2,1)$, and consider $G^{1}(i_{1},\ldots,i_{k+1})\in (\Delta_{n,k}^{1})_{2}$ such that $G^{1}(i_{1},\ldots,i_{k+1})\sim \tilde{G}$.  Observe that there are $n$ options for the first coordinate $i_{1}$ of $G^{1}(i_{1},\ldots,i_{k+1})$. If we fix the first coordinate, then there are at most $(n-1)(n-2)\cdots(n-k/2-1)\leq n^{k/2-1}$ graphs with first coordinate $i_{1}$ which are equivalent to $\tilde{G}$. If we repeated the computation of the expectation of any of these equivalent graphs, we would get a term similar to \eqref{Equ:SimpleInnerProd} but with different starting and ending coordinates, yielding an upper bound of $\left| \overline{u}_{i_{1}} \right| \left| v_{i_{1}} \right|$. Therefore, by the above argument and the Cauchy--Schwarz inequality, we obtain
		\begin{align}
		\left|\sum_{G\in(\Delta_{n,k}^{1})_{2}}u^{*}_{G}\E[x_{G}]v_{G}\right| & \leq \sum_{1\leq i_{1}\leq n}n^{k/2-1}\left|\overline{u}_{i_{1}}\right|\left|v_{i_{1}}\right|\notag\\
		&\leq n^{k/2-1}\lnorm u\rnorm \lnorm v\rnorm \notag\\
		& \leq n^{k/2-1}.\label{Equ:MMultKG2}
		\end{align}


		\noindent \underline{Graphs in $(\Delta_{n,k}^{1})_{3}$:}
		
		Consider an $m$-colored $k$-path graphs $G\in (\Delta_{n,k}^{1})_{3}$, and assume that $G$ is canonical. If $G$ contains any edges which were not parallel to another edge, then the graph will correspond to a term with expectation zero. So consider a canonical graph $G\in (\Delta_{n,k}^{1})_{3}$ such that all edges are parallel to at least one other edge. If $h(G)=1$, then $G=G^{1}(1,1,\ldots,1)$ and so
		\begin{equation*}
		\left| \E[x_{G}] \right| \leq \E \left| \iident{X}{1}{1}{1}\cdots\iident{X}{1}{1}{1}\right| =\E \left| \iident{X}{1}{1}{1} \right|^{k} \leq (4L)^{k}.
		\end{equation*}
		Note that this is the highest possible moment in a term. Let $M:=(4L)^{k}$.  For any canonical $m$-colored $k$-path graph $G\in (\tilde{\Delta}_{n,k}^{1})_{3}$,
		\begin{equation*}
		\E |x_{G} | \leq M. 
		\end{equation*}
		Also note that this bound holds for graphs of all starting colors, not just starting color 1. In addition, for any $G$ with maximal height $h(G)$, there are $n(n-1)\cdots (n-h(G)-1)<n^{h(G)}$ graphs in the equivalence class of $G$. By over counting, we can bound the number of distinct equivalence classes by $k^{k}$ since there are $k$ edges and at each time coordinate, the edge which starts at that time coordinate can terminate at most one height coordinate larger than it started, so any edge has a most $k$ options for an ending coordinate.


		Based on the above observations, we have 
		\begin{align}
		\left|\sum_{G\in  (\Delta_{n,k}^{1})_{3}}u^{*}_{G}\E[x_{G}]v_{G}\right|&\leq \sum_{G\in  (\tilde{\Delta}_{n,k}^{1})_{3}}\left|n^{h(G)}u^{*}_{G}\E[x_{G}]v_{G}\right|\notag\\
		&\leq n^{\frac{k}{2}-\frac{1}{2}}\sum_{G\in  (\tilde{\Delta}_{n,k}^{1})_{3}}\left|u^{*}_{G}\right| \E\left|x_{G} \right| \left| v_{G}\right|\notag\\
		&\leq n^{\frac{k}{2}-\frac{1}{2}}\sum_{G\in  (\tilde{\Delta}_{n,k}^{1})_{3}}M\notag\\
		& \ll_{L,k}n^{\frac{k}{2}-\frac{1}{2}}. \label{Equ:MMultKG3}
		\end{align}

		\noindent \underline{Combining the bounds:}
		
		By \eqref{Equ:MMultKG1}, \eqref{Equ:MMultKG2}, and \eqref{Equ:MMultKG3} we conclude that 
		\begin{align*}
		\left|\sum_{G\in\Delta_{n,k}^{1}}u^{*}_{G}\E[x_{G}]v_{G}\right| & \leq \sum_{i=1}^{3}\left|\sum_{G\in(\Delta_{n,k}^{1})_{i}}u^{*}_{G}\E[x_{G}]v_{G}\right|\\
		&\ll_{L,k} n^{k/2-1/2}.
		\end{align*}
		While the bounds above were calculated for $a=1$, the same argument applies for any $a$ by simply permuting the colors.  Therefore, in the case where $k$ is a multiple of $m$, from \eqref{Equ:FirstSumExpansion} we have
		\begin{align*}
		\left|\E\left[u^{*}\left(\frac{1}{\sqrt{n}}\blmat{Y}{n}\right)^{k}v\right]\right|&\leq n^{-k/2}\sum_{1\leq a\leq m}\left|\sum_{G\in\Delta_{n,k}^{a}}u^{*}_{G}\E[x_{G}]v_{G}\right|\\ 
		&\ll_{L,k}n^{-k/2}\sum_{1\leq a\leq m} n^{k/2-1/2}\\
		&\ll_{L,k}\frac{1}{\sqrt{n}}.
		\end{align*}

		\item[Case where $k$ is not a multiple of $m$] Now assume that $k$ is not a multiple of $m$. If $k<m$, then each matrix has at most one entry in the product on the right-hand side of \eqref{Equ:ExpandedTerm} and all terms will be independent. Hence the expectation will be zero. Therefore, consider the case when $k>m$. Then there must exist some positive integer $c$ such that 
		$$cm<k<(c+1)m.$$
		We can write $k=cm+r$ for some $0<r<c$ and in this case, a computation reveals that the only nonzero blocks in $\blmat{Y}{n}^{k}$ are blocks of the form $[a,a+r]$ where $a$ and $a+r$ are reduced modulo $m$, and the modulo class representatives are $\{1,2,\ldots,m\}$. In this case we can write
		\begin{align}
		\E\left[u^{*}\left(\frac{1}{\sqrt{n}}\blmat{Y}{n}\right)^{k}v\right]&=n^{-k/2}\sum_{1\leq a\leq m}(u^{*})^{[a]}\E\left[\left(\blmat{Y}{n}^{k}\right)^{[a,a+r]}\right]v^{[a+r]}\notag\\
		&=n^{-k/2}\sum_{1\leq a\leq m}\sum_{G\in\Delta_{n,k}^{a}}u^{*}_{G}\E[x_{G}]v_{G}\label{Equ:SecondSumExpansion}
		\end{align}
		
		Again, define $h(G)$ as in \eqref{Def:h(G)} and define $(\Delta_{n,k}^{a})_{1}$, $(\Delta_{n,k}^{a})_{2}$, and $(\Delta_{n,k}^{a})_{3}$ as in \eqref{Def:Delta1}, \eqref{Def:Delta2}, and \eqref{Def:Delta3}, respectively. Without loss of generality, assume that $a=1$.
		
		If a graph $G$ has height greater than $k/2$, by the same argument in the previous case we can see that there must be an edge which is not parallel to any other edge. Therefore when $k$ is not a multiple of $m$ we still have
		\begin{equation}
		\sum_{G\in(\Delta_{n,k}^{1})_{1}}u^{*}_{G}\E[x_{G}]v_{G}=0.
		\label{Equ:MNotMultKG1}
		\end{equation}
		If $G\in(\Delta_{n,k}^{1})_{3}$ has height less than $k/2$, then we may still bound $\E\left|x_{G}\right|\leq M$. Therefore, we may use the same argument as in the previous case to conclude that 
		\begin{equation}
		\left|\sum_{G\in(\Delta_{n,k}^{1})_{3}}u^{*}_{G}\E[x_{G}]v_{G}\right|\ll_{L,k}n^{k/2-1}. 
		\label{Equ:MNotMultKG3}
		\end{equation}
		Thus, we need only to consider graphs in $(\Delta_{n,k}^{a})_{2}$.\\

		\noindent \underline{Graphs in $(\Delta_{n,k}^{1})_{2}$:}
		
		If $k$ is odd, then $(\Delta_{n,k}^{1})_{2}$ is empty, so assume that $k$ is even. Consider a graph $G\in(\Delta_{n,k}^{1})_{2}$ and by Lemma \ref{Lem:EquivGraphsEqualInExpectation}, we may assume that $G$ is canonical. If $G$ has any edges which are not parallel to any other edge, then $\E[x_{G}]=0$, so assume each edge is parallel to at least one other edge. A counting argument reveals that in fact each edge must be parallel to exactly one other edge and by Lemma \ref{Lem:OnlyOne}, we can conclude that in fact $G=G^{1}(1,2,\ldots,k/2,1,2,\ldots,k/2,1)$.
		
		In order for this graph to correspond to a term with nonzero expectation, the colors on the pairs of parallel edges must match. In order for this to happen, we would need the edge from $(k/2+1,1)$ to $(k/2+2,2)$ to have color 1. This would force the edge from $(k/2,k/2)$ to $(k/2+1,1)$ to have color $m$. Note that if we think about drawing edges sequentially with the time coordinate, then this implies that the $k/2$th edge drawn from $(k/2,k/2)$ to $(k/2+1,1)$ is of color $m$, forcing $k/2$ to be a multiple of $m$. However, this would imply that $k$ is also a multiple of $m$, a contradiction. Hence in this case, if $G\in(\tilde{\Delta}_{n,k}^{a})_{2}$, then $\E[x_{G}]=0$. By Lemma \ref{Lem:EquivGraphsEqualInExpectation}, this gives
		\begin{equation}
		\sum_{G\in(\Delta_{n,k}^{a})_{2}}u^{*}_{G}\E[x_{G}]v_{G}=0.
		\label{Equ:MNotMultKG2}
		\end{equation}
		
		\noindent \underline{Combining the bounds:}
		
		By \eqref{Equ:MNotMultKG1}, \eqref{Equ:MNotMultKG3}, and \eqref{Equ:MNotMultKG2} we can see that 
		\begin{align*}
		\left|\sum_{G\in\Delta_{n,k}^{1}}u^{*}_{G}\E[x_{G}]v_{G}\right| & \leq \sum_{i=1}^{3}\left| \sum_{G\in(\Delta_{n,k}^{1})_{i}}u^{*}_{G}\E[x_{G}]v_{G}\right|\\
		&\ll_{L,k} n^{k/2-1}.  
		\end{align*}
		While the bounds above were calculated for $a=1$, the same arguments apply for any $a$ by simply permuting the colors.  Thus, from \eqref{Equ:SecondSumExpansion} we have
		\begin{align*}
		\left|\E\left[u^{*}\left(\frac{1}{\sqrt{n}}\blmat{Y}{n}\right)^{k}v\right]\right|&\leq n^{-k/2}\sum_{1\leq a\leq m}\left|\sum_{G\in\Delta_{n,k}^{a}}u^{*}_{G}\E[x_{G}]v_{G}\right|\\ 
		&\ll_{L,k}n^{-k/2}\sum_{1\leq a\leq m} n^{k/2-1}\\
		&\ll_{L,k}\frac{1}{n}
		\end{align*}
		in the case where $k$ is not a multiple of $m$.

	\end{description}
	Combining the cases above completes the proof of Lemma \ref{Lem:momentstozero}.  

\begin{remark}
	Note that if $m=1$, then $k$ is trivially a multiple of $m$. Hence, the case where $\blmat{Y}{n}$ is an $n\times n$ matrix follows as a special case of the above argument.
\end{remark}

\section{Proofs of results from Section \ref{Sec:RelatedResults}}
\label{Sec:RelatedResultProofs}

Before we prove the results from Section \ref{Sec:RelatedResults}, we must prove an isotropic limit law for products of repeated matrices. 

\begin{theorem}[Isotropic limit law for repeated products] \label{Thm:RepeatedProdIsotropic}
	Assume $\xi$ is a complex-valued random variable with mean zero, unit variance, finite fourth moment, and independent real and imaginary parts.  For each $n \geq 1$, let $X_{n}$ be an $n \times n$ iid random matrix with atom variable $\xi$. Define $\blmat{Y}{n}$ as in \eqref{def:Y} and define $\mathcal{G}_{n}(z)$ as in \eqref{def:G} but with $\iidmat{X}{n}{1}=\iidmat{X}{n}{2}=\cdots=\iidmat{X}{n}{m}=X_{n}$. Then, for any fixed $\delta > 0$, the following statements hold.
	\begin{enumerate}[label=(\roman*)]
		\item \label{item:pw:invertible} Almost surely, for $n$ sufficiently large, the eigenvalues of $\frac{1}{\sqrt{n}} \mathcal{Y}_n$ are contained in the disk $\{z \in \mathbb{C} : |z| \leq 1 + \delta \}$.  In particular, this implies that almost surely, for $n$ sufficiently large, the matrix $\frac{1}{\sqrt{n}} \mathcal{Y}_n - z I$ is invertible for every $z \in \mathbb{C}$ with $|z| > 1 + \delta$.
		\item \label{item:pw:invtbnd} There exists a constant $c>0$ (depending only on $\delta$ and $m$) such that almost surely, for $n$ sufficiently large,
		\begin{equation*}
		\sup_{z\in\C:|z|>1+\delta}\lnorm\mathcal{G}_{n}(z)\rnorm \leq c.
		\end{equation*}
		\item \label{item:pw:isortopic} For each $n \geq 1$, let $u_n, v_n \in \mathbb{C}^{mn}$ be deterministic unit vectors.  Then
		$$ \sup_{z \in \mathbb{C} : |z| > 1 + \delta} \left| u_n^\ast \mathcal{G}_n(z) v_n + \frac{1}{z} u_n^\ast v_n \right| \longrightarrow 0 $$
		almost surely as $n \to \infty$. 
	\end{enumerate}
	\label{Thm:isotropic power}
\end{theorem}

\begin{proof}
	Fix $\delta>0$. From \cite[Theorem 1.4]{Tout}, the spectral radius of $\frac{1}{\sqrt{n}}X_{n}$ converges to $1$ almost surely as $n \to \infty$. Thus, $n^{-m/2}(X_{n})^{m}$ has spectral radius converging to $1$ almost surely as well.  It follows that the spectral radius of $n^{-m/2}(\blmat{Y}{n})^{m}$ converges to $1$ almost surely, which in turn implies that the spectral radius of $\frac{1}{\sqrt{n}}\blmat{Y}{n}$ converges to $1$ almost surely as $n \to \infty$, proving claim \ref{item:pw:invertible}.


   	To prove part \ref{item:pw:invtbnd}, we consider two events, both of which hold almost surely. By \cite[Theorem 1.4]{Tout}, there exists a constant $K> 0$ such that almost surely, for $n$ sufficiently large, $n^{-1/2} \|X_n \| \leq K$, and hence, on the same event, $n^{-1/2} \lnorm\blmat{Y}{n}\rnorm \leq K$.  By Lemma \ref{Lem:specnorm} this implies that almost surely, for $n$ sufficiently large,
   	\begin{equation}
   	\sup_{z\in\C:|z|\geq K+1}\lnorm \left(\frac{1}{\sqrt{n}}\blmat{Y}{n}-zI\right)^{-1}\rnorm\leq 1.
   	\end{equation}
   	To deal with $1+\delta\leq |z|\leq K+1$, we observe that 
	\begin{equation} \label{eq:inverseblockbnd}
	\left(\left(\frac{1}{\sqrt{n}}\blmat{Y}{n}-zI\right)^{-1}\right)^{[a,b]} = z^{(m-1)-\alpha}n^{-\alpha/2}X_{n}^{\alpha}\left(n^{-m/2}X_{n}^{m}-z^{m}I\right)^{-1} 
	\end{equation}
	by a block inverse computation, where $\alpha = (b-a) \;(\mathrm{mod}\  m)$. Thus, we have
	\begin{align*} 
	&\lnorm\left(\left(\frac{1}{\sqrt{n}}\blmat{Y}{n}-zI\right)^{-1}\right)^{[a,b]}\rnorm\\ 
	& \quad \leq | z |^{(m-1)-\alpha} \lnorm n^{-\alpha/2}X_{n}^{\alpha}\rnorm\lnorm\left(n^{-m/2}X_{n}^{m}-z^{m}I\right)^{-1}\rnorm\\
	& \quad \leq | z |^{(m-1)-\alpha} \lnorm n^{-\alpha/2}X_{n}^{\alpha}\rnorm \prod_{k=1}^{m}\lnorm\left(n^{-1/2}X_{n}-ze^{2\pi k\sqrt{-1}/m}I\right)^{-1}\rnorm.
	\end{align*}  
	We now bound 
	\begin{equation*}
	\sup_{z\in\C:1+\delta<|z|<K+1}| z|^{(m-1)-\alpha}\lnorm n^{-\alpha/2}X_{n}^{\alpha}\rnorm\prod_{k=1}^{m}\lnorm\left(n^{-1/2}X_{n}-ze^{2\pi k\sqrt{-1}/m}I\right)^{-1}\rnorm.
	\end{equation*}
	Note that almost surely, for $n$ sufficiently large $\lnorm n^{-\alpha/2}X_{n}^{\alpha}\rnorm \leq K^\alpha \leq K^{m-1}$.  Hence, we obtain
	\begin{align*}
	&\sup_{z\in\C:1+\delta<|z|<K+1}| z|^{(m-1)-\alpha}\lnorm n^{-\alpha/2}X_{n}^{\alpha}\rnorm\prod_{k=1}^{m}\lnorm\left(n^{-1/2}X_{n}-ze^{2\pi k\sqrt{-1}/m}I\right)^{-1}\rnorm \\
	&\qquad\qquad \leq (K+1)^{m-1} K^{m-1} \sup_{z\in\C:1+\delta<|z|<K+1}\prod_{k=1}^{m}\lnorm\left(n^{-1/2}X_{n}-ze^{2\pi k\sqrt{-1}/m}I\right)^{-1}\rnorm 
	\end{align*}
	almost surely, for $n$ sufficiently large.  The bound for 
	\[ \sup_{z\in\C:1+\delta<|z|<K+1} \prod_{k=1}^{m}\lnorm\left(n^{-1/2}X_{n}-ze^{2\pi k\sqrt{-1}/m}I\right)^{-1}\rnorm  \]
	follows from Lemma \ref{Lem:LeastSingValAwayFromZero} (taking $m = 1$).  Returning to \eqref{eq:inverseblockbnd}, we conclude that almost surely, for $n$ sufficiently large
	\[ \sup_{z \in \mathbb{C} : 1 + \delta \leq |z| \leq K + 1} \left \| \left(\left(\frac{1}{\sqrt{n}}\blmat{Y}{n}-zI\right)^{-1}\right)^{[a,b]} \right\| \leq c \]
	for some constant $c > 0$ (depending only on $\delta$ and $m$).  Since $1 \leq a,b \leq m$ are arbitrary, the proof of property \ref{item:pw:invtbnd} is complete.

	For \ref{item:pw:isortopic}, \cite[Theorem 1.4]{Tout} yields that almost surely, for $n$ sufficiently large, 
	\begin{equation} \label{eq:910bnd}
		\sup_{|z| \geq 5} \frac{1}{\sqrt{n}}\lnorm\frac{\blmat{Y}{n}}{z}\rnorm\leq \frac{9}{10}<1. 
	\end{equation}
	Thus, we expand the resolvent as a Neumann series to obtain
	$$\mathcal{G}_{n}(z)=-\frac{1}{z}\left(I+\sum_{k=1}^{\infty}\left(\frac{1}{\sqrt{n}}\frac{\blmat{Y}{n}}{z}\right)^{k}\right)=-\frac{1}{z}I-\sum_{k=1}^{\infty}\frac{\left(\frac{\blmat{Y}{n}}{\sqrt{n}}\right)^{k}}{z^{k+1}}. $$
	Thus, we have almost surely, for $n$ sufficiently large, 
	$$u^{*}\mathcal{G}_{n}(z)v=-\frac{1}{z}u^{*}v-\sum_{k=1}^{\infty}\frac{u^{*}\left(\frac{\blmat{Y}{n}}{\sqrt{n}}\right)^{k}v}{z^{k+1}}.$$
	We will show that the series on the right-hand side converges to zero almost surely uniformly in the region $\{z \in \C : 5 \leq |z| \leq 6\}$.  Indeed, from \eqref{eq:910bnd}, the tail of the series is easily controlled.  Thus, it suffices to show that, for each fixed integer $k \geq 1$, 
	\[ \left| u_n^\ast \left( \frac{1}{\sqrt{n}} \mathcal{Y}_n \right)^k v_n \right| = o_k(1). \]
	But this follows from the block structure of $\mathcal{Y}_n$ and \cite[Lemma 2.3]{Tout}.  
	
	We now extend this convergence to the region $\{z \in \C : |z| \geq 1 + \delta\}$.  Let $\eps > 0$.  Let $M \geq 6$ be a constant to be chosen later.  By Vitali's convergence theorem (see, for instance \cite[Lemma 2.14]{BSbook}), it follows that 
	\[ \sup_{1 + \delta \leq |z| \leq M}  \left| u_n^\ast \mathcal{G}_n(z) v_n + \frac{1}{z} u_n^\ast v_n \right| \longrightarrow 0 \]
	almost surely.  In particular, almost surely, for $n$ sufficiently large, 
	\begin{equation} \label{eq:altiscond1}
	\sup_{1 + \delta \leq |z| \leq M}  \left| u_n^\ast \mathcal{G}_n(z) v_n + \frac{1}{z} u_n^\ast v_n \right| \leq \eps.
	\end{equation}
	
	Choose $M_1 > 0$ such that, for all $|z| \geq M_1$, 
	\begin{equation*}
	\lnorm \left(-\frac{1}{z}\right)u^{*}v\rnorm \leq \left|\frac{1}{z}\right|\lnorm u^{*}\rnorm\lnorm v\rnorm \leq \frac{\varepsilon}{2}.
	\end{equation*}
	Also there exists a constant $M_{2} > 0$ such that  
	\begin{equation*}
	\sup_{|z| \geq M_2} \lnorm u^{*}\mathcal{G}_{n}(z)v\rnorm \leq \frac{\varepsilon}{2}
	\end{equation*}
	almost surely, for $n$ sufficiently large, by Lemma \ref{Lem:specnorm} and \cite[Theorem 1.4]{Tout}. Take $M:=\max\{M_{1},M_{2}, 6\}$. Then almost surely, for $n$ sufficiently large, 
	\begin{equation} \label{equ:MtoInfProds}
	\sup_{|z| \geq M} \left|u^{*}\mathcal{G}_{n}(z)v+\frac{1}{z}u^{*}v\right|\leq\varepsilon. 
	\end{equation}
	
	Combining \eqref{eq:altiscond1} and \eqref{equ:MtoInfProds}, we obtain almost surely, for $n$ sufficiently large, 
	\[ \sup_{|z| \geq 1 + \delta}  \left|u^{*}\mathcal{G}_{n}(z)v+\frac{1}{z}u^{*}v\right|\leq\varepsilon. \]
	Since $\eps > 0$ was arbitrary, the proof is complete.  
\end{proof}

We also need the following lemma in order to prove Theorem \ref{thm:nomixedinpower}.

\begin{lemma}
	Let $\xi$ be a complex-valued random variable with mean zero, unit variance, finite fourth moment, and independent real and imaginary parts.  For each $n \geq 1$, let $X_{n}$ be an $n \times n$ iid random matrix with atom variable $\xi$. Let $m$ be a positive integer. Then, for any fixed $\delta > 0$, the following statements hold.  
	\begin{enumerate}[label=(\roman*)]
		\item Almost surely, for $n$ sufficiently large, the eigenvalues of $n^{-m/2}{X}_{n}^{m}$ are contained in the disk $\{z \in \mathbb{C} : |z| \leq 1 + \delta \}$.  In particular, this implies that almost surely, for $n$ sufficiently large, the matrix $n^{-m/2}{X}_{n}^{m}- z I$ is invertible for every $z \in \mathbb{C}$ with $|z| > 1 + \delta$.
		\item There exists a constant $c > 0$ such that almost surely, for $n$ sufficiently large, 
		$$ \sup_{z \in \mathbb{C} : |z| > 1 + \delta} \lnorm \left(n^{-m/2}{X}_{n}^{m}-zI\right)^{-1} \rnorm \leq c. $$
		\item For each $n \geq 1$, let $u_n, v_n \in \mathbb{C}^{n}$ be deterministic unit vectors.  Then 
		$$ \sup_{z \in \mathbb{C} : |z| > 1 + \delta} \left| u_n^\ast \left(n^{-m/2}{X}_{n}^{m}-zI\right)^{-1} v_n + \frac{1}{z} u_n^\ast v_n \right| \longrightarrow 0 $$
		almost surely as $n \to \infty$. 
	\end{enumerate}
	\label{lem:PowerIsotropic}
\end{lemma}

The proof of this lemma is similar to the proof of Corollary \ref{cor:ProductIsotropic}; we omit the details.

With these results, we may proceed to the proofs of the results in Section \ref{Sec:RelatedResults}. The proofs of Theorems \ref{Thm:NoOutlierInPowerPert},  \ref{thm:nomixedinpower}, and \ref{Thm:RepeatedProdOutliers} follow the proofs of Theorems \ref{Thm:NoOutlierInProductPert}, \ref{thm:nomixed}, and \ref{thm:outliers}, respectively, verbatim, except for the following changes:
\begin{itemize}
	\item Take $X_{n,1}=\cdots = X_{n,m} = X_{n}$,
	\item Replace all occurrences of Theorem \ref{thm:isotropic} by Theorem \ref{Thm:RepeatedProdIsotropic}
	\item Replace all occurrences of Corollary \ref{cor:ProductIsotropic} by Lemma \ref{lem:PowerIsotropic}. 
	\item The scaling factor of $\sigma$ needs to be replaced by $\sigma^{m}$.
\end{itemize}

\appendix
\section{Proof of Lemma \ref{lem:Truncate}}
\label{Sec:ProofOfTruncation}
In this section, we present the proof of Lemma \ref{lem:Truncate}.
\begin{proof}[Proof of Lemma \ref{lem:Truncate}]
	Take $L_{0} := \sqrt{8 \E|\xi|^{4}}$. We begin by proving \ref{item:truncation:ii}. Observe that
	\begin{align*}
	1 &= \E|\xi|^{2}\\
	&=\E|\Re(\xi)|^{2}+\E|\Im(\xi)|^{2}\\
	&=\E\left[|\Re(\xi)|^{2}\indicator{|\Re(\xi)|\leq L/\sqrt{2}}\right]+\E\left[|\Re(\xi)|^{2}\indicator{|\Re(\xi)|> L/\sqrt{2}}\right]\\
	&\quad \quad +\E\left[|\Im(\xi)|^{2}\indicator{|\Im(\xi)|\leq L/\sqrt{2}}\right]+\E\left[|\Im(\xi)|^{2}\indicator{|\Im(\xi)|> L/\sqrt{2}}\right]\\
	&=\Var(\Re(\tilde{\xi}))+\Var(\Im(\tilde{\xi}))\\ 
	&\quad \quad+\left|\E\left[\Re(\xi)\indicator{|\Re(\xi)|\leq L/\sqrt{2}}\right]\right|^{2}+\E\left[|\Re(\xi)|^{2}\indicator{|\Re(\xi)|> L/\sqrt{2}}\right]\\
	&\quad \quad +\left|\E\left[\Im(\xi)\indicator{|\Im(\xi)|\leq L/\sqrt{2}}\right]\right|^{2}+\E\left[|\Im(\xi)|^{2}\indicator{|\Im(\xi)|> L/\sqrt{2}}\right], 
	\end{align*}
	which implies
	\begin{align*}
	1-\Var(\tilde{\xi})=&\left|\E\left[\Re(\xi)\indicator{|\Re(\xi)|\leq L/\sqrt{2}}\right]\right|^{2}+\E\left[|\Re(\xi)|^{2}\indicator{|\Re(\xi)|> L/\sqrt{2}}\right]\\
	& +\left|\E\left[\Im(\xi)\indicator{|\Im(\xi)|\leq L/\sqrt{2}}\right]\right|^{2}+\E\left[|\Im(\xi)|^{2}\indicator{|\Im(\xi)|> L/\sqrt{2}}\right].  
	\end{align*}
	Thus, using the fact that $\Re(\xi)$ and $\Im(\xi)$ both have mean zero (so, for example, $\E\left[\Re(\xi)\indicator{|\Re(\xi)|\leq L/\sqrt{2}}\right] = -\E\left[\Re(\xi)\indicator{|\Re(\xi)|> L/\sqrt{2}}\right]$) and then applying Jensen's inequality, we obtain 
	\begin{align*}
	|1-\Var(\tilde{\xi})| &\leq 2\E[|\Re(\xi)|^{2}\indicator{|\Re(\xi)|> L/\sqrt{2}}] +2\E[|\Im(\xi)|^{2}\indicator{|\Im(\xi)|> L/\sqrt{2}}]\\
	&\leq 2 \E [| \xi |^2  \indicator{|\xi| > L /\sqrt{2}}] \\
	&\leq \frac{4}{L^2} \E |\xi|^4. 
	\end{align*}
	This concludes the proof of \ref{item:truncation:ii}.  
	
	Property \ref{item:truncation:i} follows easily from \ref{item:truncation:ii} by the choice of $L_{0}$.  Next we move onto the proof of \ref{item:truncation:iii}. One can see that since $\Var(\tilde{\xi})\geq \frac{1}{2}$, 
	\begin{align*}
	\left|\hat{\xi}\right| 
	&\leq \frac{\left|\Re(\xi)\right|\indicator{|\Re(\xi)|\leq L/\sqrt{2}}+\E\left[\left|\Re(\xi)\right|\indicator{|\Re(\xi)|\leq L/\sqrt{2}}\right]}{\sqrt{\Var(\tilde{\xi})}}\\
	&\quad \quad \quad +\frac{\left|\Im(\xi)\right|\indicator{|\Im(\xi)|\leq L/\sqrt{2}}+\E\left[\left|\Im(\xi)\right|\indicator{|\Im(\xi)|\leq L/\sqrt{2}}\right]}{\sqrt{\Var(\tilde{\xi})}}\\
	&\leq 4L
	\end{align*}
	almost surely.  
	
	For \ref{item:truncation:iv}, we observe that $\hat{\xi}$ has mean zero and unit variance by construction. Additionally, since the real and imaginary parts of $\hat{\xi}$ depend only on the real and imaginary parts of $\xi$ respectively, they are independent by construction.  For the fourth moment, we use that $\Var(\tilde{\xi})\geq \frac{1}{2}$ and Jensen's inequality inequality to obtain
	\begin{align*}
	\E|\hat{\xi}|^{4} &\ll  \frac{1}{\Var(\tilde{\xi})^{2}}\left(\E\left[\left|\Re(\xi)\right|^{4}\indicator{|\Re(\xi)|\leq L/\sqrt{2}}\right]+\E\left[\left|\Im(\xi)\right|^{4}\indicator{|\Im(\xi)|\leq L/\sqrt{2}}\right]\right)\\
	& \ll \E |\xi|^4, 
	\end{align*}
	as desired.  
\end{proof}

\section{Proof of Theorem \ref{Thm:LeastTruncSingValNonZero}} \label{sec:singoutlier}

This section is devoted to the proof of Theorem \ref{Thm:LeastTruncSingValNonZero}.  We begin with Lemma \ref{lemma:singoutlier} below, which is based on \cite[Theorem 4]{N}.  Throughout this section, we use $\sqrt{-1}$ for the imaginary unit and reserve $i$ as an index.  

\begin{lemma} \label{lemma:singoutlier}
Let $\mu$ be a probability measure on $[0, \infty)$, and for each $n \geq 1$, let
\[ \mu_n := \frac{1}{n} \sum_{i=1}^n \delta_{\lambda_{n,i}} \]
for some triangular array $\{ \lambda_{n,i} \}_{i \leq n}$ of nonnegative real numbers.  Let $m_n$ by the Stieltjes transform of $\mu_n$ and $m$ be the Stieltjes transform of $\mu$, i.e., 
\[ m_n(z) := \int \frac{ d \mu_n(x) }{ x - z }, \quad m(z) := \int \frac{ d \mu(x) }{x - z } \]
for all $z \in \mathbb{C}$ with $\Im(z) > 0$.  Assume
\begin{enumerate}[label=(\roman*)]
	\item $\mu_n \to \mu$ as $n \to \infty$,
	\item there exists a constant $c > 0$ such that $\mu([0,c]) = 0$,
	\item \label{item:strate} $\sup_{E \in [0,c]} \left| m_n(E + \sqrt{-1} n^{-1/2}) - m(E + \sqrt{-1} n^{-1/2}) \right| = o(n^{-1/2})$.
\end{enumerate}
Then there exists a constant $n_0 \geq 1$ such that $\mu_n([0,c/2]) = 0$ for all $n > n_0$.  
\end{lemma}
\begin{proof}
Observe that
\begin{align*}
	\Im m_n(E + \sqrt{-1} n^{-1/2}) - \Im m(E + \sqrt{-1} n^{-1/2}) = \int \frac{ n^{-1/2} d (\mu_n - \mu)(x) }{ (E - x)^2 + n^{-1} }.
\end{align*}
From assumption \ref{item:strate}, we conclude that
\[ \sup_{E \in [0,c]} \left| \int \frac{ d (\mu_n - \mu)(x) }{ (E - x)^2 + n^{-1} } \right| = o(1).  \]
We decompose this integral into two parts
\[ \int \frac{ d (\mu_n - \mu)(x) }{ (E - x)^2 + n^{-1} } = \int_0^c \frac{ d \mu_n(x) }{(E - x)^2 + n^{-1}} + \int_c^\infty \frac{ d (\mu_n - \mu)(x) }{ (E - x)^2 + n^{-1} }, \]
where we used the assumption that $\mu([0,c]) = 0$.  

Observe that 
\[ \int_c^\infty \frac{ d (\mu_n - \mu)(x) }{ (E - x)^2 + n^{-1} }  \longrightarrow 0 \]
uniformly for any $E \in [0,c/2]$ by the assumption that $\mu_n \to \mu$.  Therefore, it must be the case that
\[ \sup_{E \in [0,c/2]} \int_0^c \frac{ d \mu_n(x) }{(E - x)^2 + n^{-1}}  \longrightarrow 0. \]
Take $n_0 \geq 1$ such that 
\begin{equation} \label{eq:convcont}
	\sup_{E \in [0,c/2]} \int_0^c \frac{ d \mu_n(x) }{(E - x)^2 + n^{-1}} \leq 1/2 
\end{equation}
for all $n \geq n_0$.  

In order to reach a contradiction, assume there exists $n > n_0$ and $i \in [n]$ such that $\lambda_{n,i} \in [0,c/2]$.  Then 
\begin{align*}
	\sup_{E \in [0,c/2]} \int_0^c \frac{ d \mu_n(x) }{(E - x)^2 + n^{-1}} &= \sup_{E \in [0,c/2]} \frac{1}{n} \sum_{j=1}^n \frac{1}{(E - \lambda_{n,j})^2 + n^{-1}} \\
	&\geq \sup_{E \in [0,c/2]} \frac{1}{n} \frac{1}{(E - \lambda_{n,i})^2 + n^{-1}} \\
	&\geq 1, 
\end{align*}
a contradiction of \eqref{eq:convcont}.  We conclude that $\mu_n([0,c/2]) = 0$ for all $n > n_0$.  
\end{proof}

With Lemma \ref{lemma:singoutlier} in hand, we are now prepared to prove Theorem \ref{Thm:LeastTruncSingValNonZero}.  The proof below is based on a slight modification to the arguments from \cite{N, N2}.  As such, in some places we will omit technical computations and only provide appropriate references and necessary changes to results from \cite{N, N2}.  

Fix $\delta > 0$.  It suffices to prove that
\begin{equation} \label{eq:singvalshow1}
	\inf_{1 + \delta \leq |z| \leq 6} s_{mn} \left( \frac{1}{\sqrt{n}} \mathcal{Y}_n - z I \right) \geq c 
\end{equation}
and 
\begin{equation} \label{eq:singvalshow2}
	\inf_{|z| > 6} s_{mn} \left( \frac{1}{\sqrt{n}} \mathcal{Y}_n - zI \right) \geq c' 
\end{equation}
with overwhelming probability for some constants $c, c' > 0$ depending only on $\delta$.  

The second bound \eqref{eq:singvalshow2} follows by Lemma \ref{Lem:specnorm}.  Indeed, a bound on the spectral norm of $\mathcal{Y}_n$ (which follows from standard bounds on the spectral norms of $X_{n,k}$; see, for example, \cite[Theorem 1.4]{Tout}) gives 
\[ \| \mathcal{Y}_n \| \leq 3 \sqrt{n} \]
with overwhelming probability.  The bound in \eqref{eq:singvalshow2} then follows by applying Lemma \ref{Lem:specnorm}.  

We now turn to the bound in \eqref{eq:singvalshow1}.  To prove this bound, we will use Lemma \ref{lemma:singoutlier}.  Let $\mu_{n,z}$ be the empirical spectral measure constructed from the eigenvalues of 
\[ \left( \frac{1}{\sqrt{n}} \mathcal{Y}_n - z I \right) \left( \frac{1}{\sqrt{n}} \mathcal{Y}_n - z I \right)^\ast. \]
From \cite[Theorem 2.6]{N2},  for all $|z| \geq 1 + \delta$, there exists a probability measure $\mu_z$ supported on $[0, \infty)$ such that $\mu_{n,z} \to \mu_z$ with overwhelming probability.  Moreover, from \cite[Lemma 4.2]{Bcirc} there exists a constant $c > 0$ (depending only on $\delta$) such that $\mu_z([0,c]) = 0$ for all $|z| \geq 1 + \delta$.  Lastly, condition \ref{item:strate} in Lemma \ref{lemma:singoutlier} follows for all $1 + \delta \leq |z| \leq 6$ with overwhelming probability from \cite[Theorem 5]{N}.  Applying Lemma \ref{lemma:singoutlier}, we conclude that
\[ s_{mn} \left( \frac{1}{\sqrt{n}} \mathcal{Y}_n - z I \right) \geq c/2 \]
with overwhelming probability uniformly for all $1+ \delta \leq |z| \leq 6$.  The bound for the infimum can now be obtained by a simple net argument and Weyl's inequality \eqref{eq:weyl}.  The proof of Theorem \ref{Thm:LeastTruncSingValNonZero} is complete.

\section{Proof of Lemma \ref{Lem:OnlyOne}} \label{Sec:OnlyOne}

This section is devoted to the proof of Lemma \ref{Lem:OnlyOne}.  
	
	\begin{proof}[Proof of Lemma \ref{Lem:OnlyOne}]
		The proof proceeds inductively.  We begin with a graph only containing the vertex $(1,1)$ and then add vertices and edges sequentially with time. First, edge 1 is added; it will span from $(1,1)$ to $(2,i_{2})$.  Next, edge 2 is added and will span from $(2,i_{2})$ to $(3,i_{3})$, and so on. We use induction to prove that at each time step $t$, there is only one possible choice for $i_{t+1}$, resulting in a unique canonical graph with maximal height $k/2$ and in which each edge is parallel to exactly one other edge.  

		The edge starting at vertex $(1,1)$ can either be of type II (terminating on $(2,1)$) or of type I (terminating on $(2,2)$). By way of contradiction, assume the edge is type II. Since $G$ still has $k/2-1$ more height coordinates left to reach, it would require at least $k/2-1$ type I edges to reach hight coordinate $k/2$. Since each edge must be parallel to exactly one other edge, at some point there must be a type II edge, returning to a height coordinate previously visited. This edge will also need to be parallel to another edge. Counting all pairs of parallel edges shows that $G$ must have at least $k+1$ more edges, a contradiction. Hence the edge starting at vertex $(1,1)$ must be of type I.

		
		Assume that all edges up to time coordinate $t$, where $1\leq t < k/2-1$, are type I edges. Then $G$ must have an edge starting at vertex $(t+1,t+1)$. This edge can either of type I or type II. In order to reach a contradiction, assume that the edge is type II. Then $G$ must have at least $k/2-t-1$ more type I edges in order to reach the height $k/2$, and $G$ has exactly $k-t-1$ more edges to be added. Visiting each unvisited height coordinate would require at least $k/2-t-1$ more type I edges, and at some point after visiting new height coordinates, $G$ must return to a smaller hight coordinate, resulting in a type II edge. None of these edges could be parallel to any previous edges. Thus, overall $G$ would need to have at least $k-t+1$ more edges, a contradiction to the fact that $G$ must have exactly $k-t-1$ more edges. 
		
		We conclude that each edge of $G$ must be type I until the hight coordinate $k/2$ is reached. Namely, we have vertices $(1,1),\ldots, (k/2,k/2)$.
		
		At this point $G$ must have an edge starting at vertex $(k/2,k/2)$. Note that $G$ has $k/2-1$ edges up to this point, none of which are parallel to any other edge. $G$ must have edges parallel to the edges previously introduced and $G$ has exactly $k/2+1$ edges remaining to do so. Since there are no remaining unvisited height coordinates, the edge which starts at vertex $(k/2,k/2)$ must terminate at $(k/2+1,i_{k/2+1})$ for some $1\leq i_{k/2+1}\leq k/2$, resulting in the first type II edge. See Figure \ref{fig:proofExample} for a visual representation of the graph up to this point.
		
		We now claim that this first type II edge must in fact terminate at $(k/2+1, 1)$. By way of contradiction, suppose this edge terminates at vertex $(k/2+1,i)$ for any $1<i\leq \frac{k}{2}$. Since this is the first type II edge, it cannot be parallel to any other previously drawn edge. Up to this point $G$ has $k/2$ edges drawn and $k/2$ edges remaining to be drawn. Since all edges are by assumption to be parallel to exactly one other, a simple counting argument reveals that each edge drawn from this point on must be parallel to an existing edge. Since there is only one edge which starts at height coordinate $i$, we must now draw the edge starting at $(k/2+1,i)$ and terminating at $(k/2+2,i+1)$. By continuing this argument inductively, we must draw edges which start at $(k/2+j+1, i+j)$ and terminate at vertex $(k/2+j+2, i+j+1)$ for $0\leq j\leq k/2-i$, until the height coordinate $k/2$ is reached again. Now, since $k-i+1<k$, we must draw at least one more edge, and this edge must start at vertex $(k-i,k/2)$. In order to draw an edge parallel to an existing edge, this edge must terminate at vertex $(k-i+1, i)$. However, if we do this, we must now draw an edge parallel to an existing edge which starts a height coordinate $i$, a contradiction because the only previous edges which began at height coordinate $i$ are parallel to one another and there cannot be three edges parallel. This concludes the proof of the claim.    
		
		By the previous claim, the first type II edge must terminate at  $(k/2+1, 1)$. Again, since this is the first type II edge it cannot be parallel to any other edge. Up to this point $G$ has $k/2$ edges drawn and $k/2$ edges remaining to be drawn. Since all edges are by assumption parallel to exactly one other, a simple counting argument reveals that each edge drawn from this point on must be parallel to an existing edge. Since there is only one edge which starts at height coordinate $1$, we must now draw the edge starting at $(k/2+1,1)$ and terminating at $(k/2+2,2)$. By continuing this argument inductively, we must draw edges which start at $(k/2+j+1, j+1)$ and terminate at vertex $(k/2+j+2, j+2)$ for $0\leq j\leq k/2-1$, until the height coordinate $k/2$ is reached again. Up to this point, $k-1$ edges of $G$ have been drawn and we must draw one more edge which starts at vertex $(k,k/2)$. Since there is only one previous edge in the graph which starts at height coordinate $k/2$, the final edge must terminate at height coordinate $1$ and all edges are parallel to exactly one other edge. This results in the vertex set
		\begin{equation*} 
		V=\{(1,1),\;(2,2),\;\dots,(k/2,k/2),\;(k/2+1,1),\;(k/2+2,2),\dots (k,k/2),\;(k+1,1)\}.
		\end{equation*}
		The corresponding canonical $m$-colored $k$-path graph would be 
		\begin{equation*}
		G^{1}(1,2,\ldots,k/2,1,2,\ldots,k/2, 1), 
		\end{equation*}
		and the proof is complete.  
	\end{proof}
	
	\begin{figure}[h]
		\includegraphics[scale = 0.5]{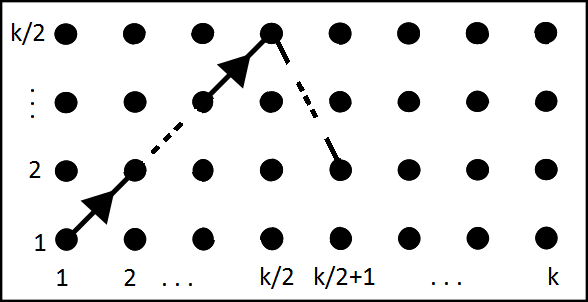}
		\caption{This is an example of a path graph that has type I edges drawn until a height of $k/2$ is reached, then the first type II edge is drawn. }
		\label{fig:proofExample}
	\end{figure}
	
\section{Useful inequalities}

\begin{lemma}[Lemma A.1 from \cite{BScov}]
	For $X = (x_{1},x_{2},\ldots,x_{N})^{T}$ iid standardized complex entries, $B$ an $N\times N$ Hermitian nonnegative definite matrix, we have, for any $p\geq 1$,
	\begin{equation*}
	\E\left|X^{*}BX\right|^{p}\leq K_{p}\left(\left(\emph{tr} B\right)^{p}+\E|x_{1}|^{2p}\emph{tr}B^{p}\right).
	\end{equation*}
	where $K_{p}>0$ depends only on $p$. 
	\label{Lem:BilinearForms}
\end{lemma}


\begin{lemma}[Follows from Lemma 10 in \cite{PW}] 
	Let $A$ and $B$ be $k\times k$ matrices with $\lnorm A\rnorm, \lnorm B\rnorm = O(1)$.  Then 
	\begin{equation*}
	\left| \det(A)-\det(B)\right| \ll_{k} \lnorm A-B\rnorm.
	\end{equation*}
	\label{Lem:normtodet}
\end{lemma}

\begin{lemma}[Spectral norm bound for large $|z|$; Lemma 3.1 from \cite{OR}]
	Let $A$ be a square matrix that satisfies $\lnorm A\rnorm \leq \mathcal{K}$. Then
	\begin{equation*}
	\lnorm \left(A-zI\right)^{-1}\rnorm \leq \frac{1}{\varepsilon}
	\end{equation*}
	for all $z\in\C$ with $|z|\geq \mathcal{K}+\varepsilon$. 
	\label{Lem:specnorm}
\end{lemma}

%
%
%
%
%

\end{document}